\documentclass[10pt, reqno]{article}
\usepackage{fullpage}
\usepackage[dvipsnames]{xcolor}  
\usepackage{tipa}
\usepackage{dsfont}
\usepackage{tikz-cd}
\usepackage{tikz}
\usetikzlibrary{decorations.markings, patterns, decorations.pathmorphing}
\usetikzlibrary{cd,fit,shapes.geometric}

\tikzset{E/.style = {ellipse, draw=blue, dashed,
                    inner xsep=0mm,inner ysep=0mm,
                    fit=#1
                   }
        }

\tikzset{->-/.style={decoration={
			markings,
			mark=at position #1 with {\arrow{>}}},postaction={decorate}}}

\usepackage{graphicx}

\usepackage{amsxtra}
\usepackage{amsmath}
\usepackage{amssymb}
\usepackage{amsfonts}
\usepackage{mathrsfs}
\usepackage{amsthm}
\usepackage[all]{xy}
\usepackage{enumitem}
\usepackage{hyperref}
\usepackage{subfigure}
\usepackage{mathabx}
\usepackage{euscript}
\usepackage{xargs}                      
\usepackage{leftindex}
\usepackage{bbm}

\usepackage[colorinlistoftodos,prependcaption,textsize=tiny]{todonotes}
\newcommandx{\note}[2][1=]{\todo[linecolor=Plum,backgroundcolor=Plum!25,bordercolor=Plum,#1]{#2}}

\theoremstyle{plain}
\newtheorem{thm}{Theorem}[section]
\newtheorem{lem}[thm]{Lemma}
\newtheorem{prop}[thm]{Proposition}
\newtheorem{cor}[thm]{Corollary}
\theoremstyle{definition}
\newtheorem{rem}[thm]{Remark}
\newtheorem{defi}[thm]{Definition}

\newtheorem{exa}[thm]{Example}
\newtheorem{notation}[thm]{Notation}

\hypersetup{%
	colorlinks,%
	linkcolor={red!50!black},%
	citecolor={blue!50!black},%
	urlcolor={blue!80!black}%
}

\numberwithin{equation}{section}

\setlist[enumerate,1]{label=(\arabic{*})}
\setlist[enumerate,2]{label=(\roman{*})}
\setlist[enumerate,3]{label=(\alph{*})}

\def\eR{ {\EuScript R}}
\def\eS{ {\EuScript S}}
\def\eT{ {\EuScript T}}

\def\A{ {\EuScript A}}

\def\B{ {\EuScript B}}

\def\can{\operatorname{can}}

\def\X{{\EuScript X}}

\def\CC{\mathbb{C}}

\def\Cat{{\mathcal Cat}}
\def\Cat{{\EuScript Cat}}

\def\cone{\operatorname{cone}}
\def\cof{\operatorname{cof}}

\def\D{{\EuScript D}}

\def\Dp{{}^{\mathfrak p}\!{D}}

\def\Rec{\on{Rec}}
\def\Ext{\on{Ext}}

\def\Exit{\on{Exit}}

\def\F{{ \EuScript F}}

\def\fib{\operatorname{fib}}

\def\Fun{\operatorname{Fun}}

\def\G{{ \EuScript G}}

\def\h{\mathfrak{h}}
\def\H{{\EuScript{H}}}

\def\Hom{\operatorname{Hom}}

\def\hra{\hookrightarrow}

\def\Im{\on{Im}}

\def\id{\on{id}}
\def\Id{\on{Id}}

\def\llim{\mathop{\underleftarrow{{\mathbbm{l}}\text{im}}}}

\def\L{{\EuScript L}}

\def\llra{\longleftrightarrow}
\def\lra{\longrightarrow}

\def\Map{\operatorname{Map}}

\def\Mod{\on{Mod}}

\def\N{\operatorname{N}}
\def\NN{\mathbb{N}}

\def\on{\operatorname}

\def\op{{\operatorname{op}}}

\def\Open{\mathbf{Open}}

 \def\P{{\EuScript P}}

\def\Perv{{\on{Perv}}}

\def\phi{{\varphi}}

\def\pt{ {\on{pt}}}

\def\p{\mathfrak{p}}

\def\Re{\on{Re}}

\def\RR{\mathbb{R}}

 \def\Set{ {\operatorname{Set}}}

\def\Sh{\on{Sh}}

 \def\sSet{{\Set}_{\Delta}}
 \def\St{{\EuScript S}t}

 \def\Sym{{\on{Sym}}}
  
 \def\sd{\operatorname{sd}}

\def\totfib{\on{totfib}}

\def\tfib{\operatorname{tfib}}
 
\def\U{ {\EuScript U}}

\def\V{ {\EuScript V}}

\def\Vect{ {\on{Vect}(k)}}

\def\var{\operatorname{var}}

\def\W{ {\EuScript W}}

\def\ZZ{\mathbb{Z}}

\def\Z{\mathcal{Z}}

 \def\={{\,\, \simeq\,\, }}
 \def\-{{\setminus}}
\def\<<{\langle {}\hskip -.1cm {}\langle}
\def\>>{\rangle \hskip -.1cm \rangle}

\newcommand{\doi}[1]{\href{https://dx.doi.org/#1}{{\small DOI:#1}}}

\newcommand{\googlebooks}[1]{(preview at \href{https://books.google.com/books?id=#1}{google books})}

\newcommand{\numdam}[1]{}
\newcommand{\mat}[1]{\begin{smallmatrix}#1\end{smallmatrix}}

\usepackage{subfiles}

\title{Lax descent and foliated Verdier duality}

\author{Tobias Dyckerhoff\footnote{Universität Hamburg, Fachbereich Mathematik, Bundesstraße 55, 20146 Hamburg, Germany
	email: {\tt first.last@uni-hamburg.de}}, 
    Till Heine\footnotemark[\value{footnote}],
Simon Schneider\footnotemark[\value{footnote}],
Kostiantyn Tolmachov\footnotemark[\value{footnote}]
    }

\begin{document}

\maketitle

\begin{abstract} 
    Higher categorical methods allow us to view recollements as lax descent
    data and build a corresponding descent theory governed by an
    $\infty$-bicategorical theory of limits. We demonstrate some concrete
    innovations brought along by these new techniques in the classical context
    of constructible and perverse sheaves on stratified spaces. 
    Our first application uses partially (op)lax descent to recover the famous result of
    Beilinson telling us how to glue perverse sheaves in terms of vanishing and
    nearby cycles. The remainder of the article is concerened with establishing
    a uniform view on a series of classification results for perverse sheaves
    culminating in recent work of Kapranov-Schechtman. Namely, we use lax
    descent to establish a ``foliated'' variant of Verdier duality which,
    informally speaking, identifies the derived constructible category of a
    suitably foliated space with sections of a categorical cosheaf on the leaf
    space whose stalks are the categories of sheaves along the leaves. Again,
    this is obtained as an application of lax descent, and seems to provide a
    reasonable framework to formalize the idea of ``mixed sheaf-cosheaf
    objects'' promoted by Kapranov and Schechtman. 
\end{abstract}

\tableofcontents

\section{Introduction}
\label{sec:intro}

The concept of a recollement was introduced in \cite{BBD} as a means to
systematically construct $t$-structures on derived categories of sheaves,
ultimately leading to the definition of perverse sheaves. When equipping the
derived categories with a suitable enhancement, say differential graded, then a
recollement can be regarded as a kind of descent datum: In analogy to classical
descent, describing the category of sheaves on a space $X$ as a limit over the
\v{C}ech nerve of an open cover, a recollement gives rise to a description of
the derived category $\D(X)$ of sheaves on $X$ as a {\em lax} limit
parameterized by an open subset and its closed complement.

Recent $\infty$-(bi)categorical tools provide us with a framework to study this
notion of {\em lax descent} with the flexibility familiar from the classical
theory of descent for open covers: For example, the work \cite{amr:stratified} is
concerned with generalizing the above-mentioned statement for recollements from
open-closed decompositions to more general stratifications, leading to descent
statements in terms of (op)lax limits over the poset parameterizing the strata.
Ongoing work of Abellan and Loubaton implements, as part of an
$(\infty,2)$-categorical concept of a topos, even more general forms of
partially lax descent that apply to covers by constructible subsets.

The point of this work is to demonstrate some applications of these new
techniques in the habitat of their origin: the theory of constructible, and
more specifically perverse sheaves, on a stratified complex analytic space $X$: 

As a first illustration we use partially (op)lax descent for constructible
covers to obtain a variant of Beilinson gluing. As a novelty, we obtain the
gluing equivalence in terms of vanishing and nearby cycles already on level of
derived categories, recovering the description of perverse sheaves from
$t$-exactness.

The main focus of the remainder of the paper is to establish a uniform
perspective on a series of classification results for perverse sheaves
culminating in a series of articles by Kapranov-Schechtman (cf.
\cite{gm:rebroussement, GMP85, KS:Riemann, KS:hyperplane, ks:bruhat}). Here,
we use lax descent for stratifications that are, in a suitable sense, foliated
by totally real manifolds leading to a ``partial'' variant of Verdier duality,
which we call {\em foliated Verdier duality}: in contrast to ordinary Verdier
duality, which is obtained by passing to sections with compact support, here we
only impose partial compactness contraints in directions transverse to the
leaves of the foliation. The equivalence brought along by this duality may be
viewed as a formalization of the ``mixed sheaf'' phenomenon promoted by
Kapranov-Schechtman: objects of the derived category can be identified with
sections of a {\em cosheaf} of categories on the leaf space with stalks given
by the various categories of {\em sheaves} along the leaves. 
The perverse $t$-exactness of foliated Verdier duality then recovers the
above-mentioned classification results for perverse sheaves on Riemann
surfaces, hyperplane arrangements, and colored configuration spaces.

Let us note right away that, in this work, we develop foliated Verdier duality
in a purely combinatorial context, in the presence of regular cell
structures. While, unfortunately, this does ruin some of the geometric appeal
of the setup, it is justified by the desire to connect to the explicit
classification results as quickly as possible. It seems like a very interesting
task to provide an intrinsically geometric account of foliated Verdier duality
-- we hope to return to it in future work.

In what follows, we outline these applications in some more detail.

\subsection{Basic notation} Let $\D$ be a stable and presentable
$\infty$-category. For a locally compact Hausdorff space $X$ we write $\Sh(X;
\D)$ and  $\Sh_c(X; \D)$ for the $\infty$-categories of $\D$-valued sheaves on
$X$ and $\infty$-category of complex analytically constructible $\D$-valued
sheaves on $X$, respectively. If $\D$ is a derived category of an abelian
category $\A$ we write $\Perv(X; \A)$ for the heart of the perverse
t-structure, where appropriate. If $P$ is a poset, and $X \to P$ is a
continuous map defining a stratification of $X$, we write $\Sh_P(X; \D),
\Perv_P(X; \A)$ for the sheaves and perverse sheaves constructible with respect
to the $P$-stratification, leaving the map $X \to P$ implicit. We refer the reader
to Appendix \ref{sec:sheaves} for the details.  

\subsection{Beilinson gluing via lax descent}%
\label{sub:beilinson_introduction}

We sketch how Beilinson's method for gluing perverse sheaves in terms of
vanishing and nearby cycles arises via a nesting of lax descent results. The
details of the construction will be discussed in \S
\ref{sec:beilinson_gluing_via_lax_descent}. Let $X$ be a complex manifold
equipped with a holomorphic map $f: X \to \CC$, let $\A$ be the category of abelian groups and $\D = \D(\A).$

In this context, we may apply a nesting of two different types of lax descent:

\begin{enumerate}[label=(D\arabic *)]
    \item\label{descent:1} First, we cover
        \[
            X = U \cup Z_+ 
        \]
        by the open subset $j: U = f^{-1}(\CC \setminus \{0\}) \hra X$ and the
        (non-complementary!) closed subset $s: Z_+ = f^{-1}(\RR_{\ge 0}) \hra X$.
        In this situation, it turns out that the derived category $\Sh(X; \D)$ forms
        the tip of a universal cone of the shape
        \[
        \begin{tikzcd}
            & \Sh(X; \D) \ar{dl} \ar[dr, ""{name=U}] & \\
            \Sh(U; \D) \ar[to=U, shorten <=30pt, shorten >=20pt, Rightarrow] \ar[rr]  \ar[swap]{dr} & & \Sh(Z_+, \D) \ar{dl}\\
              & \Sh(U \cap Z_+; \D) & 
            \arrow[from=1-2, to=3-2, crossing over]
        \end{tikzcd}
        \]
        which amounts to saying that $\Sh(X;\D)$ can, up to equivalence, be
        described as the category formed by the following {\em partially lax descent
        data}:
        \begin{itemize}
            \item an object $E \in \Sh(U; \D)$,
            \item an object $F \in \Sh(Z_+; \D)$,
            \item a morphism $\alpha: s^!j_!E \to F$, 
            \item such that $\alpha|_{U \cap Z_+}$ is an equivalence.
        \end{itemize}
        Note that there is a (simpler) lax descent statement obtained by
        covering $X$ with $U$ and its {\em complement}, leading to a
        recollement, but this will {\em not} lead to the desired outcome.
    \item\label{descent:2} Second, we cover 
        \[
            Z_+ = Z \cup V
        \]
        by the closed subset $k: Z = f^{-1}(0) \hra Z_+$ and its open complement $l: V = Z_+
        \setminus Z \hra Z_+$. This leads to a recollement which allows for two
        different types of descent: lax and oplax. Here, we use {\em oplax
        descent} to exhibit $\Sh(Z_+; \D)$ as the tip of a universal cone of the
        shape
        \[
        \begin{tikzcd}
            & \Sh(Z_+, \D) \ar{dl} \ar[dr, ""{name=U}]   & \\
            \Sh(V; \D) \ar[from=U, shorten <=20pt, shorten >=20pt, Rightarrow] \ar[rr] & & \Sh(Z; \D)
        \end{tikzcd}
        \]
        which exhibits $\Sh(X; \D)$ as equivalent to the category of {\em oplax descent
        data}:
        \begin{itemize}
            \item an object $G \in \Sh(V; \D)$,
            \item an object $H \in \Sh(Z; \D)$,
            \item a morphism $\beta: H \to k^*l_*G$. 
        \end{itemize}
\end{enumerate}

Unravelling the nesting of the two equivalences from \ref{descent:1} and
\ref{descent:2}, removing redundant bits of data, we obtain a description of
$\Sh_c(X;\D)$ by the following data (where, for simplicity, we restrict ourselves to
objects with complex-analytically constructible cohomology):
\begin{itemize}
    \item an object $A \in \Sh_c(U; \D)$,
    \item an object $B \in \Sh_c(Z; \D)$,
    \item morphisms $\can: \Psi_f(A) \to B$ and $\var: B \to \Psi_f(A)$,
    \item a homotopy making the diagram
        \[
        \begin{tikzcd}
            \Psi_f(A) \ar[dr, swap,"\id - T"]\ar[r,"\can"] &  B \ar[d,"\var"] \\
             &  \Psi_f(A)
        \end{tikzcd}
        \]
\end{itemize}
where, given an object $F \in \Sh_c(X; \D)$, the data corresponds to the
familiar vanishing and nearby cycle package with $A = F|_U$ and $B =
\Phi_f(F)$. The Beilinson gluing for perverse sheaves now follows from the
observation that the equivalence restricts to an equivalence involving the
exact same data, but with values in the hearts of the various perverse
$t$-structures. 

\subsection{Perverse sheaves via foliated Verdier duality}
  
\subsubsection{Approach in a nutshell}%
\label{sub:nutshell}

Our approach to the above-mentioned classification results for perverse sheaves
involves the following components: 

\begin{enumerate}[label = (\Alph *)]
    \item\label{met:1} {\em Lax descent.} For every stratification of $X$, indexed by a
        poset $L$, there is a canonical equivalence
        \begin{equation}
            \label{eq:laxdescent}
            \nu: \Sh(X;\D) \overset{\simeq}{\lra} \llim_{l \in L} \Sh(X_l;\D)
        \end{equation}
        between the $\infty$-category of $\D$-valued sheaves on $X$ and the lax
        limit over the $\infty$-categories of sheaves on the various strata
        $X_l$, $l \in L$. The equivalence is given by associating to a sheaf
        $F$, the family $\{\underline{\Gamma}_{X_l}(F)\}_{l \in L}$ of sections
        with support on the respective strata, related by lax gluing data. 

    \item\label{met:2} {\em $t$-exactness.} Assume that $\D = \D(\A)$ is the
        derived category of an abelian category $\A$. For a ``totally real
        foliated'' $L$-stratification of $X$, suitably adapted to a given
        complex analytic stratification of $X$, indexed by a poset $P$, the
        restriction of $\nu$, mildly normalized by dimension shifts, yields a
        fully faithful embedding
        \begin{equation}
            \label{eq:pureperv}
            \Perv_P(X;\A) \subset \llim_{l \in L} \Sh(X_l;\A)
        \end{equation}
        where $\Perv_P(X;\A) \subset \Sh(X;\D)$ be the abelian subcategory of
        perverse sheaves which are constructible with respect to the
        complex analytic stratification of $X$. 
\end{enumerate}

In an ideal situation, unravelling the obtained embedding \eqref{eq:pureperv}
yields an elementary, purely abelian description of the category of perverse
sheaves without any reference to derived categories. 

\subsubsection{Interpretation as foliated Verdier duality}

Lurie's variant of Verdier duality takes the form of an equivalence
\[
    \Sh(X;\D) \simeq \Sh(X;\D^{\op})^{\op}
\]
stating that, via passage to sections with compact support, any sheaf can be
identified with a corresponding cosheaf. In the context of \ref{met:1}, we may
interpret the $L$-indexed diagram $\{\Sh(X_l;\D)\}_{l \in L}$ as defining a
constructible (possibly lax) cosheaf of categories (of sheaves
along the various leaves). The equivalence \eqref{eq:laxdescent} then states
that a sheaf on $X$ may be described as a lax global section of this
cosheaf, i.e. a family of sheaves along the leaves of the foliation, related by
gluing data along the corestriction functors of the categorical cosheaf. 
Thus, while global Verdier duality identifies sheaves and cosheaves, the
equivalence \eqref{eq:laxdescent} may be interpreted as an intermediate variant
of Verdier duality that identifies sheaves on $X$ with
gadgets that behave like sheaves along the leaves and cosheaves transverse to
the leaves. With this in mind, we sometimes refer to the equivalence
\eqref{eq:laxdescent}, or rather some variant normalized by dimension shifts, as {\em
foliated Verdier duality}.

As pointed out by Kapranov-Schechtman, in many examples the classification data
for perverse sheaves exhibits some kind of ``mixed sheaf'' behaviour -- lax
descent seems to provide a convenient conceptual framework to axiomatize this
behaviour. 

\subsubsection{An illustrative example}
\label{subsec:illustrative}

We illustrate the procedure outlined in \S \ref{sub:nutshell} in the
simplest example of interest: Consider the complex plane $X = \CC$ with $P =
\{0 < 1\}$ and $P$-stratification given by $X = \{0\} \cup (\CC \setminus
\{0\})$. As a foliation, we may consider the horizontal foliation of the
quadratic differential $\frac{1}{z} dz^2$ (cf. \cite{Strebel1984}), depicted in
Figure \ref{fig:foliation}.
The leaves can be described as the fibers of the map
\[
    \lambda: X \to \RR_{\ge 0},\; z \mapsto \Im(\sqrt{z})
\]
so that we may refer to the space $\RR_{\ge 0}$ as the ``leaf space'' of the
foliation. It comes equipped with an $L= \{0 <1\}$ stratification $\RR_{\ge 0} = \{0\} \cup
\RR_{>0}$ distinguishing the different types of leaves. Pulling back this
latter stratification along $\lambda$, we obtain a ``leaf stratification'' $X
\to L$ with strata $X_1 = \CC \setminus \RR_{\ge 0}$ and $X_0 = \RR_{\ge 0}$. 
The stratum $X_1$ is exhausted by regular leaves homeomorphic to $Z = \RR$, in
fact $X_1 \cong \RR \times Z$ while the stratum $X_0$ consists of the single
``singular'' leaf $Z_0 = \RR_{\ge 0}$ with induced $P$-stratification $\{0\}
\cup \RR_{>0}$. 

\begin{figure}[htp]
    \centering
    \includegraphics[width=0.3\textwidth]{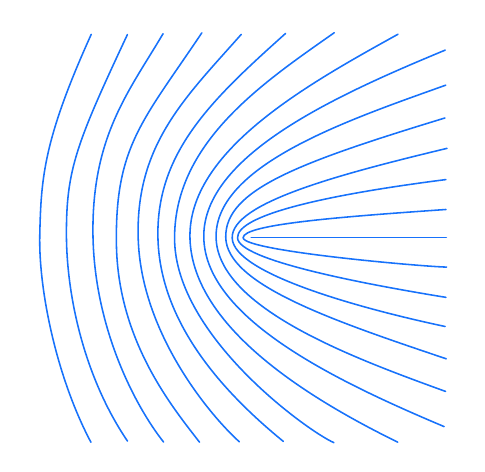}
    \caption{Horizontal foliation of $\frac{1}{z}dz^2$}
    \label{fig:foliation}
\end{figure}

According to \eqref{eq:laxdescent}, we obtain
\begin{equation}
    \label{eq:c_laxdesc}
    Sh_{P \times L}(\CC;\D) \simeq \llim(\; Sh_P(\RR_{\ge 0}; \D) \overset{\gamma_{01}}{\leftarrow} Sh_P(Z; \D)\; )
\end{equation}
By passage to an exit path presentation, we may identify $Sh_P(\RR_{\ge 0};
\D) \simeq \Fun(\{0 < 1 \}, \D)$ and $Sh_P(Z; \D) \simeq \D$. In terms of
these identifications, the gluing functor is then given by 
\[
    \gamma_{01}: \D \to \Fun(\{0 < 1 \}, \D), E \mapsto (E \to E \oplus E)
\]
so that \eqref{eq:c_laxdesc} identifies $Sh_{P \times L}(\CC;\D)$ with the
$\infty$-category of 
\begin{enumerate}
    \item triples $(A,B,E) \in \D^3$, along with  
    \item a coherent square
        \[
            \begin{tikzcd}
                E \ar{r}\ar{d} & E \oplus E \ar{d} \\
                A \ar{r} & B 
            \end{tikzcd}
        \]
    where the horizontal arrow is the diagonal map.
\end{enumerate}
The subcategory of $P$-constructible sheaves corresponds to those data for which the map $E \oplus E
\to B$ restricts to an equivalence on each summand. Finally, suppose that $\D =
\D(\A)$ is the derived category of an abelian category $\A$. Then, due to the perverse
$t$-exactness of the restriction of the equivalence \eqref{eq:c_laxdesc} to
$P$-constructible sheaves, the above data describes a perverse sheaf if and
only if the objects $A,B,E$ lie in $\A$ so that we recover the classical
description of $\Perv_P(\CC;\A)$ in terms of vanishing cycles $\Phi = A \in
\A$, nearby cycles $\Psi = B \in \A$, along with maps $\on{can}: \Psi \to \Phi$
and $\on{var}: \Phi \to \Psi$ such that $\on{var} \circ \on{can} - \id_{\Psi}$
is invertible.

\subsection*{Acknowledgements}
We thank Mikhail Kapranov and Vadim Schechtman for discussions on various
aspects of their work on perverse sheaves. We would further like to thank Felix
Loubaton for explaining to us the concept of partially lax descent for
constructible covers from $2$-topos theory. The version for stable coefficients
used in this work emerged from discussions with Fernando Abellan, Jonte
Gödicke, and Tashi Walde.

\subsection*{Funding}
The authors acknowledge support from the Deutsche Forschungsgemeinschaft (DFG,
German Research Foundation) under Germany's Excellence Strategy - EXC 2121
``Quantum Universe'' - 390833306 and the Collaborative Research Center - SFB
1624 ``Higher structures, moduli spaces and integrability'' - 506632645.

\section{Lax descent for open-closed covers}%
\label{sec:lax_topology}

In this section, we fix a locally compact Hausdorff space $X$ and a
stable presentable $\infty$-category $\D$. We make use of the functor formalism
for the categories $\Sh(Y; \D)$ for various subspaces $Y \subset X$,
established in greater generality in \cite{volpe2025six} (and briefly recalled
in \S \ref{sec:sheaves}).

We record different types of (partially lax) descent which allow us to recover the
$\infty$-category $\Sh(X)$ on a space $X$ from
\begin{enumerate}[label= \arabic*.]
    \item an open cover $X = U \cup V$,
    \item a partition $X = U \cup Z$ into an open $U \subset X$ and its closed complement $Z \subset X$, 
    \item a constructible cover $X = U \cup Z$ with $U \subset X$ open and $Z \subset X$ closed. 
\end{enumerate}

This latter type of descent (and even more general variants for covers by
constructible subsets) was explained to us by Felix Loubaton, pointing out that
it appears naturally in his approach to $2$-topos theory. Here, we prove a
variant for stable coefficients.

\subsection{Open cover}

We begin with ordinary descent for an open covering $X = U \cup V$. Denote the
inclusions of the opens $U$ and $V$ by $j: U \to X$ and $k: V \to X$ as well as
$\tilde{j}: U \cap V \to V$ and $\tilde{j}: U \cap V \to U$.
We obtain a commutative square
\[
    \begin{tikzcd}
        & \Sh(X) \ar[swap]{dl}{j^*} \ar{dr}{k^*} & \\ \Sh(U) \ar[swap]{dr}{\tilde{k}^*} & & \Sh(V)\ar{dl}{\tilde{j}^*}\\ & \Sh(U \cap V) & 
    \end{tikzcd} 
\]
which, as we claim, is a pullback square thus exhibiting the $\infty$-category
$\Sh(X)$ as the fiber product $\Sh(U) \times_{\Sh(U \cap V)} \Sh(V)$. We
provide a ``noncommutative'' proof of this well-known fact:  

\begin{prop}
    \label{prop:noncomopen}
    Let $\X$ be a stable $\infty$-category, and let $\U,\V \subset \X$ be full
    stable subcategories. Assume that
    \begin{enumerate}[label= \arabic*.]
        \item the subcategories $\U$ and $\V$ generate $\X$ (i.e. $\X$ is the
            smallest stable subcategory containing $\U$ and $\V$), and
        \item the commutative square
            \begin{equation}\label{eq:opensquare}
            \begin{tikzcd}
                \U \ar{r} &  \X \\ \ar{u} \U \cap \V \ar{r} &\ar{u} \V
            \end{tikzcd}
        \end{equation}
        is horizontally and vertically right adjointable, i.e. all functors
        admit right adjoints and the squares obtained by passing to horizontal
        (resp. vertical) right adjoints commute.
    \end{enumerate}
    Then the square of right adjoints of \eqref{eq:opensquare} exhibits an equivalence
    \[
        \X \overset{\simeq}{\lra} \U \times_{\U \cap \V} \V
    \]
\end{prop}
\begin{proof}
    We denote the embedding $\U \subset \X$ by $j_!$ and its right adjoint by
    $j^*$. We claim that, for every $x \in \X$, the cone of the counit $j_! j^*
    x \to x$ lies in $\V \subset \X$. Indeed, those $x \in \X$ for which the claim holds
    form a full stable subcategory of $\X$ so that it suffices to verify the
    statement for $u \in \U$ (where it is obvious, since the
    cone is $0$) and for $v \in \V$. In the latter case, it holds due to
    the horizontal adjointability of \eqref{eq:opensquare}, which implies $j^*
    v \in \U \cap \V \subset \V$. In particular, we obtain an exact functor $p:
    \X \to \X, x \mapsto \cone(j_!j^*x \to x)$, taking values in $\V$, so that
    there is an exact triangle
    \begin{equation}
        \label{eq:exactopen}
        j_!j^* \to \id \to p.
    \end{equation}

    Now let $K$ denote the $\infty$-category of diagrams in $\X$ of the form
    \begin{equation}
        \label{eq:universalsquare}
        \begin{tikzcd}
            u \ar{r}{\alpha} &  x\\
           \ar{r}{*} \ar{u}{*} w &  v \ar{u}{\beta}
        \end{tikzcd}
    \end{equation}
    where the morphisms marked $*$ are $(\U \cap \V,\U)$-Cartesian,
    respectively $(\U \cap \V,\V)$-Cartesian (cf. Definition \ref{defi:universal}).
    We claim that, for a given diagram $\kappa \in K$, the following conditions are equivalent:
    \begin{enumerate}[label=(\roman*)]
        \item\label{open:1} The morphisms $\alpha$ and $\beta$ are $(\U,\X)$-Cartesian and $(\V,\X)$-Cartesian, respectively. 
        \item\label{open:2} The square \eqref{eq:universalsquare} is a pushout square in $\X$. 
    \end{enumerate}
    Assuming \ref{open:1}, due to \eqref{eq:exactopen}, it suffices to check
    that the squares obtained from $\kappa$ by applying $j_!j^*$ and $p$,
    respectively, are pushout squares. This can be seen immediately, since, due
    to the adjointability assumption, the horizontal (resp. vertical) edges of
    these squares are equivalences. 

    Vice versa, assuming \ref{open:2}, we apply $j_!j^*$ to $\kappa$ to
    deduce that the canonical edge $u \to j_!j^*x$ is an equivalence (since it
    is the top horizontal edge of a pushout square whose bottom edge is an
    equivalence). In other words, the edge $\alpha: u \to x$ is
    $(\U,\X)$-Cartesian. By the analogous argument for the inclusion $\V
    \subset \X$, we obtain (using the vertical adjointability) that $\beta$ is $(\V,\X)$-Cartesian.

    Denoting by $K' \subset K$ the full subcategory of
    diagrams satisfying the equivalent conditions \ref{open:1} and
    \ref{open:2}, we obtain a roof of equivalences
    \[
        \X \overset{\simeq}{\longleftarrow}  K' \overset{\simeq}{\lra} \{ u \overset{*}{\leftarrow} w \overset{*}{\rightarrow} v \}
    \]
    given by the apparent forgetful functors. The right-hand category of
    diagrams can be identified with the $\infty$-category of Cartesian sections
    of the Grothendieck construction of the diagram
    \[
        \U \lra \U \cap \V \longleftarrow \V
    \]
    thus providing a model for the claimed limit.
\end{proof}

\begin{rem}
    \label{rem:opendescent} The above-mentioned open descent statement for
    sheaves follows from Proposition \ref{prop:noncomopen} by realizing the
    various categories of sheaves on the opens as full subcategories by means
    of $!$-embeddings and using base change as recorded in Remark
    \ref{rem: comparisonshandhypsh}.
\end{rem}

\subsection{Partition and open-closed cover}
\label{sec:part-open-clos}
We discuss a variant of descent for a cover of a space by an open and a closed
subsets -- partitions are the special case where the subsets are disjoint which
will hence be treated afterwards.  

Denote the inclusions of the opens $U$ and $Z$ by $j: U \to X$ and $i: Z \to X$
as well as $\tilde{j}: U \cap Z \to Z$ and $\tilde{i}: U \cap Z \to U$.
We obtain a commutative (since $j^* = j^!$) square
\begin{equation}
    \label{eq:laxsquare}
    \begin{tikzcd}
        & \Sh(X; \D) \ar[swap]{dl}{j^*} \ar{dr}{i^!} & \\ \Sh(U; \D) \ar[swap]{dr}{\tilde{i}^!} & & \Sh(Z; \D)\ar{dl}{\tilde{j}^*}\\ & \Sh(U \cap Z; \D) & 
    \end{tikzcd}
\end{equation}
from which we produce a partially lax cone with tip $\Sh(X; \D)$ over the commutative diagram 
\begin{equation}
    \label{eq:partiallylax}
    \begin{tikzcd}
        \Sh(U; \D) \ar{rr}{i^!j_!} \ar{dr}{\tilde{i}^!} & & \Sh(Z; \D) \ar{dl}{\tilde{j}^*}\\
          & \Sh(U \cap Z; \D) & 
    \end{tikzcd}
\end{equation}
We will show that this cone exhibits $\Sh(X; \D)$ as a partially lax limit where
the cone triangles over the edges $\tilde{i}^!$ and $\tilde{j}^*$ commute
strictly.

Again, we provide a purely ``noncommutative'' proof of this statement:

\begin{prop}
    \label{prop:noncomconst}
    Let $\X$ be a stable $\infty$-category, and let $\U,\Z \subset \X$ be full
    stable subcategories. Assume that
    \begin{enumerate}[label= \arabic*.]
        \item the subcategories $\U$ and $\Z$ generate $\X$, and
        \item all functors in the commutative square
            \begin{equation}\label{eq:constsquare}
            \begin{tikzcd}
                \U \ar{r}  & \X \\ \ar{u} \U \cap \Z \ar{r}&\ar{u} \Z 
            \end{tikzcd}
        \end{equation}
        admit right adjoints and the square is horizontally right adjointable. 
    \end{enumerate}
    Then $\X$ is a partially lax limit
    \[
        \X \overset{\simeq}{\lra} \U \times_{\U \cap \Z}^{\to} \Z.
    \]
\end{prop}
\begin{proof}
    Let $K$ denote the $\infty$-category of diagrams in $\X$ of the form
    \[
    \begin{tikzcd}
        u \ar{r}{\alpha} &  x\\
        z' \ar{r}\ar{u}{*} &  \ar[swap]{u}{\beta} z
    \end{tikzcd}
    \]
    where 
    \begin{enumerate}[label=\arabic*.]
        \item $z,z' \in \Z$, $u \in \U$, $x \in \X$, 
        \item the morphism marked $*$ is $(\Z,\X)$-Cartesian,

    \end{enumerate}
    We claim that, for a given diagram $\kappa \in K$, the following conditions are equivalent:
    \begin{enumerate}[label=(\roman*)]
        \item\label{const:1} The morphisms $\alpha$ and $\beta$ are
            $(\U,\X)$-Cartesian and $(\Z,\X)$-Cartesian, respectively. 
        \item\label{const:2} $\kappa$ is a pushout square in $\X$ and the image of the morphism $z' \to z$ under $j^\ast$ is an equivalence. 
    \end{enumerate}
    The proof is analogous to the argument for Proposition
    \ref{prop:noncomopen}: Again denote by $p = \cone(j_!j^\ast \to \id)$ (it takes values in $\Z$) and assume $\kappa \in K$ satisfies \ref{const:1}. Then $\kappa$ is a pushout square if it is after restricting along $j_!j^\ast$ and $p$ respectively. When restricting along $j_!j^\ast$ horizontal adjoinability guarantees that $j_!j^\ast(z' \to z)$ is indeed an equivalence, whilst $j_!j^\ast(u \xrightarrow{\alpha} x)$ is an equivalence by the assumption that $\alpha$ is $(\U,\X)$-Cartesian. Notice that $p(\kappa)$ is pushout if and only if $\on{totcofib}(j_!j^\ast \kappa \to \kappa) \simeq 0$, where $j_!j^\ast \kappa \to \kappa$ is interpreted as a $3$-dimensional cube. By taking iterated cofibers, we conclude that it is pushout if and only if the map $\cone(z' \to z) \to \cone(u \to x)$ is an equivalence. Since $\cone(z' \to z) = i^!\cone(u \to x)$ and $\cone(u \to x) \in \Z$ because $u \simeq j_!j^\ast x$ and $p(x) \in \Z$ by the previous arguments, the claim follows. Conversely assume \ref{const:2}, restricting $\kappa$ to $j_!j^\ast$ we see that $u \to j_!j^\ast x$ is an equivalence. To show that $z \xrightarrow{\beta} x$ is $(\Z,\X)$-Cartesian we apply $i^!$ to $\kappa$ and notice that $i^!(z' \xrightarrow{\ast} u)$ is an equivalence and, since $\kappa$ is a pushout square, deduce that $z \to i_\ast i^!x$ is an equivalence as well.
    
     Denoting by $K' \subset K$ the full subcategory of
    diagrams satisfying the equivalent conditions \ref{const:1} and
    \ref{const:2}, we obtain a roof of equivalences
    \[
        \X \overset{\simeq}{\longleftarrow}  K' \overset{\simeq}{\lra} \{ u \overset{*}{\leftarrow} z' \overset{g}{\rightarrow} z \; | \; \text{$g|\U\cap\Z$ is an equivalence}\}
    \]
    given by the apparent forgetful functors. This latter category can be identified
    with the $\infty$-category of subdivided sections
    \begin{equation}
        \label{eq:subdivsec}
        \begin{tikzcd}[column sep={3em, between origins},row sep={2em, between origins}, nodes in empty cells]
            &   &  z  &    & \\
            &&&&& \\
            & w' \ar{uur}{*}\ar{ddl}{\simeq} & & z' \ar{uul} \ar{ddr}{*} & \\
            &  & w \ar{ur}{*} \ar{ul}{\simeq} \ar{d}{\simeq} &  & \\
            w'''  & & \ar{ll}{\simeq} w'' \ar{rr}{*}& & u
        \end{tikzcd}
    \end{equation}
    of the Grothendieck construction of the diagram 
    \[
    \begin{tikzcd}
        \U \ar{rr}{i^!j_!} \ar[swap]{dr}{\tilde{j}^*} & & \Z \ar{dl}{\tilde{i}^!}\\
          & \U \cap \Z & 
    \end{tikzcd}
    \]
    thus providing a model for the claimed partially lax limit.
\end{proof}

\begin{rem}
    \label{rem:constdescent} Just as for the open descent statement,
    Proposition \ref{prop:noncomconst} provides the claimed partially lax
    descent equivalence for $\Sh(X; \D)$. Note that, in the present situation, only
    one of the adjointability conditions is guaranteed due to the
    base change square
    \begin{equation}\label{eq:basechange}
        \begin{tikzcd}
            \Sh(U; \D)& \Sh(X; \D)\ar[swap]{l}{j^*}  \\ \Sh(U \cap Z; \D)\ar{u}{\tilde{i}_! = \tilde{i}_*}  & \ar[swap]{u}{i_! = i_*} \Sh(Z; \D)\ar{l}{\tilde{j}^*} .
        \end{tikzcd}
    \end{equation}
\end{rem}

\begin{rem}
    \label{rem:recoll}
    The special case of a partition corresponds to the additional assumption
    $\U \cap \Z = 0$ in Proposition \ref{prop:noncomconst}. Here the
    subcategories $\U$ and $\Z$ form a recollement. In this situation, the
    condition on $g | \U \cap \Z$ is vacuous so that the partially lax limit
    becomes a purely lax one, recovering $\X$ as a lax limit over the diagram 
    \[
        \U \overset{i^!j_!}{\lra} \Z
    \]
    given by the gluing functor (parameterized by the $(\Z,\X)$-Cartesian edges
    with target $\U$). See \S \ref{sec:recollements} for a review of 
    foundational aspects of of recollements. 
\end{rem}

Finally, there is an oplax variant of Proposition \ref{prop:noncomconst} which
corresponds to the following geometric variation: Instead of the square
\eqref{eq:laxsquare}, consider the commutative square
\begin{equation}
    \label{eq:oplaxsquare}
    \begin{tikzcd}
        & \Sh(X; \D) \ar[swap]{dl}{j^*} \ar{dr}{i^*} & \\ \Sh(U; \D) \ar[swap]{dr}{\tilde{i}^*} & & \Sh(Z; \D)\ar{dl}{\tilde{j}^*}\\ & \Sh(U \cap Z; \D) & 
    \end{tikzcd}
\end{equation}
from which we produce a partially oplax cone with tip $\Sh(X; \D)$ over the commutative diagram 
\begin{equation}
    \label{eq:partiallyoplax}
    \begin{tikzcd}
        \Sh(U; \D) \ar{rr}{i^*j_*} \ar{dr}{\tilde{i}^*} & & \Sh(Z; \D) \ar{dl}{\tilde{j}^*}\\
          & \Sh(U \cap Z; \D) & 
    \end{tikzcd}
\end{equation}
We will show that this cone exhibits $\Sh(X; \D)$ as a partially lax limit where
the cone triangles over the edges $\tilde{i}^!$ and $\tilde{j}^*$ commute
strictly. The ``noncommutative'' formulation reads as follows: 

\begin{prop}
    \label{prop:opnoncomconst}
    Let $\X$ be a stable $\infty$-category, and let $\U,\Z \subset \X$ be full
    stable subcategories. Assume that
    \begin{enumerate}[label= \arabic*.]
        \item the subcategories $\U$ and $\Z$ generate $\X$, and
        \item all functors in the commutative square
            \begin{equation}\label{eq:opconstsquare}
            \begin{tikzcd}
                \U \ar{r}  & \X \\ \ar{u} \U \cap \Z \ar{r}&\ar{u} \Z 
            \end{tikzcd}
        \end{equation}
        admit left adjoints and the square is vertically left adjointable. 
    \end{enumerate}
    Then $\X$ is a partially oplax limit of the diagram 
    \[
    \begin{tikzcd}
        \U \ar{rr}{i^*j_*} \ar[swap]{dr}{\tilde{j}^*} & & \Z \ar{dl}{\tilde{i}^*}\\
          & \U \cap \Z. & 
    \end{tikzcd}
    \]
\end{prop}

\subsection{Application: Beilinson gluing via lax descent}%
\label{sec:beilinson_gluing_via_lax_descent}
We explain how to recover a derived variant of Beilinson gluing for perverse
sheaves as a direct corollary of (op)lax descent for a constructible cover,
both in the complex-analytic and the algebraic setting. 

\subsubsection{Complex-analytic}

Let $X$ be a complex manifold equipped with a holomorphic map $f: X \to \CC$.
Set $Z = f^{-1}(\{0\})$, $U = X \setminus Z$.

The fundamental result of Beilinson \cite{Beilinson1987howtoglue} explains how to ``glue'' the abelian category $\Perv(X)$ of perverse sheaves on $X$ from the categories $\Perv(U), \Perv(Z)$: the category $\Perv(X)$ is equivalent to the category of tuples $(\F_U, \F_Z, u, v)$ where $\F_U \in \Perv(U), \F_Z \in \Perv(Z)$ are arbitrary perverse sheaves and $u: \Psi_f^{un}(\F_U) \to \F_Z, v: \F_Z \to \Psi_f^{un}(\F_U)$ is a pair of maps such that $\operatorname{Id}_{\Psi_f^{un}(\F_U)} - v \circ u$ is invertible. Here $\Psi_f^{un}(\F_U)$ is the unipotent nearby cycles functor, which is known to be $t$-exact with respect to the perverse $t$-structure. 

Given a perverse sheaf $\F \in \Perv(X)$ one recovers the corresponding tuple as $(\F|_U, \Phi^{un}_f(\F), \can_\F, \var_\F)$, where $\Phi^{un}_f(\F)$ is the unipotent vanishing cycles functor, and $\can: \Psi_f^{un} \to \Phi^{un}_f, \var: \Phi^{un}_f \to \Psi_f^{un}$ are certain natural transformations obtained as a part of the construction of the functors $\Psi^{un}_f, \Phi^{un}_f$.

Let now $X$ be an arbitrary locally compact Hausdorff topological space, $Z$ a closed subset and $U = X \setminus Z$ a complimentary open subset. Write $i:Z \to X, j: U \to X$ for the corresponding closed and open embeddings, respectively. The results of Section \ref{sec:part-open-clos} for the stable category $\Sh(X; \D)$ of sheaves on $X$ can be formulated as follows. Applying Propositions \ref{prop:noncomconst}, \ref{prop:opnoncomconst} with $\X = \Sh(X; \D)$, $\mathcal{U} = \Sh(U;\D)$ and $\mathcal{Z} = \Sh(Z; \D)$, we get
\begin{prop}
  \label{sec:appl-beil-gluing}
  \begin{enumerate}
  \item \label{item:1} the category $\Sh(X; \D)$ is equivalent to the category of tuples $(\F_U, \F_Z, g)$, where $\F_U \in \D(U), \F_Z \in \D(Z)$ and $g: i^!j_!\F_U \to \F_Z$ is a gluing map. Given a sheaf $\F \in \D(X)$ one recovers the corresponding tuple as $(\F|_U, i^!\F, i^!(j_!j^!\F \to \F))$, where the map $j_!j^!\F \to \F$ comes from the counit of the adjunction $(j_!, j^!)$.
  \item \label{item:2} the category $\Sh(X; \D)$ is equivalent to the category of tuples $(\F_U, \F_Z, g)$, where $\F_U \in \D(U), \F_Z \in \D(Z)$ and $g:\F_Z \to i^*j_*\F_U$ is a gluing map. Given a sheaf $\F \in \D(X)$ one recovers the corresponding tuple as $(\F|_U, i^*\F, i^*(\F \to j_*j^*\F))$, where the map $\F \to j_*j^*\F$ comes from the unit of the adjunction $(j^*, j_*)$. 

  If a pair of open and closed subsets $U, Z$ is not assumed to be complementary, we also recover the following more general statement. Assume that the categories $j_!\Sh(U; \D), i_*\D(Z)$ generate $\D(X)$. Let $i, j$ be as before, and let $\tilde{j}: U \cap Z \to Z, \tilde{i}: U \cap Z \to U$ be the embeddings.  
  \item \label{item:3} The category $\Sh(X; \D)$ is equivalent to the category of tuples $(\F_U, \F_Z, g)$, where $\F_U \in \D(U), \F_Z \in \D(Z)$ and $g:i^!j_!\F_U \to \F_Z$ is a gluing map such that $\tilde{j}^!g$ is an equivalence. Given a sheaf $\F \in \D(X)$ one recovers the corresponding tuple as $(\F|_U, i^!\F, i^!(j_!j^!\F \to \F))$, where the map $j_!j^!\F \to \F$ comes from the counit of the adjunction $(j_!, j^!)$.
  \end{enumerate}
\end{prop}
  The Beilinson construction is recovered when one applies lax descent followed by an oplax descent for a suitable auxiliary stratification.

  Let again $Z = f^{-1}(0), U = X \setminus Z$. Let $Z_+ = f^{-1}(\mathbb{R}_{\geq 0})$, $V = Z_+ \setminus Z$, and $W = X \setminus Z_+$.
We denote the corresponding embeddings as follows: 
\begin{equation}
    \label{eq:glossary0}
    \begin{tikzcd}
        Z \ar[r,equal]\ar[d,swap,hook', "k"] & Z \ar[d,swap,hook', "i"] &  \\
        Z_+ \ar[r,hook,"s"] & X & W \ar[l,swap,hook',"t"] \\
        V \ar[r,hook,"r"] \ar[u,hook,"l"] & U \ar[u,hook,"j"] & 
    \end{tikzcd}
\end{equation}
  
  Applying the lax descent to a pair $U, Z_+$ one gets that the category $\Sh(X; \D)$ is equivalent to the category of tuples $(\F_U, \F_{Z_+}, g)$, where $\F_U \in \D(U), \F_{Z_+} \in \D(Z_+)$ and \[g:s^!j_!\F_U \to \F_{Z_+}\] is a gluing map such that $l^!g$ is an equivalence. 

  We now apply the \emph{oplax} descent with respect to the open-closed decomposition of $Z_+ = Z \cup V$ for the sheaves $s^!j_!\F_U, \F_{Z_+}$. It follows that the data of $(\F_U, \F_{Z_+}, g)$ is equivalent to the data of $(\F_U, \F_V, \F_Z, g_V, g_Z, g_+)$ for $g_V, g_Z, g_+$ are maps in the following commutative diagram
    \begin{equation}
      \label{eq:1}
      \begin{tikzcd}
        k^*s^!j_!\F_U \arrow[d] \arrow[r, "g_Z"] &  \F_Z \arrow[d, "g_+"]  \\
 k^*l_*l^*s^!j_!\F_U \arrow[r, "k^*l_*g_V"']      & k^*l_* \F_V  
      \end{tikzcd}
    \end{equation}

    Here we think of $g_V, g_Z$ as of components of $g$ in the ``oplax gluing coordinates'' on $Z_+$.

    \begin{rem}
      \label{sec:appl-beil-gluing-1}
      Note that
      the condition that $l^!g$ is an equivalence makes $\F_V$ determined by the rest of the data. We may, thus, replace the bottom row with the identity map.
    \end{rem}
  It follows from Proposition \ref{sec:appl-beil-gluing} that, given $\F \in \Sh(X; \D)$, the sheaf $\F_Z$ is defined as $k^*s^!\F.$ 
\begin{defi}
    \label{defi:nearby_vanishing}
    Given sheaves $\F \in \Sh(X; \D), \F_U \in \Sh(U; \D)$, we define
    \begin{enumerate}
        \item the {\em vanishing cycles} $\Phi_f(\F) := k^* s^! \F \in \Sh(Z; \D)$, 
        \item the {\em left nearby cycles} $\Psi^l_f(\F_U) := \Phi_f(j_!\F_U) \in \Sh(Z; \D)$, and
        \item the {\em right nearby cycles} $\Psi^r_f(\F_U) := \Phi_f(j_*\F_U)
            \simeq k^* l_* r^! \F_U \in \Sh(Z; \D)$ where the equivalence follows
            from base change along the bottom square of \eqref{eq:glossary0}. 
    \end{enumerate}
\end{defi}
\begin{rem}
    \label{rem:leftandright}
    By applying $\Phi_f$, to the exact triangle
    $j_!j^* \F \to \F \to i_* i^* \F$,
    we obtain an exact triangle
    \[
        \Psi^l_f(j^*\F) \overset{\can}{\lra} \Phi_f(\F) \lra i^* \F.
    \]
    Similarly, by applying $\Phi_f$, to the exact triangle
    $ i_* i^! \F \to \F \to j_*j^* \F$,
    we obtain an exact triangle
    \[
        i^!\F \lra \Phi_f(\F) \overset{\var}{\lra} \Psi^r_f(j^*\F)
    \]
    which, using base change, can alternatively be obtained by applying the
    triangle 
        $k_*k^! \to \id_{Z_+} \to l_*l^*$
    to $s^!\F$ and composing with $k^*(-)$. 
\end{rem}

Substituting these definitions and keeping in mind Remark \ref{sec:appl-beil-gluing-1}, we obtain the following gluing data describing the sheaf $\F \in \Sh(X; \D)$.
    \begin{equation}
      \label{eq:3}
      \begin{tikzcd}
        \Psi^l_f j^*\F \arrow[d] \arrow[r, "\can"] &  \Phi_f(\F) \arrow[d, "\var"]  \\
 \Psi^r_f j^*\F \arrow[r, "\Id"']      & \Psi^r_f \F_U 
      \end{tikzcd}
    \end{equation}

It remains to identify the left vertical arrow in this diagram. Note that, by definition, it only depends on the restriction $j^*\F$ of $\F$ to $U$.

\begin{prop}
    \label{prop:canvar}
    Let $\F \in \Sh(X; \D)$.
    \begin{enumerate}
        \item The composite 
            \[
                \Psi^l_f(j^*\F) \overset{\can}{\to} \Phi_f(\F) \overset{\var}{\to} \Psi^r_f(j^*\F)
            \]
            decomposes as $\var \circ \can = \tau_+ - \tau_-$ for canonical maps $\tau_+$ and $\tau_-$.
        \item Suppose that $\F$ is constructible with respect to a complex
            analytic stratification of $X$. Then $\tau_+$ and $\tau_-$ are equivalences. 
    \end{enumerate}
\end{prop}
\begin{proof}
    By whiskering the map $\can$ with the unit $\id \to l_*l^*$ we obtain a commutative square
    \begin{equation}
        \label{eq:varcansquare}
        \begin{tikzcd}
            k^* s^! j_! j^* \F \ar{r}{\can}\ar[d, "h"'] & k^* s^! \F  \ar{d}{\var}\\
            k^* l_* l^* s^! j_! j^*\F \ar{r}{\simeq} & k^* l_* l^* s^! \F
        \end{tikzcd}
    \end{equation}
    Let $W_1 = f^{-1}(\{\Re <0\} \cup \{\Im >0\})$ and $W_2 = f^{-1}(\{\Re <0\}
    \cup \{\Im <0\})$ so that $W_1$ and $W_2$ cover $W$ with intersection $f^{-1}(\{\Re
    <0\})$. Denoting the embeddings of $W_1$, $W_2$, and $W_1 \cap W_2$ into $X$
    by $a$, $b$, and $c$, respectively, we obtain a Cech exact triangle
    \begin{equation}
        \label{eq:resot}
        t_*t^* \to a_*a^* \oplus b_*b^* \to c_*c^*
    \end{equation}
    Now expressing $s_*s^!$ as the fiber of the map  
    \[
        \id \to t_* t^*
    \]
    and using \eqref{eq:resot} to resolve $t_*t^*$, we obtain that the map $h$
    can be described as the map induced on horizontal fibers of the square 
    \begin{equation}
        \label{eq:resolution}
        \begin{tikzcd}
            0 \ar{r}\ar{d} & i^*(t_*t^*\F) \ar{d}{(t_1,-t_2)}\\
            i^*(s_*l_*l^*s^*\F) \ar{r}{r} & k^*( l_*l^*(a_*a^*\F \oplus b_*b^*\F))
        \end{tikzcd}
    \end{equation}
    This implies the claimed sum decomposition where
    $\tau_+$ and $\tau_-$ are the maps induced by $(t_1,0)$ and $(0,-t_2)$,
    respectively. If $\F$ is constructible then it can be deduced, for example, from \cite[Lemma 1.1.1, Example 1.1.3]{zbMATH02012928} that $t_1$, $t_2$ and both components
    of $r$ in \eqref{eq:resolution} are equivalences which implies that both
    $\tau_1$ and $\tau_2$ are equivalences as well.
\end{proof}

Summarizing the discussion in this section, we obtain
\begin{cor}[Stable Beilinson gluing]
    \label{cor:derivedbeilinson}
    There is an equivalence between $\Sh(X; \D)$ and the $\infty$-category comprised of
    \begin{enumerate}
        \item a sheaf $\F_Z$ on $Z$,
        \item a sheaf $\F_U$ on $U$,
        \item maps $a: \Psi^l_f(\F_U) \to \F_Z$ and $b: \F_Z \to \Psi^r_f(\F_U)$, 
        \item a homotopy identifying $ba$ with $\tau_1 - \tau_2$.
    \end{enumerate}
\end{cor}

\begin{rem}
    \label{rem:nearby}
    For constructible $\F \in \Sh(U; \D)$, we refer to $\Psi_f(\F) = \Psi^l_f(\F)$ as
    the {\em nearby cycles} and, by convention, use $\tau_-$ to implicitly
    identify $\Psi_f(\F)$ with $\Psi^r_f(\F)$. We further obtain an automorphism $\tau =
    (\tau_+)^{-1} \circ \tau_-$ of $\Psi_f(\F)$, called {\em monodromy}. In
    particular, we recover the familiar relation 
    \[
        \var \circ \can \simeq \id - \tau.
    \]
\end{rem}

\begin{rem}
    \label{rem:texactbeilinson}
    In view of Remark \ref{rem:nearby}, when restricting the equivalence from
    Corollary \ref{cor:derivedbeilinson} to complex analytically constructible
    sheaves, then we obtain an equivalence with 
    \begin{enumerate}
        \item a constructible sheaf $\F_Z$ on $Z$,
        \item a constructible sheaf $\F_U$ on $U$,
        \item maps $a: \Psi_f(\F_U) \to \F_Z$ and $b: \F_Z \to \Psi_f(\F_U)$, 
        \item a homotopy identifying $ba$ with $\id - \tau$,
    \end{enumerate}
    and, restricting further to perverse sheaves, we obtain an equivalence with
    such data where, in addition, the sheaves $\F_Z$ and $\F_U$ are perverse. This
    latter statement follows from the classical $t$-exactness of the nearby and
    vanishing cycles functors. 
\end{rem}

    \begin{rem}
      Note that the description we obtained is formally different from Beilinson's description in that the full (as opposed to unipotent) nearby and vanishing cycles appear. It is easy to see, however, that the commutativity of the diagram \eqref{eq:3} implies that the non-unipotent part of each of its corners is determined by $\F_U$, and all its maps are isomorphisms on the non-unipotent part.
    \end{rem}

\subsubsection{Algebraic}

We can also consider the following algebraic variant of the construction above. Fix a prime $\ell$. All schemes we will consider will be quasi-compact, quasi-separated schemes defined over $\mathbb{Z}[1/\ell]$. For such a scheme $X$, let $\Sh(X, \bar{\mathbb{Q}}_\ell)$ be a stable category of \'{e}tale $\bar{\mathbb{Q}}_\ell$-sheaves on $X$ \cite{zbMATH07716347}.

Let $S = \mathrm{Spec}(V)$ be a (for simplicity, strict) henselian trait with a special point $t$, generic point $\eta$ and geometric point $\bar{\eta}$ over $\eta$. Let $\bar{S} = \mathrm{Spec}(\bar{V})$, where $\bar{V}$ is the integral closure of $V$ in $k(\bar{\eta})$.

Let $X$ be a scheme over $S$ with a structure map $f: X \to S$. Then write
$Z = X_{\bar{t}}$,
$U = X_\eta$,
$Z_+ = X \times_S \bar{S}$,
$V = X_{\bar{\eta}}$,
with maps $i, j, k, l, s$ being the natural maps between the same-named spaces from the previous subsection.

The triangles of the Remark \ref{rem:leftandright} are \cite[(3.13, 3.14)]{luDualityNearbyCycles2019}, with the composition $\var \circ \can$ identified in \cite[Remark 2.17]{luDualityNearbyCycles2019}.

\begin{rem}
  \label{sec:algebr-vari-constr}
  Note that it follows from the construction that the functor $s^!$ is t-exact with respect to the absolute perverse t-structure on the source and relative (with respect to the map $X \to \bar{S}$) perverse t-structure on the target, see \cite{hansenRelativePerversity2023}. We think it would be interesting to investigate an algebraic analogue of the developments in Section \ref{sec:foliations_and_purity} in this context.
\end{rem}
\section{Lax descent for stratifications}%
\label{sec:stratifications}

The classical concept of recollement in the prototypical captures the simplest kind of
stratification given by a decomposition of a space into an open subset and its
closed complement. Of course, iterated recollements can then be used to analyze
stratifications with more strata, but this approach has its limitations. 
As demonstrated in \cite{amr:stratified}, the framework of $\infty$-categories
offers a beautiful and systematic framework where such nested recollements can
be controlled by lax universal properties. The starting point is the
perspective of \S \ref{sub:recollements_as_op_lax_limits}: recollements can be viewed
as data that exhibits the category $\X$ as an (op)lax limit of the associated
gluing functors (cf. Corollary \ref{cor:lax}). In this section, we survey some
results of \cite{amr:stratified}, adapted to our context -- for the reader's
convenience we provide proofs.

\subsection{An example: The \texorpdfstring{$A_3$}{A3}-quiver}
\label{sub:quiver}

Before investigating the general situation, it is useful to consider a toy
example which nevertheless captures all general features. To this end, let $\D$
be a stable $\infty$-category and consider the $\infty$-category
\[
    \X = \Fun([2],\D) \simeq \{ a_0 \to a_1 \to a_2 \}
\]
of representations of the $A_2$-quiver. We consider the stable subcategories
\[
    \begin{matrix}
        \A_0 & = \{ a_0 \to 0 \to 0 \}\\
        \A_1 & = \{ 0 \to a_1 \to 0 \}\\
        \A_2 & = \{ 0 \to 0 \to a_2 \}
    \end{matrix}
\]
of $\X$. Note that we may interpret this datum as a nested recollement, first
combining $\A_0$ and $\A_1$, as a recollement of $\Fun([1],\D)$ which can then
be combined with $\A_2$ to form a recollement of $\X$. But, of course, this is
unsatisfying since it depends on ad-hoc choices and will be hard to implement
systematically in more complicated situations. We therefore investigate the
situation more directly, determining the functors corresponding to the various
$(\eR,\eS)$-Cartesian edges (cf. Definition \ref{defi:universal} in $\X$. 

We determine an analogue of the lax limit cone \eqref{eq:laxtriangle} for the
given ``$3$-component recollement''. Determining the various
$(\A_i,\X)$-Cartesian as well as $(\A_i,\A_j)$-Cartesian edges in $\X$ by
direct computation, we observe that they describe the following lax diagram of
functors 
\begin{equation}\label{eq:laxcone}
    \begin{tikzcd}[sep=1cm]
        & \X \ar[swap,bend right]{ddl}{\fib(a_0 \to a_1)} \ar[d, "\fib(a_1 \to a_2)" description]\ar[bend left]{ddr}{a_2}&\\ 
        & \ar[swap]{ld}{[-1]} \A_1 & \\
        \A_0 && \ar[swap]{lu}{[-1]}\A_2 \ar{ll}{0}
    \end{tikzcd}
\end{equation}
where we have identified each $\A_i$, via projection to the nonzero component, with $\D$. 

\begin{rem}
    \label{rem:laxgluing}
    We already observe one curious feature of \eqref{eq:laxcone}: As to be
    expected from a lax limit cone, the triangles connecting the tip of the
    cone to its base only commute laxly, i.e. up to noninvertible natural
    transformation. But more remarkably, also the ``diagram of gluing
    functors'' that forms the base of the cone does only commute {\em laxly}. 
\end{rem}

A bit more formally, we may construct this diagram as a fibration as in Remark
\ref{rem:fibration}: Let $\G$ denote the full subcategory of $\X \times
\Delta^3$ on the vertices of $\A_i \times \{i\}$ and $\X \times \{3\}$. Then
the projection functor $\pi: \G \to \Delta^3$ is a locally Cartesian fibration
modelling the lax diagram. 
We formulate, for a general locally Cartesian fibration $\Delta^3 =
(\Delta^2)^{\rhd}$, the criterion for it to model a lax limit cone:

\begin{prop}
    \label{prop:laxlimcone}
    Let $\pi: \G \to \Delta^3$ be a locally Cartesian fibration. 
    \begin{itemize}
        \item Denote by $\Gamma^*(\G)$ the $\infty$-category of diagrams of the form
            \begin{equation}
                \label{eq:bary}
                \begin{tikzcd}[column sep={3em, between origins},row sep={2em, between origins}, nodes in empty cells]
                    &   &   x_{1} &    & \\
                    &&&&& \\
                    & x_{01} \ar{uur}{*}\ar{ddl} & & x_{12} \ar{uul} \ar{ddr}{*} & \\
                    &  & x_{012} \ar{ur}{*} \ar{ul} \ar{d} &  & \\
                    x_0  & & \ar{ll} x_{02} \ar{rr}{*}& & x_{2} 
                \end{tikzcd}
            \end{equation}
            in $\G$ where, for $I \subset [2]$, we have $\pi(x_I) = \min(I)$ and all
            arrows marked $*$ are locally $\pi$-Cartesian. 
        \item Denote further by $\Gamma(\G)$ the $\infty$-category of cone
            shaped diagrams in $\G$ with base diagram as in \eqref{eq:bary}
            augmented by a final tip $x_3 \in \pi^{-1}(3)$ such that all edges
            $x_i \to x_3$ are locally $\pi$-Cartesian over the edges $i \to 3$, $0 \le
            i \le 2$. 
    \end{itemize}
    Then the forgetful functor $\tau: \Gamma(\G) \to \G_3$ given by projection to the
    tip vertex $x_3$ is a trivial Kan fibration.
\end{prop}

The proof will be given in greater generality in \S \ref{sub:laxlimcones}. In
consequence, by composing an inverse to the functor $\tau$ with the projection
functor $\beta: \Gamma(\G) \to \Gamma^*(\G)$ to the base of the cone, we obtain a functor
\[
    \lambda: \G_3 \to \Gamma^*(\G)
\]
which associates to every object of $\G_3$ a canonical diagram as in
\eqref{eq:bary}. In our running example, this functor associates to a quiver
representation $a_0 \to a_1 \to a_2$, the diagram
\begin{equation}
    \label{eq:pullbackcone}
\begin{tikzcd}[column sep={4em, between origins},row sep={3em, between origins}, nodes in empty cells]
    &   &   \fib(a_1 \to a_2) &    & \\
    &&&&\\
    & \fib(a_1 \to a_2)[-1] \ar{uur}{*}\ar{ddl} & & a_2[-1] \ar{uul} \ar{ddr}{*} & \\
    &  & a_2[-2] \ar{ur}{*} \ar{ul} \ar{d} &  & \\
    \fib(a_0 \to a_1)  & & \ar{ll} 0 \ar{rr}{*}& & a_2.
\end{tikzcd}
\end{equation}

\begin{defi}
    \label{defi:laxlimcone}
    Let $\pi: \G \to \Delta^3$ be a locally Cartesian fibration. We say that
    $\pi$ models a lax limit cone, if the projection functor
    \[
        \beta: \Gamma(\G) \to \Gamma^*(\G)
    \]
    is an equivalence of $\infty$-categories. 
\end{defi}

\begin{rem}
    \label{rem:informal} A bit more informally, we may formulate the statement
    that $\pi$ models a lax limit cone as follows: 
    Unravelling the data comprised in a diagram of the
    shape \eqref{eq:bary}, and using Proposition \ref{prop:laxlimcone}, we
    obtain an equivalence between the $\infty$-category $\G_3$, forming the tip
    of the cone, and the $\infty$-category described informally by the following data:
    \begin{itemize}
        \item objects $x_0 \in \G_0$, $x_1 \in \G_1$,  $x_2 \in \G_2$,
        \item morphisms $f: F_{0<1}(x_1) \to x_0$, $g: F_{1<2}(x_2) \to x_1$, $h: F_{0<2}(x_2) \to x_0$,
        \item a homotopy filling the square
            \[
            \begin{tikzcd}
                F_{0<1} F_{1<2}(x_2) \ar{r}{F_{0<1}(g)} \ar{d}{\alpha_{x_2}} & F_{0<1}(x_1)\ar{d}{f}\\
                F_{0<2}(x_2) \ar{r}{h} & x_0
            \end{tikzcd}
            \]
    \end{itemize}
    where $F_{i<j}: \G_j \to \G_i$ are choices of the various functors modelled
    by the locally Cartesian fibration $\pi$ and $\alpha: F_{0<1} F_{1<2}
    \Rightarrow F_{0<2}$ corresponds to the natural transformation (also
    encoded by $\pi$). 
\end{rem}

As will be shown within the more general context of ``stratified stable
$\infty$-categories'', for the example at hand, the locally Cartesian fibration
$\pi$ does indeed model a lax limit cone which exhibits the category $\X$ as a
lax limit of the lax diagram
\[
\begin{tikzcd}[sep=.5cm]
    & \ar[swap]{ld}{[-1]} \D & \\
    \D && \ar[swap]{lu}{[-1]}\D. \ar{ll}{0}
\end{tikzcd}
\]
Eliminating redundand parts of the diagram \eqref{eq:bary}, the
$\infty$-category $\Gamma^*(\G)$ can be identified with the category of $3$-term
complexes in $\D$, i.e. diagrams
\[
\begin{tikzcd}
    x \ar{r}\ar{d} & y\ar{d} \\
    0 \ar{r} & z
\end{tikzcd}
\]
in $\D$ where $0$ is a zero object (and no exactness constraints). The
composite equivalence 
\[
    \X \overset{\simeq}{\longleftarrow} \Gamma(\G) \lra \Gamma^*(\G)
\]
associates to a quiver representation $a_0 \to a_1 \to a_2$ the $3$-term
complex
\[
\begin{tikzcd}
    a_2[-2] \ar{r}\ar{d} & \fib(a_1 \to a_2)[-1]\ar{d} \\
    0 \ar{r} & \fib(a_0 \to a_1)
\end{tikzcd}
\]
which is part of the diagram \eqref{eq:pullbackcone}. This statement, and its
generalization to the $A_n$-quiver, recovers the classical result that derived
representations of this quiver are equivalent to derived representations of the
$A_n$-quiver with zero relations for consecutive arrows. We refer to
\cite{rio:sod} for detailed proofs of these statements where they are used to
establish a general theory of stable semiorthogonal decompositions of finite length.

\subsection{Stratified stable \texorpdfstring{$\infty$}{infty}-categories}
\label{sec:stratified_stable_infty_categories}

Throughout, we fix a stable $\infty$-category $\X$ and a finite poset $P$. We
give a definition of a $P$-stratification of $\X$ that immediately puts the
focus on the strata, which are of primary concern since they appear in the
gluing diagram.

\begin{defi}
    Let $\X$ be a stable $\infty$-category and $P$ a finite poset. 
    A {\em right-admissible $P$-stratification} of $\X$ consists of a collection of full stable subcategories $\X_p \subset \X$ for $p \in P$, satisfying the following conditions:
    \begin{enumerate}
        \item The subcategories $\X_p \subset \X$ are right-admissible.
        \item Denote for any subset $Q \subset P$ by $\X_{Q} \subset \X$ the
            smallest full stable subcateory containing all subcategories
            $\X_q$, $q \in Q$. Then, for every downward-closed subset $A
            \subset P$ (i.e. a closed subset in the Alexandrov topology), the
            inclusion $\X_A \subset \X$ is admissable.
        \item We have $\X_P = \X$.
        \item\label{strat:4} If $p \nleq q$ in $P$, then $\Map_{\X}(x_p,x_q) \simeq \{0\}$ for $x_p \in \X_p$ and $x_q \in \X_q$.
    \end{enumerate}
    A {\em left-admissible $P$-stratification} of $\X$ consists of a collection
    of full stable subcategories $\X_p \subset \X$ for $p \in P$, such that
    $\{\X_p^{\op}\}$ defines a right-admissible $P$-stratification of $\X^{\op}$. 
\end{defi}

\begin{exa}
    \label{exa:strat_space}
    Let $X$ be a locally compact Hausdorff space and $\D$ a presentable stable
    $\infty$-category. Let $P$ be a poset and $X \to P$ be a stratified space.
    For $p \in P$, denote by $f_p: X_p \hra X$ the embedding of the $p$-stratum
    of $X$. Then the collection $\{(f_p)_*\Sh(X_p;\D)\}$ (resp.
    $\{(f_p)_!\Sh(X_p;\D)\}$) forms a right-admissible (resp. left-admissible)
    $P$-stratification of $\Sh(X;\D)$.
\end{exa}

\begin{prop}
    \label{prop:recollement}
    Let $\{\X_p\}$ be a right-admissible (resp. left-admissible) $P$-stratification of $\X$ and let $A \subset
    P$ be downward-closed. Then there is a recollement
    \[
        \begin{tikzcd}
            \X_A \ar{r} &\ar[shift left=.5em,swap]{l} \ar[shift right=.5em, swap]{l} \X \ar{r} & \ar[shift left=.5em]{l}{\beta} \ar[shift right=.5em, swap]{l}{\alpha} \X_{P \setminus A}.
        \end{tikzcd}
    \]
    where $\alpha$ (resp. $\beta$) indicates the inclusion functor $\X_{P \setminus A} \subset
    \X$.   
\end{prop}
\begin{proof}
    We verify the statement for right-admissible stratifications. Since $\X_A
    \subset \X$ is admissible, there is a recollement
    \[
        \begin{tikzcd}
            \X_A \ar{r} &\ar[shift left=.5em,swap]{l} \ar[shift right=.5em, swap]{l} \X \ar{r} & \ar[shift left=.5em]{l} \ar[shift right=.5em, swap]{l}{\alpha} \leftindex^{\perp} {\X}_A.
        \end{tikzcd}
    \]
    where $\alpha$ is the inclusion functor (see Appendix \ref{sec:recollements}). 
    By \ref{strat:4}, we have $\X_{P \setminus A} \subset \leftindex^{\perp}
    {\X}_A$, we need to show $\leftindex^{\perp} {\X}_A \subset \X_{P \setminus A}$. 
    Consider the functor 
    \[
        p = \fib(\id \to i_*i^*): \X \to \X. 
    \]
    We first claim that it takes values in $\X_{P \setminus A}$: since $\X_P =
    \X$, it suffices to verify this for objects in $\X_A$ and objects in $\X_b$
    for $b \notin A$. For $x \in \X_A$ we have $p(x) = 0$. For $y \in \X_b$
    with $b \notin A$, we have, by \ref{strat:4}, for every $x \in \X_A$, 
    \[
        * \simeq \Map_{\X}(y,x) \simeq \Map_{\X_A}(i^*y, x)
    \]
    so that $i^*y \simeq 0$ and hence $p(y) \simeq y$. This shows $p(\X)
    \subset \X_{P \setminus A}$. 
    Now by the last part of the argument, we actually also have $p(z) \simeq z$
    for every $z \in \leftindex^{\perp} {\X}_A$, so that $\leftindex^{\perp}
    {\X}_A \subset \X_{P \setminus A}$. 
\end{proof}

\begin{rem}
    \label{rem:leftvsright}
    The concepts of left- and right-admissible stratifications are in fact
    different perspectives on the same structure: Given a right-admissible
    $P$-stratification $\{\X_p\}$ of $\X$ consider, for $p \in P$, the
    recollement
    \[
        \begin{tikzcd}
            \X_{\not\ge p} \ar{r} &\ar[shift left=.5em,swap]{l} \ar[shift right=.5em, swap]{l} \X \ar{r} & \ar[shift left=.5em]{l}{\beta} \ar[shift right=.5em, swap]{l}{\alpha} \X_{\ge p}.
        \end{tikzcd}
    \]
    and define $\X_p' := \beta(\X_p)$. Then $\{\X_p'\}$ forms a
    left-admissible $P$-stratification. In other words, for a closed subset $A
    \subset P$, the subcategories $\X_A \subset \X$ agree for both concepts but
    the ``open complements'' are taken to be either the left- or the
    right-orthogonal complements of $\X_A$ in $\X$. 
\end{rem}

\begin{defi}
    \label{defi:grothendieck} Let $\X$ be a stable $\infty$-category equipped
    with a right-admissible $P$-stratification $\{\X_p\}$. We introduce the
    Grothendieck construction $\pi: \G(\X) \to P^{\rhd}$ by letting
    \[
        \G(\X) \subset \X \times P^{\rhd}
    \]
    be the full (generally {\em not} stable) subcategory spanned by the
    vertices of $\X_l \times \{l\}$ along with $\X \times \{*\}$. The map $\pi$
    is a locally Cartesian fibration.
\end{defi}

\begin{thm}\label{thm:lax_lim}
    Let $\X$ be a stable $\infty$-category equipped with a right-admissible
    $P$-stratification $\{\X_p\}$. 
    For a finite poset $P$ the diagrams in $\Gamma(\G(\X)) =
    \Fun_{/P^{\rhd}}^{\natural}(\sd(P)^{\rhd},\G(\X))$ are precisely the
    colimit diagrams in $\X$ which restrict to diagrams in
    $\Fun_{/P}^{\natural}(\sd(P),\G(\X))$. In particular, the restriction 
    \[
        \Gamma(\G(\X)) \to \Gamma^*(\G(\X))
    \]
    is a trivial Kan fibration exhibitting $\X$ as a lax limit of $\G(\X)|P$.
\end{thm}

\begin{rem}
    \label{rem:informal_lax_lim}
    A bit more informally speaking, Theorem \ref{thm:lax_lim}
    is saying that, given a right-admissible $P$-stratification of $\X$ with
    strata $\X_p$, $p \in P$, we may recover $\X$ as the lax limit of the lax
    gluing diagram
    \[
        \gamma: P^{op} \to \Cat_{\infty},\; p \mapsto \X_p
    \] 
    obtained by straightening the locally Cartesian fibration $\G(\X)|P \to P$. 
\end{rem}

In preparation of the proof, we first recall the situation for a recollement
\[
\begin{tikzcd}
	{\A} & \X & {\B}
	\arrow[from=1-1, to=1-2]
	\arrow[shift left=3, from=1-2, to=1-1]
	\arrow[shift right=3, from=1-2, to=1-1,]
	\arrow[from=1-2, to=1-3]
	\arrow[shift left=3, from=1-3, to=1-2, "i_*"]
	\arrow[shift right=3, from=1-3, to=1-2, "i_!"']
\end{tikzcd}
\]
of stable $\infty$-categories, which correspondends to the case $P=\Delta^1$.
Recall that an edge $i_!b \to x$ in $\X$ is $(i_! \B,\X)$-Cartesian if and only if the fiber $\fib(i_!b \to x)$ lies in $\A$.
Similarly, a morphism $a \to x$ in $\X$ is $(\A,\X)$-Cartesian if and only if $\fib(i_* a \to x) \in i_* \B$.

\begin{lem}\label{lem:cart_bicat}
    Let
    \[\begin{tikzcd}
	{a'} & {a} \\
        {i_!b} & x \\
	{}
	\arrow[, from=1-1, to=1-2]
	\arrow["h"', from=1-1, to=2-1]
	\arrow["g", from=1-2, to=2-2]
	\arrow["f"', from=2-1, to=2-2]
    \end{tikzcd}\]
    be a biCartesian square in $\X$ with $a,a' \in \A$ and $b \in B$. Then
    \begin{itemize}
        \item $f$ is $(i_! \B,\X)$-Cartesian,
        \item $g$ is $(\A,\X)$-Cartesian if and only if $h$ is $(\A,i_! \B)$-Cartesian.
    \end{itemize}
\end{lem}
\begin{proof}
    Since $\cof(i_! b \to x) \simeq \cof(a' \to a) \in \A$, the first part follows. For the second claim, we extend the diagram to a diagram of biCartesian squares
    \[
    \begin{tikzcd}
	y & {a'} & a & 0 \\
	0 & {i_!b} & x & z \\
	& {}
	\arrow[from=1-1, to=1-2]
	\arrow[from=1-1, to=2-1]
	\arrow[from=1-2, to=1-3]
	\arrow["h"', from=1-2, to=2-2]
	\arrow[from=1-3, to=1-4]
	\arrow["g", from=1-3, to=2-3]
	\arrow[from=1-4, to=2-4]
	\arrow[from=2-1, to=2-2]
	\arrow["f"', from=2-2, to=2-3]
	\arrow[from=2-3, to=2-4]
    \end{tikzcd}
    \]
    Note that $h$ is $(\A,i_! \B)$-Cartesian if and only if $y \in i_* \B$, i.e. the left square is an $M$-square by \cite[2.4.2]{DKS:spherical}.
    Further, $z \cong y[1]$ and $g$ is $(\A,\X)$-Cartesian if and only if $z \in i_* \B$ which concludes the argument.
\end{proof}

\begin{cor}
    \label{cor:recoll_case}
    Let
    \[\begin{tikzcd}
	{a'} & {a} \\
        {i_!b} & x \\
	{}
	\arrow[, from=1-1, to=1-2]
	\arrow["h"', from=1-1, to=2-1]
	\arrow["g", from=1-2, to=2-2]
	\arrow["f"', from=2-1, to=2-2]
    \end{tikzcd}\]
    be a square in $\X$ with $a,a' \in \A, b \in \B$ and $h$ $(\A,i_! \B)$-Cartesian. Then the following are equivalent:
    \begin{enumerate}
        \item The square is biCartesian.
        \item The edges $f$ and $g$ are $(i_! \B,\X)$-Cartesian and $(\A,\X)$-Cartesian respectively.
    \end{enumerate}
\end{cor}
\begin{proof}
    A biCartesian square fulfills the second criterion by \ref{lem:cart_bicat}. Conversely, extend the square to a diagram
\[\begin{tikzcd}
	{a'} \\
	& y & a \\
	& {i_!b} & x & {} \\
	& 0 & z
	\arrow[from=1-1, to=2-2]
	\arrow[from=1-1, to=2-3]
	\arrow["h"', from=1-1, to=3-2]
	\arrow[from=2-2, to=2-3]
	\arrow[from=2-2, to=3-2]
	\arrow["g", from=2-3, to=3-3]
	\arrow["f"', from=3-2, to=3-3]
	\arrow[from=3-2, to=4-2]
	\arrow[from=3-3, to=4-3]
	\arrow[from=4-2, to=4-3]
\end{tikzcd}\]
where both squares are biCartesian. Since $f$ is $(i_! \B,\X)$-Cartesian, we have $z \in \A$ and therefore also $y \simeq \fib(a \to z) \in \A.$
It follows from \ref{lem:cart_bicat} that the edge $y \to i_!b$ is $(\A,i_!\B)$-Cartesian. Since $h$ is $(\A,i_!\B)$-Cartesian as well,
the edge $a' \to y$ is an equivalence, so that the original square is also biCartesian.
\end{proof}

\begin{proof}[Proof of Theorem \ref{thm:lax_lim}]
    We will inductively reduce the statement to the case of a single recollement. Let $p \in P$ be a maximal element. We write $P'=P\backslash \{p\}$.
    We implement a decomposition argument by introducing an intermediate extension: Let $\sd(P)^{\op}_p \supset \sd(P)^{\op}$ denote the
    poset obtained from $\sd(P)^{\op}$ by adjoining:
    \begin{itemize}
        \item A maximal element $M$ such that, for $S \subset P$, $S < M$ if and only if $p \notin S$.
        \item An element $C$ with $C < \{p\}$, $C < M$, such that, for $S \subset P$, $S<C$ if and only if $p \in S$ and
            $S \neq \{p\}$.
    \end{itemize}
    Let $D:(\sd(P)^{\op})^{\triangleright} \to \X$ be a diagram whose restriction lies in $\Fun^{\natural}_{/P}(\sd(P)^{\op},\G(\X))$.
    We will write $x_S$ for the value $D(S)$ at an object $S \in \sd(P)^{\op}$ and $x$ for the value at the tip. Our goal is to show that $D$
    is a colimit diagram if and only if all the edges $x_q \to x$ are $(\X_q,\X)$-Cartesian for $q \in P$.

    Let $D':(\sd(P)^{\op}_p)^{\triangleright} \to \X$ denote a left Kan extension of $D$ along the natural inclusion. Then $D$ is a colimit diagram if and only if $D'$ is a colimit diagram by composability of Kan extensions. Note further that the diagram
    \[\begin{tikzcd}
	M \\
	C & p
	\arrow[from=2-1, to=1-1]
	\arrow[from=2-1, to=2-2]
\end{tikzcd}\]
is cofinal in $\sd(P)^{\op}_p$, so $D'$ is a colimit diagram if and only if the square
\[\begin{tikzcd}
	{x_M} & x \\
	{x_C} & {x_p}
	\arrow[from=1-1, to=1-2]
	\arrow[from=2-1, to=1-1]
	\arrow[from=2-1, to=2-2]
	\arrow[from=2-2, to=1-2]
\end{tikzcd}\]
is biCartesian. Note further that $x_M, x_C \in \X_{P'}$. We now wish to apply Corollary \ref{cor:recoll_case} for the recollement
\[
    \begin{tikzcd}
        {\X_{P'}} & \X & {\X_{p}}
	\arrow[from=1-1, to=1-2]
	\arrow[shift left=3, from=1-2, to=1-1]
	\arrow[shift right=3, from=1-2, to=1-1,]
	\arrow[from=1-2, to=1-3]
	\arrow[shift left=3, from=1-3, to=1-2 ]
	\arrow[shift right=3, from=1-3, to=1-2 ]
\end{tikzcd}
\]
To this end, we first argue that the edge $x_C \to x_p$ is $(\X_{P'},\X)$-Cartesian (and therefore also $(\X_{P'},\X_p)$-Cartesian):
The subposet $(\sd(P)^{\op}_p)_{\leq C}$ of elements $\leq C$ is isomorphic as a marked poset to the poset $(\sd(P_{< p})^{\op})^{\triangleright}$ (by mapping $S \mapsto S \backslash \{p\}$.)
The restriction of $D'$ to this diagram is a colimit diagram, so, by induction hypothesis, all the edges
$x_{qp} \to x_C$ for $q < p$ are $(\X_q,\X)$-Cartesian. It follows that the edge $x_C \to x_p$ is $(\X_{P'},\X)$-Cartesian. Indeed,
let $y \to x_p$ be a $(\X_{P'},\X)$-Cartesian edge, then $y \in \X_{P_{<p}}$ since $x_p \in \X_p$.
Further, the induced edges $x_{pq} \to y$ are again $(\X_{\{q,p\}},\X)$-Cartesian, since the edges $x_{qp} \to x_p$ are.
By inductive hypothesis, $y$ is also a colimit of the restriction of $D$ to $(\sd(P)^{\op})_{<C}$, so the induced edge $x_C \to y$ is an equivalence.

Applying Corollary \ref{cor:recoll_case}, we conclude that $D'$ is a colimit diagram if and only if the edges $x_M \to x$ and $x_p \to x$ are $(\X_{P'},\X)$-Cartesian and $(\X_p,\X)$-Cartesian respectively.
Since $x_M$ is the colimit of the restriction of $D$ to $(\sd(P)^{\op})_{< M} \cong (\sd(P')^{\op})^{\triangleright}$, the inductive hypothesis implies that the edges
$x_q \to x_M$ are $(\X_q,\X)$-Cartesian. As above, this implies that the edge $x_M \to x$ is $(\X_{P'},\X)$-Cartesian if and only if all the edges $x_q \to x$ for $q \neq p$ are $(\X_q,\X)$-Cartesian.

In summary, we have shown that $D$ is a colimit diagram if and only if all the edges $x_q \to x$ for $q \in P$ are $(\X_q,\X)$-Cartesian, as claimed.
\end{proof}

\begin{rem}
    \label{rem:left} Let $\X$ be a stable $\infty$-category equipped with a
    left-admissible $P$-stratification $\{\X_p\}$. Then, applying Theorem
    \ref{thm:lax_lim} to $\{\X_p^{\op}\}$, we obtain that $\X$ is an {\em
    op}lax limit of the oplax gluing diagram 
    \[
        P^{\op} \to \Cat_{\infty},\; p \mapsto \X_p
    \]
\end{rem}

\subsection{Application: Cellular sheaves}%
\label{sub:cellular_sheaves}

Let $X$ be a locally compact Hausdorff space and $\D$ a stable
$\infty$-category. Assume that $X$ comes equipped {\em regular cell
decomposition} by which we mean a stratification $f: X \to Q$ such that, for
every $q \in Q$, the closure $\overline{X_q} \subset X$ in $X$ satisfies:
\begin{enumerate}
    \item \label{it:1} $\overline{X_q} = X_{\le q}$,
    \item \label{it:2} the inclusion $X_q \subset \overline{X_q}$ can be extended to a diagram
        \[
        \begin{tikzcd}
            X_q \ar[d, hookrightarrow] \ar{r}{\cong}& \mathring{B}^n \ar[d, hookrightarrow]\\
            \overline{X_q}  \ar[r, hookrightarrow]&  B^n
        \end{tikzcd}
        \]
        where $B^n$ denotes a closed ball of dimension $n$ and $X_q
        \cong \mathring{B}^n$ is a homeomorphism onto its
        interior.
\end{enumerate} 
In this case, we also refer to the strata $X_q$ as the {\em cells} of the cell
decomposition. Let $\Sh(X,\D)$ be the $\infty$-category of $\D$-valued sheaves on $X$. A sheaf
$F$ in $\Sh(X,\D)$ is called {\em $Q$-constructible}, or {\em cellular}, if its
restriction to each cell $X_q$, $q \in Q$, is constant, i.e. in the image of
the pullback 
\[
    \D \to \Sh(X_q,\D)
\]
along the constant map $X_q \to \pt$. We denote by 
\[
    \Sh_Q(X;\D) \subset \Sh(X;\D)
\]
the full subcategory spanned by the cellular sheaves.
Consider the inclusion $c_q: X_q \to X$, and set further
\[
    \X_q := (c_q)_* \Sh_Q(X_q,\D) \subset \Sh(X,\D),
\]
noting that we have an equivalence $\D \simeq \Sh_Q(X_q,\D)$ by \cite[A.2.1]{lurie:ha}.

\begin{prop}
    \label{prop:strat}
    \begin{enumerate}
        \item\label{strat:1} The family $\{\X_q\}_{q \in Q}$ forms a left-admissible stratification of the stable
    $\infty$-category $\X = \Sh_Q(X,\D)$. 
        \item\label{strat:2} The corresponding lax gluing diagram can be identified with the
            constant $Q$-diagram with value $\D$ so that we obtain an equivalence
            \[
                \Sh_Q(X,\D) \simeq \Fun(Q,\D).
            \]
    \end{enumerate}
\end{prop}
\begin{proof}
    We first verify that $\X_q$ is indeed contained in $\X$. To this
    end, let $F \in \X_q$ and note that we may factor the inclusion as $X_q
    \subset \overline{X_q} \subset X$. Since $\overline{X_q}$ is closed, the
    pushforward $(c_q)_*F$ has support on $\overline{X_q} = X_{\le q}$. It now
    follows from \ref{it:2} that $(c_q)_*F$ is constant (since the pushforward
    along $\mathring{B}^n \subset B^n$ preserves constant sheaves). In
    particular, the restriction to each cell $X_p$ with $p \le q$ is constant
    so that $(c_q)_*F \in \X$. 

    The coCartesian gluing functor for $p \le q$ is the composite $(c_p)^*
    (c_q)_*$ which fits into the diagram  
    \[
    \begin{tikzcd}
        \Sh_Q(X_q,\D) \ar{r}{(c_q)_*} & \Sh_Q(X_{\le q},\D) \ar{r}{c_p^*} & \Sh_Q(X_p,\D)\\
        \D \ar{u}{\cong} \ar{r}{\id}  & \D \ar[u] \ar{r}{\id} & \D \ar{u}{\cong} 
    \end{tikzcd}
    \]
    where all vertical equivalences are pullbacks along the constant map to
    $\pt$. It is straightforward to extend this diagram to a coherent
    identification of the lax gluing diagram 
    \[
        Q^{\op} \to \St, q \mapsto \X_q
    \]
    with the constant diagram with value $\D$. In particular, its oplax limit
    identifies with $\Fun(Q,\D)$ yielding \ref{strat:2} by Remark \ref{rem:left}.
\end{proof}

Given $X\to P$ a stratified space. We say $X \to P$ admits a \textit{refined cellular decomposition}, with cell poset $Q$, if the map $X \to P$ factors through $Q$, i.e there exists a dashed arrow making the triangle
\[
\begin{tikzcd}
	X \arrow[r] \arrow[d] & P \\
	Q \arrow[ru, dashed]  &  
\end{tikzcd}
\]
commute.
\begin{prop}\label{prop: localizedexodromy}
	Let $\D$ be a stable and presentable $\infty$-category. Given $X \to P$ a complex conically stratified manifold with complex strata $X_p$ together with a refined cellular decomposition with cell poset over $P$ given by $s:Q \to P$. Then
	\[
	\Sh_P(X) \simeq \Fun^P(Q,\D),
	\]
	where $\Fun^P(Q,\D) \subset \Fun(Q,\D)$ is the full subcategory of functors $F:Q \to \D$ sending $q \leq q'$ to an equivalence in $\D$ if $s(q) = s(q')$.
\end{prop}
\begin{proof}
	Notice that $\Fun^P(Q,\D) \simeq \Fun(Q[\W^{-1}],\D)$ for $\W$ the set of morphisms $q \leq q'$ in $Q$, satisfying $s(q) = s(q')$. Further $X$ is paracompact and has finite covering dimension. Then the statement follows from \S \ref{sub:comparison} and \ref{prop:stabilityundercoarsening}.
\end{proof}
\begin{rem}
	This should be valid in a much broader setting, then just complex manifolds. We stated it however in the context that is relevant for the purpose of this work.
\end{rem}

\subsection{Application: Foliated Verdier duality}%
\label{sec:combinatorial_foliated_verdier_duality}

As alluded to in the introduction, we will be interested in investigating
constructible (and ultimately perverse) sheaves on a complex manifold $X$ equipped with a 
$P$-stratification into complex submanifolds and, in addition, equipped with the
geometric auxiliary datum of a suitable foliation. 
Since, in this work, we are mainly interested in performing explicit
computations, instead of keeping track of the above data geometrically, we will
put ourselves into a combinatorial framework which is applicable in all
examples of interest. 

To this end, we assume that the space $X$ admits the structure of a regular
cell decomposition, with incidence poset $Q$, as discussed in \S
\ref{sub:cellular_sheaves}. We further assume that $Q$ a refinement of the
$P$-stratification of primary interest -- in other words, there is a map of posets
$s: Q \to P$. 
In this situation, by Proposition \ref{prop:strat}, we may first identify
$Q$-constructible sheaves on $X$ with cellular sheaves
\[
    \Sh_Q(X;\D) \simeq \Fun(Q,\D)
\]
and then observe that the category $\Sh_P(X;\D)$ can be described as the full
subcategory of $\Fun(Q,\D)$ consisting of those $Q$-diagrams that send every
arrow $q < q' \in Q$ with $s(q) = s(q')$ to an equivalence in $\D$. Thus, in what
follows, we focus on analyzing the category $\Fun(Q,\D)$, leaving the passage
to the subcategory $\Sh_P(X;\D)$ until the very end of the discussion.

The stratification corresponding to the foliation of $X$ will be kept track of
as a map $Q \to L$ where $L$ is a poset parametrizing the cellular
decomposition of the leaf space of the stratification. We denote the fiber of
this map over a given $l \in L$ by $Q_l$. The $\infty$-category $\Fun(Q_l,\D)$
will be implicitly identified with the full subcategory of $\Fun(Q,\D)$
consisting of diagrams with support on $Q_l$ (via the equivalence induced by
the restriction functor $\Fun(Q,\D) \to \Fun(Q_l,\D)$). 
Finally, we assume given a map of posets $\delta: L \to \NN$ keeping track of the
codimension of a leaf within its cell: if $l$ parametrizes a cell $X_l$ homeomorphic to $\RR^d \times
Z_l$ for a leaf $Z_l$, then we set $\delta(l) = d$. The equivalence 
\begin{equation}
    \label{eq:strat_to_leaf}
    \underline{\Gamma}_{Z_l}: \Sh_P(\RR^d \times Z_l;\D) \overset{\simeq}{\lra} \Sh_P(Z_l;\D)
\end{equation}
from sheaves on a ``leaf-stratum'' to sheaves on a leaf within the stratum can
then be modelled combinatorially by the shift functor 
\begin{equation}
    \label{eq:normalization}
    [-\delta(l)]: \Fun(Q_l,\D) \overset{\simeq}{\lra} \Fun(Q_l,\D).
\end{equation}

We set $\X := \Fun(Q,\D)$ and $\X_l := \Fun(Q_l,\D)$. We will implicitly
identify $\X_l$ with the full subcategory $\X_l \subset \X$ consisting of
$Q$-diagrams supported on $Q_l$, i.e. $F \in \X$ such that, for every $q \in
Q_l$, $F(x) \simeq 0$.  

\begin{prop}
    \label{prop:laxlim}
    The collection $\{\X_l\}_{l \in L}$ defines a right-admissible stratification of the stable
    $\infty$-category $\X$. 
\end{prop}
\begin{proof}
    Immediate.
\end{proof}

In particular, by Theorem \ref{thm:lax_lim}, we obtain a locally Cartesian
fibration $\pi: \G(\X) \to L^{\rhd}$ which exhibits 
\[
    \X \simeq \llim_{l \in L} \X_l  
\]
as a lax limit of the {\em lax gluing diagram} corresponding to $\pi|L$. 
For $l < l' \in L$, we denote by 
\begin{equation}\label{eq:gluingdef}
    \gamma_{l,l'}: \X_{l'} \to \X_l
\end{equation}
the {\em (generalized) gluing functor} associated to the Cartesian fibration
$\pi|\{l<l'\}$ over the edge $l < l'$. 
Further, for $l \in L$, we denote by 
\[
    \nu_l: \X \to \X_l
\]
the {\em projection functor} determined by the restriction $\pi|\{l<*\}$ over
the edge $l < *$. 

\begin{prop}
    \label{prop:nuformula}
    Let $l \in L$. The projection functor $\nu_l$ is determined, for $q \in
    Q_l$, by the pointwise formulas
    \begin{equation}\label{eq:nu}
        \nu_l(F)(q) \simeq \tfib(F|(\{q\} \cup q/Q_{>l})).
    \end{equation}
\end{prop}
\begin{proof}
    The projection functor $\nu_l: \X \to \X_l$ is, by definition, right
    adjoint to the inclusion $\X_l \subset \X$ as the full subcategory of
    $Q$-digagrams supported on $Q_l$, which can be further factored as $\X_l
    \subset \X_{\ge l} \subset \X$. The inclusion $\X_{\ge l} \subset \X$ participates as $j_!$
    in the recollement
    \[
        \begin{tikzcd}
            \X_{\ngeq l} \ar{r} &\ar[shift left=1em,swap]{l}
            \ar[shift right=1em, swap]{l} \X \ar["j^*" description]{r} & \ar[shift
            left=1em,swap]{l} \ar[shift right=1em, "j_!" description]{l} \X_{\ge
            l}.
        \end{tikzcd}
    \]
    Its right adjoint $j^*$ is given by restriction along $Q_{\ge l} \subset
    Q$. Further, the inclusion $\X_l \subset \X_{\ge l}$ participates as $i_*$ in the
    recollement
    \[
        \begin{tikzcd}
            \X_{l} \ar["i_*" description]{r} &\ar[shift left=1em, "i^!" description]{l}
            \ar[shift right=1em, swap]{l} \X_{\ge l} \ar["j^*" description]{r} & \ar[shift
            left=1em, "j_*" description]{l} \ar[shift right=1em]{l} \X_{> l}.
        \end{tikzcd}
    \]
    Due to the triangle \eqref{eq:exactcartesian} for the SOD $(\X_{>l}',
    \X_l)$ (which corresponds classically to the distinguished triangle
    \ref{ext:2}), we obtain the formula
    \[
        i_*i^! = \fib( \id \to j_* j^*).
    \]
    The functor $j_*$ is, as a right adjoint to the restriction functor $j^*$,
    given by right Kan extension along $Q_{>l} \subset Q$, so that we obtain,
    for $q \in Q_l$, 
    \[
        \nu(F)(q) \simeq \fib(F(q) \to \lim F|(q/Q_{>l})) = \tfib(F|(\{q\} \cup q/Q_{>l}))
    \]
    by the pointwise formula for Kan extensions.
\end{proof}

\begin{rem}
    \label{rem:gluing}
    Note that, for $l < l'$, the gluing functor
    \[
        \gamma_{l,l'}: \X_{l'} \to \X_{l}
    \]
    from \eqref{eq:gluingdef} is given by the restriction
    \begin{equation}\label{eq:resglue}
        \gamma_{l,l'} = \nu_l|\X_{l'}
    \end{equation}
    so that the formula \eqref{eq:nu} for $\nu_l$ also determines a formula for
    $\gamma_{l,l'}$. 
\end{rem}

\begin{rem}
    \label{rem:normalization}
    As explained above, to match the combinatorical setup with the geometric
    situation where we take sections with support on a single leaf as opposed to
    sections with support on the cell within which the leaf is locally constant, we need
    to normalize the projection functor $\nu_l$ by the shift from
    \eqref{eq:normalization}. To incorporate this, we also introduce the {\em
    normalized projection functors}
    \begin{equation}
        \label{eq:normnu}
        \widetilde{\nu}_l = [-\delta(l)] \circ \nu_l
    \end{equation}
    along with the {\em normalized gluing functors}
    \begin{equation}
        \label{eq:normgamma}
        \widetilde{\gamma}_{l,l'} = [\delta(l')-\delta(l)]\gamma_{l,l'}.
    \end{equation}
    The normalized projection and gluing functors are part of a normalized
    variant of the lax gluing diagram along with the lax cone that exhibits
    $\X$ as its lax limit. We refer to the corresponding equivalence
    \begin{equation}
        \label{eq:folver}
        \{\widetilde{\nu}_l\}_{l \in \L}: \X \simeq \llim_{l \in L} \X_l
    \end{equation}
    as {\em (cellular) foliated Verdier duality}.
\end{rem} 

\subsection{Transverse foliations}
\label{sec:universal}

In this section, we supplement the combinatorial setup of \S
\ref{sec:combinatorial_foliated_verdier_duality} introducing a combinatorial
avatar of a ``transverse'' foliation. This leads to particularly nice formulas for the foliated Verdier duality \eqref{eq:folver}. Given maps of posets,
\[
l:Q \to L, \quad \delta: L \to \NN.
\] 

We introduce the combinatorial framework to capture such transverse foliations:
\begin{enumerate}
    \item\label{tfitem1} The stratification corresponding to the transverse foliation will be
        encoded by another map of posets $Q \to T$. The fiber of this map over
        $t \in T$ will be denoted $Q^t$. 
    \item We assume that the resulting map $\pi: Q \to L \times T$ is a
        Cartesian fibration and write
        \[
            \pi(q) = (l(q),t(q))
        \]
        for the $L$- and $T$-components of $\pi(q)$, respectively.
    \item The fibers of $\pi$ are discrete. We denote the fiber of $\pi$ over
        $(l,t) \in L \times T$ by $Q^t_l$.  
    \item\label{tfitem4} For $l<l' \in L$, the poset $l/L/l'$ is a cube of dimension
        $\delta(l') - \delta(l)$.
\end{enumerate}
\begin{lem}
    \label{lem:cubical}
    Let $\pi: Q \to L \times T$ be as above and $t \in T$. Then for every $q <
    q' \in Q^t$, the poset $q/Q^t/q'$ is a cube of dimension $\delta(l(q')) -
    \delta(l(q))$.
\end{lem}
\begin{proof}
    This is clear since, because $\pi$ is Cartesian with discrete fibers, it
    induces an identification of posets $q/Q^t/q' \to l(q)/L/l(q')$ and the
    latter is a cube by assumption.
\end{proof}

\begin{thm} 
    \label{thm:conclusion}
    Let $\pi: Q \to L \times T$ be as above and let $\D$ be a stable
    $\infty$-category. Then the following hold:
    \begin{enumerate}
        \item\label{it:c1} The normalized diagram $L^\op \to \Cat_{\infty}, \; l \mapsto
            \Fun(Q_l, \D)$ of gluing functors commutes (coherently).
        \item\label{it:c2} The projection functors $\nu_l: \Fun(Q,\D) \to \Fun(Q_l,\D)$ are
            given by the pointwise formulas
            \begin{equation}\label{eq:projectioncommon}
                \nu_l(F)(q) = \totfib(F|q/Q^{t(q)}).
            \end{equation}
        \item\label{it:c3} For $l < l'$ and $F \in \Fun(Q,\D)$ supported on $Q_{l'}$, the
            formula \eqref{eq:projectioncommon} specializes to 
            \begin{equation}\label{eq:gluingcommon}
                \widetilde{\nu}_l(F)(q) \simeq \bigoplus_{q' \in q/Q^{t(q)}_{l'}} F(q').
            \end{equation}
    \end{enumerate}
\end{thm}
\begin{proof}
    We introduce the notation $\X := \Fun(Q,\D)$ and, for $l \in L$, $\X_l := \Fun(Q_l,\D)$.
    According to Proposition \ref{prop:nuformula}, the projection functors
    $\nu_l: \X \to \X_l$ admit the pointwise formula
    \begin{equation}\label{eq:pointwiseproj}
        \nu_l(F)(q) = \fib(F(q) \to \lim F|q/Q_{>l}).
    \end{equation}
    The assumption that $\pi$ is a Cartesian fibration implies that the inclusion of posets
    \begin{equation}\label{eq:coinitialproj}
        q/Q_{>l} \subset q/Q^{t(q)}_{>l}
    \end{equation}
    is coinitial which implies formula \eqref{eq:projectioncommon} and hence
    \ref{it:c2}. To show \ref{it:c1}, it suffices to verify that locally
    Cartesian edges in $\X$ are closed under composition. To this end, we
    investigate the characterizing properties of locally Cartesian edges in
    $\X$. For $l < l'$ in $L$, given an edge $\eta: F \to F'$ in $\X$ with $F
    \in X_l$ and $F' \in \X_{l'}$, we may interpret $\eta$ as a map $[1] \times
    Q \to \D$. For $q \in Q$, we restrict $\eta$ to $D_q:= [1] \times (q/Q_{>l} \cup
    \{q\})$ to obtain a cone shaped diagram with initial vertex $(0,q) \mapsto
    F(q)$. Unravelling the definitions, the pointwise criterion
    \eqref{eq:pointwiseproj} says more precisely, that $\eta$ is a locally
    Cartesian edge (modelling the functor $\nu_l|\X_{l'}$) if and only if, for
    every $q \in Q_l$, the corresponding $D_q$ diagram is a limit diagram.
    Applying a straightforward ``decomposition of limits'' argument, we may
    further reformulate this condition as stating that the canonical square 
    \begin{equation}\label{eq:characterization}
        \begin{tikzcd}
            F(q) \ar{r}\ar{d} &\ar{d} 0\\
            0 \ar{r} & \lim F' | q/Q_{>l}
        \end{tikzcd}
    \end{equation}
    exhibited by the $D_q$-diagram corresponding to $\eta$ is a pullback
    square (which now explicitly amounts to the pointwise formula for $F'(q) =
    0$). We may further simplify by restricting along the coinitial inclusion
    \eqref{eq:coinitialproj} to obtain the equivalent condition that the diagram 
    \[
    \begin{tikzcd}
        F(q) \ar{r}\ar{d} &\ar{d} 0\\
        0 \ar{r} & \lim F' | q/Q^{t(q)}_{>l}
    \end{tikzcd}
    \]
    be a pullback diagram or, put equivalently, the restriction of $\eta$ to
    $[1] \times Q^t$ defines a locally Cartesian edge in $\Fun(Q^t,\D)$.
    Therefore, in order to verify that locally Cartesian edges are closed under
    composition, we may, for every $t \in T$, verify the same condition for
    $Q$ replaced by the transversal slice $Q^t$. 

    In other words, we may assume that $Q \to L$ is a Cartesian
    fibration with discrete fibers so that the category $\X_{l'}$ is equivalent
    to a finite direct sum of copies of $\D$. Since the functor $\nu_l|\X_{l'}$
    commutes with finite sums, we may assume that $F': Q_{l'} \to \D$ is
    supported on a single vertex $q' \in Q_{l'}$. Under this assumption, we may
    further simplify the characterizing condition \eqref{eq:characterization}
    via the equivalence
    \[
        \lim F' | q/Q_{>l} \simeq \lim F' | (q/Q/q' \setminus \{q\})
    \]
    using that $F'|q/Q_{>l}$ is a right Kan extension of its restriction to
    $q/Q/q'$. Due to Lemma \ref{lem:cubical} the resulting simplification of
    \eqref{eq:characterization} reduces the proof of \ref{it:c1} to Proposition
    \ref{prop:stackcubes}.
\end{proof}

\begin{prop}
    \label{prop:stackcubes}
    Let $C$ denote the poset of subsets of $\{1,...,n\}$ and consider the map
    of posets $\delta: C \to \NN$ given by the cardinality. Let $\D$ be a stable
    $\infty$-category and $\X = \Fun(C,\D)$ equipped with the stratification
    $\{\X_{I}\}_{I \in C}$ where $\X_I = \Fun(\{I\},\D) \simeq \D$. Then the associated
    gluing diagram
    \[
        C^{op} \to \Cat_{\infty}, \; I \mapsto \D
    \]
    is strictly commutative with transition functor corresponding to $I<J$
    given by $(\delta(J)-\delta(I))$-fold looping. In particular, the normalized
    gluing diagram is equivalent to the constant diagram with value $\D$.
\end{prop}
\begin{proof}
    For $I < J$ in $C$ an edge in $\X$ from $F \in \X_I \simeq \D$ to $G \in
    \X_J \simeq \D$ corresponds to a (cubical) diagram
    \begin{equation}\label{eq:cubeedge}
        \eta: \Delta^1 \times C \to \D
    \end{equation}
    with only nonzero vertices $\eta(0,I) = F$ and $\eta(1,J) = G$. The edge is
    locally Cartesian if and only if the restriction of \eqref{eq:cubeedge} to
    $\Delta^1 \times I/C/J$ is a biCartesian cube, which hence exhibits $F$ as
    the $(\delta(J)-\delta(I))$-fold looping of $G$. 

    Now given $I < J < K$ in $C$ along with $F \in X_I$, $G \in \X_J$, and $H
    \in \X_K$, then a composite of edges $F \to G$ and $G \to H$
    in $\X$ can be identified with a diagram
    \begin{equation}\label{eq:cubecomp}
        \chi: \Delta^2 \times C \to \D
    \end{equation}
    with only nonzero vertices $\chi(0,I) = F$, $\chi(1,J) = G$, and $\chi(2,K)
    = H$. The edge $F \to G$  (resp. $G \to H$) being locally Cartesian amounts to the cube $q$ (resp. $q'$)
    obtained by restricting to $\{0,1\} \times I/C/J$ (resp. $\{1,2\} \times
    J/C/K$) being biCartesian. The composite edge being locally Cartesian
    corresponds to the restriction $q''$ of \eqref{eq:cubecomp} to $\{0,2\}
    \times I/C/K$ being biCartesian. Now we claim that, if $F \to G$ and $G \to
    H$ are both locally Cartesian, then in fact the restrictions
    $\widetilde{q}$ (resp. $\widetilde{q'}$) to $\{0,1\} \times I/C/K$  (resp.
    $\{1,2\} \times I/C/K$) are biCartesian, so that the ``stacked'' cube $q''$
    is biCartesian as well by, e.g. \cite[Corollary A.18]{djw:auslander}. 

    The claim that the restriction to $\{0,1\} \times I/C/K$ is biCartesian
    follows, since it can be interpreted as a $J/C/K$-cube of biCartesian
    $\{0,1\} \times I/C/J$-cubes: one of these biCartesian cubes is the cube
    $q$ and the remaining ones consist of zero objects. Thus, realizing the
    totalization of $\widetilde{q}$ as an iterated totalization (cf. \cite[Corollary
    A.13]{djw:auslander}) yields the desired result. The argument for
    $\widetilde{q'}$ is analogous.
\end{proof}

\section{Foliations and $t$-exactness}%
\label{sec:foliations_and_purity}

In \S \ref{sec:combinatorial_foliated_verdier_duality}, we have introduced a
cellular version of foliated Verdier duality which we may think of an
equivalence expressing a sheaf on a space $X$ in terms of the family of
sheaves $\{i^!_{\Lambda}(F)\}_{\Lambda \subset X}$ parameterized by the leaves
$i_{\Lambda}: \Lambda \hra X$ of a foliation.
Now let $\D = \D(\A)$ be the derived category of a Grothendieck abelian
category $\A$, equipped with its standard $t$-structure (cf. \S
\ref{sub:tstructures}). In this section we will show that, if the foliation is
totally real and admits a transverse foliation, then the functors
$i^!_{\Lambda}$ are exact with respect to the perverse $t$-structure on
$\Sh_P(X;\D)$ and the standard $t$-structures on $\Sh_P(\Lambda;\D)$.

Even though we will not make use of it later, we first discuss this
$t$-exactness in a geometric context in \S \ref{sub:geometric}, since it helps
motivate our cellular setup, designed to combinatorialize its geometric
counterpart, which will be the subject of \S \ref{sub:cellular}.

\subsection{Geometric}%
\label{sub:geometric}

Let $P$ be a finite poset, and let $X$ be a complex analytic space equipped
with a $P$-stratification, such that the strata $X_p$, $p \in P$, are complex
manifolds. 

\begin{rem}
    \label{rem:clumsy}
    To avoid clumsy notation, we use the convention to write $(D^{\ge 0},
    D^{\le 0})$ to denote a $t$-structure on a given stable $\infty$-category
    $\A$ with a $t$-structure, leaving the reference to $\A$ implicit: For
    example, for $F \in \Sh_{P}(X;\D)$ we will write $F \in \D^{\ge 0}$ to
    indicate that $F \in  \Sh_{P}(X;\D)^{\ge 0}$. Perverse $t$-structures will
    be referred to as $(\Dp^{\ge 0},\Dp^{\le 0})$ where the standard convention
    is to use the middle perversity function.
\end{rem}

\begin{prop}[Left exactness]
    \label{prop:leftexact}
    Let $i_{\Lambda}: \Lambda \subset X$ be a closed subspace and $F \in \Sh_P(X;\D)$. Assume that
    \begin{enumerate}[label=(R)]
        \item\label{it:real} for every $p \in P$, the subspace
            \[
                \Lambda_p := \Lambda \cap X_p \subset X_p
            \]
            is a smooth real submanifold with $\dim_{\RR}(\Lambda_p) \le \dim_{\CC}(X_p)$.
    \end{enumerate}
    Then the functor
    \[
        i^!_{\Lambda}:  \Sh_P(X;\D) \to \Sh_P(\Lambda;\D)
    \]
    is left exact with respect to the perverse $t$-structure on $\Sh_P(X;\D)$
    and the standard $t$-structure on $\Sh_P(\Lambda;\D)$.
\end{prop}
\begin{proof}
    Let $F \in \Sh_P(X;\D)$ with $F \in \Dp^{\ge 0}$. We need to show
    $i^!_{\Lambda}F \in \D^{\ge 0}$. We prove the statement by induction on the
    depth of elements in the poset $P$: For a given $p \in P$, suppose we
    already have 
    \[
        i^!_{\Lambda} F | \Lambda_{>p} \in D^{\ge 0}.
    \] 
    where $\Lambda_{>p} = X_{>p} \cap \Lambda$. We wish to show that 
    \[
        i^!_{\Lambda_{\ge p}} F = i^!_{\Lambda} F | \Lambda_{\ge p} \in D^{\ge 0}.
    \] 
    Since the standard $t$-structure on $\Sh_P(\Lambda_{\ge p};\D)$ is obtained
    by recollement with respect to $\Lambda_{\ge p} = \Lambda_{>p} \cup
    \Lambda_p$ this is equivalent to showing
    \begin{enumerate}
        \item $i^!_{\Lambda_p} F \in D^{\ge 0}$, and
        \item $i^!_{\Lambda} F | \Lambda_{> p} \in D^{\ge 0}$.
    \end{enumerate}
    The second statement is true by induction hypothesis. The first statement
    follows, since we have 
    \[
        \Lambda_p \overset{l}{\subset} X_p \overset{k}{\subset} X
    \]
    so that 
    \[  
        i^!_{\Lambda_p} F = l^!k^!F
    \]
    Now since $F \in \Dp^{\ge 0}$, we have $k^!F \in D^{\ge - \dim_{\CC}(X_p)}$.
    Further, since the codimension of $\Lambda_p \subset X_p$ is least
    $\dim_{\CC}(X_p)$, we have $l^!k^!F \in \Dp^{\ge 0}$, by Lemma
    \ref{lem:shrieksub}. Note that we may begin the induction by applying the
    induction step at a maximal element $p$ of $P$ where $\Lambda_{>p} =
    \emptyset$.
\end{proof}

\begin{lem}
    \label{lem:shrieksub}
    Let $i: \RR^m \subset \RR^n$ denote the inclusion via the first $m$
    coordinates and let $F \in \Sh(\RR^n,\D)$ be locally constant. Then, for
    every $k \in \ZZ$ and $x \in \RR^m$, we have a canonical isomorphism
    \[
        H^k(i^!F)_x \cong H^{k+m-n}(F)_x.
    \]
\end{lem}
\begin{proof} By induction on the codimension, we may assume $n = m+1$. We
    denote by $j: \RR^m \setminus \RR^{m+1}$ the inclusion of the open
    complement of $\RR^m$. For every $x \in \RR^m$, we consider the exact
    triangle
    \[
        (i_*i^!F)_x \to F_x \to (j_* j^*F)_x.
    \]
    where the map $H^k(F_x) \to H^k(j_* j^*F)_x$ can be identified with a
    diagonal map $E \to E \oplus E$. The claim follows.
\end{proof}

\begin{prop}[Right exactness]
    \label{prop:rightexact}
    Let $i_{\Lambda}:\Lambda \subset X$ be a closed subset with complement
    $j: X \setminus \Lambda \subset X$ and assume that
    \begin{enumerate}[label=(T)]
        \item\label{it:transverse} for every $G \in \Sh_P(X;\D)$ with $G \in \Dp^{\le -1}$, we have
            $i_{\Lambda}^*j_*j^*G \in \D^{\le -1}$. 
    \end{enumerate}
    Then the functor
    \[
        i^!_{\Lambda}: \Sh_P(X;\D) \lra \Sh_P(\Lambda;\D)
    \]
    is right exact with respect to the perverse $t$-structure on $\Sh_P(X;\D)$
    and the standard $t$-structure on $\Sh_P(\Lambda;\D)$.
\end{prop}
\begin{proof}
    Let $i: \Lambda \subset X$ and $j: X \setminus \Lambda \subset X$. Consider
    the exact triangle
    \[
        i_*i^! F \to F \to j_* j^* F
    \]
    For $x \in \Lambda \cap X_p$, we have $F_x \in D^{\le - \dim(X_p)} \subset
    D^{\le 0}$ so, to guarantee that $(i_*i^! F)_x \in D^{\le 0}$, it suffices
    to show $(j_*j^*F)_x \in D^{\le -1}$, but this is true by assumption
    \ref{it:transverse}. 
\end{proof}

\begin{cor}
    \label{cor:texact}
    Under the hypothesis of Proposition \ref{prop:rightexact}, the functor
    \[
        i^!_{\Lambda}: \Sh_P(X;\D) \to \Sh_P(\Lambda;\D) 
    \]
    is $t$-exact with respect to the perverse $t$-structure, respectively standard 
    $t$-structure.
\end{cor}
\begin{proof}
    Proposition \ref{prop:rightexact} and Proposition \ref{prop:leftexact}.
\end{proof}

The hypothesis \ref{it:transverse} can be viewed as some (cohomological)
tubular regularity of the subspace $\Lambda \subset X$. There may be different
geometric setups in which this regularity holds: as already alluded to above,
one natural way to guarantee \ref{it:transverse} along with the hypothesis
\ref{it:real} of Proposition \ref{prop:leftexact}, is to assume that $\Lambda$
is the leaf of a totally real foliation which admits a ``transverse'' foliation
in the following sense:

\begin{defi}
    \label{defi:foliated}
    \begin{enumerate}
        \item By a {\em totally real foliation of $X$}, we mean a continuous map $f: X \to R$
        such that, for every $r \in R$ and $p \in P$, the subspace
        \[
            X_{r,p} = f^{-1}(r) \cap X_p \subset X_p
        \]
        is a totally real submanifold, and such that $\Lambda = f^{-1}(r)$ is locally
        conical in the following sense:
        \begin{itemize}
        \item for every $r \in R$, $p \in P$, and every $x \in X_{r,p}$, there is an open
            $P_{\le p}$-stratified embedding $\RR^k \times C(M)$ where $M$ is a
            compact $P_{> p}$-stratified topological space whose strata $M_q$, $q
            \in P_{>p}$ are smooth real manifolds of real dimension
            $\dim_{\CC}(X_q) - 1$.
        \end{itemize}
        The fibers of $f$ are called the {\em leaves} of the foliation. 
        \item A pair of foliations $f: X \to R$ and $g: X \to S$ is called {\em
            transverse} if, for every $x \in X$, the leaves $A = f^{-1}(f(x))$ and
            $B = g^{-1}(g(x))$ satisfy: there exists a neighbourhood $x \in U
            \subset X$ such that 
        \begin{enumerate}[label=(T\arabic*)]
            \item\label{trans:0} $A \cap B \cap U = \{x \}$,
            \item\label{trans:1} $B \cap U \setminus (A \cap U) \subset U \setminus (A \cap U)$ is a
                $P$-stratified homotopy equivalence 
            \item\label{trans:2} $A \cap V \setminus (B \cap V) \subset V \setminus (B \cap V)$ is a
                $P$-stratified homotopy equivalence.
        \end{enumerate}
        \end{enumerate}
\end{defi}

\begin{lem}
    \label{lem:trans}
    Let $\Lambda \subset X$ be a leaf of a totally real foliation which admits
    a transverse foliation. Then the \ref{it:transverse} holds.
\end{lem}
\begin{proof}
    Let $G \in \Sh_P(X;\D)$ with $G \in \Dp^{\le -1}$.
    Due to \ref{trans:1} we have an equivalence
    \[
        (j_*j^*G)_x \simeq (k_*k^*G)_x
    \]
    where $k:\Lambda^{\perp} \setminus \{x\} \subset X$ is the punctured
    transverse leaf. Since the transverse leaf intersects each stratum $X_p$ in
    a totally real submanifold, in particular of real dimension
    $\dim_{\CC}(X_p)$, we may apply Lemma \ref{lem:stratcone} to obtain
    that, for every open neighborhood $x \in U \subset X$, we have
    $(k_*k^*G)(U) \in D^{\le -1}$ and thus $(k_*k^*G)_x\in D^{\le -1}$.
\end{proof}

\begin{lem}
    \label{lem:stratcone} 
    Let $A$ be a finite poset, $M$ an $A$-stratified space whose
    strata are smooth manifolds, and consider, for $k \ge 0$, the
    $A^{\lhd}$-stratified space $\Lambda = \RR^k \times C(M)$ where $C(M)$
    denotes the open cone over $M$. Denote further $k: \Lambda \setminus
    \{(0,*)\} \hra \Lambda$. Let $F \in \Sh_{\A^{\lhd}}(\Lambda,\D)$ and
    assume that
    \begin{itemize}
        \item for every $a \in A$, we have $F|\Lambda_a \in D^{\le - \dim(\RR^{k+1} \times
            M_a)}$.
    \end{itemize}
    Then we have 
    \[
        k_*k^*F \in D^{\le -1}.
    \]
\end{lem}
\begin{proof}
    The stalk of $k_*k^*F$ can be identified with the global sections of the
    restriction of $F$ to either $\RR^{k} \setminus \{0\}$ (for $M =
    \emptyset$) or $\RR^{k+1} \times M$ (for $M$ nonempty). The statement now
    follows from Lemma \ref{lem:stratbounded} for $r = -1$. 
\end{proof}

\begin{lem}
    \label{lem:stratbounded}
    Let $M$ be a $A$-stratified space and $F \in \Sh_A(M;\D)$ where $\D$ is a
    stable $\infty$-category with $t$-structure. Let $r \in \ZZ$ and assume
    that, for every $a \in A$:
    \begin{enumerate}
        \item the stratum $M_a$ is a smooth real manifold, 
        \item $F|M_a \in D^{\le - \dim(M_a) + r}$.
    \end{enumerate}
    Then $\Gamma(M,F) \in D^{\le r}$. 
\end{lem}

\begin{cor}
	\label{cor:manifold}
	Let $M$ be a smooth real manifold of dimension $n$ and let $F \in \Sh(M;\D)$ be a locally
	constant sheaf with $F \in D^{\le -n}$. Then $\Gamma(M,F) \in D^{\le 0}$. 
\end{cor}
\begin{proof}
	Viewing a triangulation of the manifold $M$ as a regular cell decomposition, we may apply  
	Proposition \ref{prop:strat} so that the claim is a special case of Proposition \ref{prop:Qestimate}.
\end{proof}

\subsection{Cellular}%
\label{sub:cellular}

In this section we translate the geometric setting of \ref{sub:geometric} into the cellular, in order to prove $t$-exactness for the normalized projection functors
\[
\widetilde{\nu}_l: \Fun^P(Q,\D) \to \Fun(Q_l,\D).
\] 
Therefore we enhance the setup of \ref{sec:universal}, asking for the existence of a collection of maps of posets:
\[
s:Q \to P, \quad \pi:Q \to L\times T, \quad \p:P \to \ZZ, \quad \delta: L \to \NN,
\] 
being subject the additional constraints:
\begin{enumerate}[start= 5]
	\item\label{item1: regassum} The two maps $\delta,\p$ satisfy the relation 
	\[
	\p(s(q)) + \delta(l(q)) \leq 0
	\]
	for all $q \in Q$.
	\item\label{item2: regassum} For $q \in Q_l$ and $p = s(q)$, each slice poset $q/Q_p^{t(q)}\setminus \{q\}$ (this models the link around $q$ inside $Q_p$) has homotopy type of an $n$-dimensional sphere where $n \geq -(\p(p) + \delta(l)) -1$. (compare this to \ref{lem:shrieksub})
\end{enumerate}

Assume for the following that \ref{tfitem1}-\ref{tfitem4} of \ref{sec:universal} hold. We indicate in the proceeding, which additional assumption of \ref{item1: regassum} and \ref{item2: regassum} need to be satisfied in order to obtain (left/right) $t$-exactness of $\widetilde{\nu}_l$. 

\begin{rem}
	Recall, $\Fun^P(Q,\D) \subset \Fun(Q,\D)$ consists of those functor $F: Q \to \D$, such that $q \leq q'$ in $Q$ is send to an equivalence if $s(q) = s(q')$.
	
	The perverse $t$-structure on $\Fun(Q,\D)$ is given as follows
	\begin{enumerate}
		\item[$(P^+)$] $F \in {}^\p D^{\geq 0}$ if and only if $i_p^!F \in D^{\geq \p(p)}$ for all $p \in P$.
		\item[$(P^-)$] $F \in {}^\p D^{\leq 0}$ if and only if $i_p^\ast F \in D^{\leq \p(p)}$ for all $p \in P$.	
	\end{enumerate}	
	We obtain an induced $t$-structure on $\Fun^P(Q,\D) \subset \Fun(Q,\D)$. (It is straightforward to see that the truncation functors commute with the inclusion)
\end{rem}

\begin{prop}\label{defi:pfoliation}
	Let $\pi:Q \to L \times T $ be a foliation satisfying \ref{tfitem1}-\ref{tfitem4} of \ref{sec:universal} and \ref{item2: regassum}. Then
	\begin{enumerate}
		\item[$(\ast)$]\label{defi:pfoliation:item1} for each $l \in L$ and $p \in P$ we have $\widetilde{\nu}_l|_{Q_p}: \Fun(Q_p,\D^\simeq) \to \Fun(Q_{(l,p)},\D)$ is left $t$-exact for all $p \in P$.
	\end{enumerate}
\end{prop}
\begin{proof}
	By \ref{thm:conclusion} we have the formula
	\[
	\widetilde{\nu}_l(F)(q) = \totfib(F|q/Q^{t(q)}),
	\]
	since $F$ has support on $Q_p$, we obtain 
	\[
	\widetilde{\nu}_l(F)(q) = \totfib(F|q/Q_p^{t(q)}),
	\]
	by assumption $q/Q_p^{t(q)} \setminus \{q\}$ has the homotopy type of a sphere of dimension $n \geq -\delta(l) -\p(p)- 1$, since $F$ has values in $\D^\simeq$, we compute
	\[
	\totfib(F|q/Q_p^{t(q)}) \simeq \fib(F(q) \to \lim_{S^n}F(q)) \simeq F(q)[-(n+1)]
	\]
	which shows that if $F \in D^{\geq \p(p)}$ the normalized projection functor shifts it into degree $\geq 0$.
\end{proof}

Let $Q$ be a finite poset and let $d: Q \to \NN$ be a map of posets with
discrete fibers. Let $\D$ be a stable $\infty$-category with $t$-structure. 

\begin{prop}
    \label{prop:Qestimate}
    Suppose that $F: Q \to \D$ is a diagram and there exists $r \in \ZZ$ such that, for $n \in \NN$, we have
    $F|Q_n \in \D^{\le -n + r}$. 
    \begin{enumerate}
        \item\label{con:1} Then $\lim F \in \D^{\le r}$. 
        \item\label{con:2} Suppose that $Q_0 = \{x\}$ where $x \in Q$ is initial. Then $\totfib F \in \D^{\le 0}$.
    \end{enumerate}
\end{prop}
\begin{proof}
    Statement \ref{con:2} follows immediately from the defining
    exact triangle
    \[
        \tfib F \to F(x) \to \lim F|Q \setminus \{x\}
    \]
    and observing that, by \ref{con:1} with $r = -1$, we have $\lim F|Q
    \setminus \{x\} \in \D^{\le -1}$. We prove \ref{con:1}. Without loss of
    generality we may assume $r = 0$.
    Let $m \in \NN$ be maximal such that $Q_m$ is nonempty. We argue by
    induction on $m$. Denote by $j: Q_m \subset Q$ and $i: Q_{<m} \subset Q$
    the inclusions. We have an exact triangle
    \[
        i_*i^! F \to F \to j_*j^* F.
    \]
    to which we apply $\lim$ to obtain an exact triangle
    \begin{equation}
        \label{eq:limtriangle}
        \lim(i^! F) \to \lim(F) \to \lim(j^* F)
    \end{equation}
    in $\D$. Since $Q_m$ is discrete, we have $\lim(j^* F) \in \D^{\le
    -m}$. We claim that $i^! F: Q_{<m} \to \D$ satisfies the hypothesis of the
    proposition so that, by induction on $m$, we have $\lim(i^! F) \in \D^{\le
    0}$. This will imply $\lim(F) \in \D^{\le 0}$ by \eqref{eq:limtriangle}.
    To show the claim, suppose that $q \in Q_{<m}$ and consider the exact triangle
    \begin{equation}
        \label{eq:evaluatetriangle}
        i_*i^! F(q) \to F(q) \to j_*j^* F(q)
    \end{equation}
    obtained by evaluating at $q$. We have $F(q) \in \D^{ \le -d(q)}$ and, by the pointwise formula, 
    \[
        j_*j^* F(q) \simeq \bigoplus_{\substack{q' \in Q_m \\ q < q'}} F(q') \in \D^{\le -m}.
    \]
    Therefore, by \eqref{eq:evaluatetriangle}, we deduce $i^! F(q) = i_*i^! F(q) \in \D^{\le -d(q)}$. 
\end{proof}

\begin{prop}[Left $t$-exactness]
	Let $Q \to L \times T$ satisfy \ref{tfitem1}-\ref{tfitem4} of \ref{sec:universal} and \ref{item2: regassum}. Given $l \in L$ then the restricted normalized projection functor
	\[
	\widetilde{\nu}_l|_{\Fun^P(Q,\D)}: \Fun^P(Q,\D) \to \Fun(Q_l,\D)
	\]
	is left $t$-exact, with respect to the perverse $t$-structure on $\Fun^P(Q,\D)$ and the standard $t$-structure on $\Fun^P(Q_l,\D)$.
\end{prop}
\begin{proof}
		We proceed by induction: As for the induction start we notice that for $p \in P$ maximal, by \ref{defi:pfoliation} we have that
	\[
	\widetilde{\nu}_l|_{Q_p}:\Fun(Q_p,\D^\simeq) \to \Fun(Q_{(l,p)},\D)
	\]
	is left $t$-exact by \ref{defi:pfoliation}. Let $p \in P$ be arbitrary and assume that
	\[
	\widetilde{\nu}_l|_{Q_{(>p)}}: \Fun^P(Q_{(>p)},\D) \to \Fun(Q_{(l,(>p))},\D)
	\]
	is already left $t$-exact. Let $F \in {}^\p D^{\geq 0}$, by recollement gluing with respect to the open closed decomposition $Q_{(\geq p)} =  Q_{(>p)} \cup Q_p$, this is equivalent to $i_p^!F \in D^{\geq \p(p)}$ and  $F|_{Q_{(>p)}} \in {}^\p D^{\geq 0}$. To prove left $t$-exactness of $\widetilde{\nu}_l|_{Q_{(\geq p)}}$ we again argue by decomposition of $Q_{(l,(\geq p))} = Q_{l,p} \cup Q_{(l,(\geq p))}$ and denote $a_p: Q_{l,p} \subset Q_{(l,(\geq p))} $ and $b_{(>p)} : Q_{(l,(>p))} \subset Q_{(l,(\geq p))}$. Then $\widetilde{\nu}_l|_{Q_{(\geq p)}}(F) \in D^{\geq 0}$ if and only if
	\begin{enumerate}\label{thm:cellularpurity:enum2}
		\item\label{thm:cellularpurity:enum2:item1} $a_p^! \widetilde{\nu}_l|_{Q_{(\geq p)}}(F) \in D^{\geq 0} $
		\item\label{thm:cellularpurity:enum2:item2} $b_{>p}^\ast \widetilde{\nu}_l|_{Q_{(l,\geq p)}}(F)\in D^{\geq 0} $.
	\end{enumerate}
	Consider the commuting square
	\[
	\begin{tikzcd}
		{Q_{(l,(\geq p))}} \arrow[r, "{i_{(l,(\geq p))}}"]      & Q_{(\geq p)}             \\
		{Q_{(l,p)}} \arrow[u, "a_{p}"] \arrow[r, "{i_{(l,p)}}"] & Q_{p} \arrow[u, "i_{p}"]
	\end{tikzcd}
	\]
	hence $a_p^!i_{(l,(\geq p))}^!F = i_{(l,p)}^!i_p^!F$. Notice that $\widetilde{\nu}_l(F)|_{Q_{\geq p}} \simeq i_{(l,(\geq p))}^!F [d(l)]$. Since by \ref{defi:pfoliation:item1} of \ref{defi:pfoliation}, $i_{(l,p)}^![d(l)]$ is left $t$-exact and $i_p^!F$ takes values in degree $\geq \p(p)$ \ref{thm:cellularpurity:enum2:item1} follows. \ref{thm:cellularpurity:enum2:item2} follows by induction since $b_{(>p)}^\ast i_{(l,(\geq p))}^!F = i_{(l,(>p))}^!F$.
\end{proof}
\begin{prop}[Right $t$-exactness]
	Let $\pi:Q \to L\times T$ satisfy \ref{tfitem1}-\ref{tfitem4} of \ref{sec:universal} and \ref{item1: regassum}. Given $l \in L$, then the normalized projection functor
	\[
	\widetilde{\nu}_l: \Fun(Q,\D) \to \Fun(Q_l,\D).
	\]
	is right $t$-exact. On $\Fun(Q,\D)$ we put the perverse $t$-structure and on $\Fun(Q_l,\D)$ the standard $t$-structure.
\end{prop}
\begin{proof}
	By \ref{thm:conclusion} we have the formula
	\[
	\widetilde{\nu}_l(F)(q) = \totfib(F|q/Q^{t(q)})
	\]
	Consider the induced map $d: q/Q^{t(q)}\subset Q \to L \xrightarrow{\delta -\delta(l(q))} \NN$, normalized so that $d(q) = 0$ (i.e.\ substracted by $\delta(l(q))$). Given $F \in {}^\p D^{\leq 0}$, then for $r \in q/Q^{t(q)}$ we have
	\[
	F(r) \in D^{\leq \p(s(r))} \subset D^{\leq - \delta(l(r))}
	\]
	since $\p(s(r)) + \delta(l(r)) \leq 0$ by \ref{item2: regassum}, so that the assumptions of \ref{prop:Qestimate} are satisfied, and we conclude $\nu_l(F)(q) \in D^{\leq 0}$.
\end{proof}
\begin{cor}\label{cor:cellulartexactness}
	Let $Q \to L\times T$ satisfy \ref{tfitem1}-\ref{tfitem4} of \ref{sec:universal} and \ref{item1: regassum},\ref{item2: regassum}. Then the projection functor restricted to $\Fun^P(Q,\D)$
	\[
	\widetilde{\nu}_l|_{\Fun^P(Q,\D)}: \Fun^P(Q,\D)\to \Fun(Q_l,\D)
	\]
	is $t$-exact.
\end{cor}

\begin{cor}
    \label{cor:texactfoliated}
	Given $Q \to L \times T$ as in Corollary \ref{cor:cellulartexactness}. Then $F \in {}^\p D^\heartsuit$ if and only if $\nu_l(F) \in D^\heartsuit$ for all $l \in L$.
\end{cor}
\begin{proof}
	The only if direction follows directly from \ref{cor:cellulartexactness}. For the reverse direction notice that $F\in {}^\p D^\heartsuit$ if and only if ${}^\p H^i(F) = 0$ for all $i \neq 0$. This again holds if and only if $\widetilde{\nu}_l({}^\p H^i(F)) = 0$ for all $l \in L$ and $i \neq 0$ by \ref{prop:laxlim} (In fact for all $F \in \Fun(Q,\D)$ we have $F \simeq 0$ if and only if $\widetilde{\nu}_l(F) \simeq 0$ for all $l \in L$). Due to $t$-exactness of $\widetilde{\nu}_l$ (\ref{cor:cellulartexactness}) \[\widetilde{\nu}_l({}^\p H^i(F)) = H^i(\widetilde{\nu}_l(F)),\] but this implies the result since $H^i(\widetilde{\nu}_l(F)) = 0$ for all $l \in L$ and $i \neq 0$ by assumption.
\end{proof}
\subsection{Summary}
We briefly summarize the main results of \S\ref{sec:stratifications} and \S\ref{sec:foliations_and_purity}, and formulate a common framework in which the upcoming examples of \S \ref{sec:examples_of_interest} can be stated uniformly. This framework consists of a collection of maps of posets
\begin{equation}
	s:Q \to P, \quad \p:P \to \ZZ, \quad \pi: Q \to L \times T, \quad \delta: L \to \NN
\end{equation}
subject to the following conditions
\begin{enumerate}\label{enum:universalsetup}
	\item[$(U1)$] The map $\pi: Q \to L \times T$ is a
	Cartesian fibration and admits discrete fibers
	\[
	Q_l^t=\pi^{-1}(\{l\}\times \{t\}),
	\qquad l\in L,\ t\in T.
	\]  
	\item[$(U2)$] For $l<l' \in L$, the poset $l/L/l'$ is a cube of dimension
	$\delta(l') - \delta(l)$.
	\item[$(U3)$] The two maps $\delta,\p$ satisfy the relation 
	\[
	\p(s(q)) + \delta(l(q)) \leq 0
	\]
	for all $q \in Q$.
	\item[$(U4)$] For $q \in Q_l$ and $p = s(q)$, each slice poset $q/Q_p^{t(q)}\setminus \{q\}$ has homotopy type of an $n$-dimensional sphere where $n \geq -(\p(p) + \delta(l)) -1$. 
\end{enumerate}
The following is a collection of the relevant results of \S \ref{sec:stratifications} and \S\ref{sec:foliations_and_purity}:
\begin{prop}
	Let $X \to P$ be a complex conically stratified manifold with complex strata $X_p$ together with a refined cellular decomposition $X \to Q \to P$. Assume, in addition, the presence of a map $\pi: Q \to L\times T$ satisfying the conditions $(U1)$-$(U4)$. Then
	\begin{enumerate}
		\item\label{item1: summary} The refined cellular decomposition $X \to Q \to P$ yields equivalences of $\infty$-categories
		\[
		Sh_{Q}(X,\D) \simeq \Fun(Q,\D); \quad \Sh_P(X,\D) \simeq \Fun^P(Q,\D).
		\]
		\item\label{item2: summary} The stratified poset $Q \to L$ gives rise to the (cellular) foliated Verdier duality equivalence
		\[
		\Fun(Q,\D) \simeq \llim_{l \in L} \Fun(Q_l,\D).
		\]
		\item\label{item3: summary} The normalized diagram
		\begin{align*}
			\widetilde{\nu}: (L^\lhd)^\op \to \Cat_{\infty},\quad  -\infty\neq l &\mapsto
			\Fun(Q_l, \D),\\ -\infty &\mapsto \Fun^P(Q,\D)
		\end{align*}
		whose structure maps are given by the gluing and projection functors, consists entirely of $t$-exact functors. Moreover, the diagram obtained by restricting to $L$ commutes coherently. Consequently, passing to the hearts of the respective $t$-structures yields a diagram
		\[
		L^\op \longrightarrow\Cat_1
		\]
		of abelian categories. 
	\end{enumerate}
\end{prop}
\begin{proof}
	\ref{item1: summary} follows from \ref{prop: localizedexodromy}. \ref{item3: summary} follows from \ref{cor:cellulartexactness}, together with the observation, that the gluing functors are by \ref{thm:conclusion} given as direct sums, hence are also $t$-exact.
\end{proof}

\section{Perverse sheaves via foliated Verdier duality}%
\label{sec:examples_of_interest}

\subsection{Riemann surfaces}%
\label{sub:riemann_surfaces}

We generalize the introductory example from \S \ref{subsec:illustrative} using
the foliated Verdier duality approach of \S
\ref{sec:combinatorial_foliated_verdier_duality}. Let $X$ be a Riemann surface,
$P$-stratified into a finite subset $A \subset X$ and its complement. We
consider the horizontal foliation of a meromorphic quadratic differential
$\omega$ on $X$ with poles of order at most $1$. We may treat this example as a
special case of the setup of Theorem \ref{thm:conclusion} combined with
Corollary \ref{cor:texactfoliated} but, given its simplicity, we prefer to
derive the desired classification result directly. 

To begin, we focus on investigating the category of constructible (and
perverse) sheaves locally at a pole of $\omega$. To this end, we may
assume $X = \CC$, $A = \{0\}$, and $\varphi = z^n dz^2$ with $n \ge -1$. 
The leaves of the horizontal foliation of $\varphi$ are the connected
components of the fibers of the map 
\[
    \lambda: \CC \to \RR_{\ge 0},\; z \mapsto \Im(\sqrt{z^{n+2}}).
\]
In particular, we treat $Z_0 = \lambda^{-1}(0) \subset \CC$ as a single leaf: a
corolla with a vertex of valency $n+2$, given as the union of the origin with
the open rays $R_k = \RR_{> 0}\zeta^k$, $0 \le k < n+2$, where $\zeta =
\exp( \frac{2 \pi i }{n+2})$. The remaining leaves are regular leaves $Z \cong
\RR$ exhausting the open sectors between the rays. 

The poset $Q$ is the incidence poset of the cell decomposition of $\CC$ into
the origin, labelled by $* \in Q$, the $n+2$ open rays, labelled by
$i \in \ZZ/(n+2)$, and the open sectors between the rays, labelled by the pairs
$(i+1,i)$ of incident rays, $i \in \ZZ/(n+2)$. A diagram $F: Q \to \D$ may
therefore be depicted as 
\begin{equation}
    \label{diag:const}
    \begin{tikzcd}[execute at end picture = {
                \node[E = (\tikzcdmatrixname-1-3)] {};
                \node[E = (\tikzcdmatrixname-1-5)] {};
                \node[E = (\tikzcdmatrixname-1-7)] {};
                \node[E = (\tikzcdmatrixname-2-2)(\tikzcdmatrixname-2-8)(\tikzcdmatrixname-3-5)] {};
                            }
                  ]
         & & F_{(2,1)} & & F_{(1,0)} &  & F_{(0,n+1)} &  & \\
        \dots & F_2 \ar{ur}& &\ar{ul} F_1 \ar{ur} & & \ar{ul} F_0 \ar{ur} & & \ar{ul} F_{n+1} & \dots\\
              &&&& F_* \ar{uu} \ar[bend left]{uull}\ar[bend right]{uurr} \ar[bend left]{ulll}\ar{ul}\ar{ur}\ar[bend right]{urrr} &&&&
    \end{tikzcd}
\end{equation}
where the dashed regions indicate the fibers of the map to $L$. Such a diagram is
$P$-constructible if the arrows connecting the top two rows are equivalences in
$\D$. Unpacking the description of the lax limit
\begin{equation}
    \label{eq:laxlim}
    \varprojlim_{l \in L} \Fun(Q_l,\D)
\end{equation}
in terms of sections of the Grothendieck construction, yields an identification
of this category with diagrams of the form
\begin{equation}
    \label{diag:dual}
\begin{tikzcd}
     & & F'_{(2,1)} & & \ar[dotted,-,bend right=65]{drrr} \ar[dotted,-,bend left=65]{dlll} F'_{(1,0)} &  & F'_{(0,n+1)} &  & \\
    \dots & F'_2 \ar[<-, color=red]{ur}& &\ar[<-, color=red]{ul} F'_1 \ar[<-, color=red]{ur} & & \ar[<-, color=red]{ul} F'_0 \ar[<-, color=red]{ur} & & \ar[<-, color=red]{ul} F'_{n+1} & \dots\\
          &&&& F'_* \ar[<-, color=red]{uu} \ar[bend left,<-, color=red]{uull}\ar[bend right,<-, color=red]{uurr} \ar[bend left=40]{ulll}\ar{ul}\ar{ur}\ar[bend right=40]{urrr} &&&&
\end{tikzcd}
\end{equation}
where the solid triangles commute (coherently), while the dotted edges indicate
zero homotopies for all composites of the form $F'_{(i+1,i)} \to F'_* \to F'_k$
with $k \notin \{i+1,i\}$. The zero relations come about as follows: the
transition functor of the contravariant diagram $L \to \Cat_{\infty}, l \mapsto
\Fun(Q_l,\D)$ corresponding to the arrow in $L$ connecting $*$ to an open
sector, say $(1,0)$, associates to an object $F'_{(1, 0)} \in \D$ (the stalk of
a constructible, and thus constant, sheaf on the open sector) the diagram in
$\Fun(Q_0, \D)$ given by  
\begin{align*}
    * &\mapsto F'_{(1, 0)}\\
    i & \mapsto \begin{cases}  F'_{(1, 0)} & \text{for $i \in \{0,1\}$}\\ 0 & \text{else.} \end{cases}
\end{align*}
The gluing datum corresponding to the edge $* \to (1,0)$ in the lax limit
\eqref{eq:laxlim} consists of a natural transformation between this diagram
and another given diagram in $\Fun(Q_0, \D)$ which then produces precisely the
red arrows in \eqref{diag:dual} with domain $F'_{(1,0)}$ along with the claimed
zero homotopies. 

The foliated Verdier duality 
\begin{equation}
    \label{eq:foliatedverdierriemann}
\Fun(Q,\D) \simeq \llim_{l \in L} \Fun(Q_l,\D),
\end{equation}
associates to a diagram as in \eqref{diag:const} the diagram 
\begin{equation}
    \label{eq:folverdier}
        \scriptsize
        \begin{tikzcd}[column sep=.2cm]
          & F_{(2,1)}[1] & & F_{(1,0)}[1] &  & F_{(0,n+1)}[1] &   \\
          \dots  \ar[<-, color=red]{ur}& &\ar[<-, color=red]{ul} \fib(F_1 \to F_{(2,1)} \oplus F_{(1,0)}) \ar[<-, color=red]{ur} & & \ar[<-, color=red]{ul} \fib(F_0 \to F_{(1,0)} \oplus F_{(0,n+1)}) \ar[<-, color=red]{ur} & & \ar[<-, color=red]{ul} \dots  \\
         &&& \fib(F_* \to \oplus_{i \in \ZZ/(n+2)} F_{(i+1,i)}) \ar[<-, color=red]{uu} \ar[bend left,<-, color=red]{uull}\ar[bend right,<-, color=red]{uurr} \ar[bend left=40]{ulll}\ar{ul}\ar{ur}\ar[bend right=40]{urrr} &&&
    \end{tikzcd}
\end{equation}
of the form \eqref{diag:dual}. Note the distinctive feature of foliated Verdier
duality: arrows within each leaf of the foliation keep their orientation while
arrows connecting different leaves are reversed. 

From the explicit description of foliated Verdier duality in
\eqref{eq:folverdier}, we observe that it identifies $P$-constructible diagrams
with diagrams of the form \eqref{diag:dual} such that, again, the arrows
connecting the two top rows are equivalences in $\D$. Finally, $P$-smooth
perverse sheaves correspond precisely to the $P$-constructible diagrams in
$\Fun(Q,\D)$ whose foliated Verdier dual is a diagram of the form
\eqref{diag:dual} taking values in the abelian category $\A \subset \D$. 

In conclusion, while our foliated Verdier duality in the form of
\eqref{eq:foliatedverdierriemann} yields an equivalence at the level of stable
$\infty$-categories, when passing to perverse sheaves, we recover the
classication results of \cite{KS:Riemann}:
\begin{prop}
    \label{prop:riemann}
    The category of perverse sheaves on $\CC$ valued in the abelian category $\A$
    is equivalent to the category of data of the following form
    \begin{enumerate}[label=\arabic *.]
        \item an object $E_* \in \A$ along with families $\{E_i\}$ and
            $\{E_{(i+1,i)}\}$ of objects in $\A$, both parameterized by $i \in \ZZ/(n+2)$, 
        \item for every $i \in \ZZ/(n+2)$, a morphism $E_{(i+1,i)} \to E_*$ and a morphism $E_* \to E_i$ in $\A$,
    \end{enumerate}
    subject to the condition:
    \begin{itemize}
        \item for $i,k \in \ZZ/(n+1)$, the composite $E_{(i+1,i)} \to E_* \to E_k$
            is an equivalence, for $k \in \{i,i+1\}$, and $0$, else.
    \end{itemize}
\end{prop}

It is a straightforward task to generalize to more general Riemann surfaces
equipped with a quadratic differential as specified above: either one can
adapt the above setup and give a mutatis mutandis proof, or, one can find a
suitable covering of the Riemann surface by open disks and use {\em ordinary}
$\infty$-categorical descent for the association $X \mapsto \Sh_P(X;\D)$. 

\subsection{Real hyperplane arrangements}%
\label{sub:hyperplane_arrangements}

Let $\H = \{H_1, ..., H_k\}$ be a finite configuration of hyperplanes in
$\RR^n$, cut out by the linear forms $l_1, ...,l_k \in (\RR^n)^*$. We consider
the stratification $\CC \to \{0 < 1\}$ of $\CC$ into the origin and its
complement. Evaluating the real linear forms on $\CC^n$, we consider the composite
\[
    \CC^n \overset{(l_i)}{\to} \CC^k \to \{0 < 1\}^k.
\]
Restricting the codomain of this map to its image $P \subset \{0 < 1\}^k$, we
obtain a $P$-stratification $\CC^n \to P$ with open stratum given by the
complement of the union of the complexified hyperplanes $H_{\CC}$. We will
investigate the category $\Sh_P(\CC^n;\D)$ of $P$-constructible sheaves on
$\CC^n$ via foliated Verdier duality.

The leaves of the foliation to be used are the fibers of the map $\Re: \CC^n
\to \RR^n$ given by component-wise passage to real parts. Thus, all leaves are
homeomorphic to $\RR^n$, but the $P$-constructibility constraints vary as
follows: Consider the poset $\Sigma = \{- > 0 < + \}$ and the stratification
$\RR \to \Sigma$ of $\RR$ into $\RR_{<0} \amalg \{0\} \amalg \RR_{>0}$.
Evaluating the linear forms $l_i$ on $\RR^n$ yields a continuous map
\[
    \pi: \RR^n \overset{(l_i)}{\to} \RR^k \to \Sigma^k
\]
and we obtain, denoting the image of this map by $L$, an $L$-stratification
$\pi: \RR^n \to L$. We refer to the strata of this stratification of $\RR^n$ as the
{\em flats} of $\H$. For $x \in \RR^n$, the leaf $Z = \Re^{-1} x \subset \CC^n$
can be identified with the space $\RR^n$, stratified by the flats of the
configuration $\H_x := \{H \in \H | x \in H\}$ (if $\H_x = \emptyset$ then, by
convention, there is a single stratum). 

To study foliated Verdier duality in its cellular incarnation, as developed
in \S \ref{sec:combinatorial_foliated_verdier_duality}, we consider the
stratification
\[
    \CC^n \overset{(\Re, \Im)}{\to} \RR^n \times \RR^n \overset{\pi \times
    \pi}{\to} L \times L
\]
which is in fact a regular cell decomposition with cells given by $A+iB$ where
$A,B \in L$ are flats of the real configuration (cf. \cite{KS:hyperplane}). The
combinatorial setup is then obtained by setting $Q = L \times L$, equipped with
the projection $Q \to L$ to the first component and the map $f: Q \to P$ is given as follows: 
Consider the map $\rho: \Sigma^2 \to \{0 < 1\}$ which maps $(0,0)$ to $0$ and
all other elements to $1$. Then set 
\[
    L \times L \to \{0 < 1\}^k, ((x_1, ..., x_k), (y_1, ..., y_k)) \mapsto (\rho(x_1,y_1), ..., \rho(x_k,y_k))
\]
and observe that this map takes values in $P$. We define the codimension $c(A)$
of a flat $A \in L$ to be the length of a maximal chain $A = A_0 < A_1 < ... <
A_d$ and set 
\[
    \delta: L \to \NN, A \mapsto n - c(A).
\]

Let $\X = \Fun(Q,\D)$ and, for $A \in L$, set $\X_A = \Fun(Q_A,\D)$. We analyze
the foliated Verdier duality equivalence
\begin{equation}
    \label{eq:matrix}
    \tilde{\nu}:\X \simeq \llim_{A \in L} \X_A
\end{equation}
arising from Theorem \ref{thm:lax_lim}, in its normalized version as described
in Remark \ref{rem:normalization}. 

\begin{thm}
    \label{thm:hyperplanefolver} 
    Let $R$ be a ring, let $\A$ be the abelian category of $R$-modules, and let
    $\D = \Mod_R$ be the derived category of $R$-modules. 
    \begin{enumerate}
        \item\label{thm:hyperplanefolver:0} The normalized gluing diagram 
            \[
                L \to \Cat_{\infty}, A \mapsto \Fun(Q_A,\D)
            \]
            is equivalent to the constant diagram $A \mapsto \Fun(L,\D)$ where
            we identify $\Fun(Q_A,\D)$ with $\Fun(L,\D)$ via projection to the
            second factor of $L \times L$. 
        \item\label{thm:hyperplanefolver:1} The foliated Verdier duality equivalence \eqref{eq:matrix} takes the form
            \[
                \tilde{\nu}: \X = \Fun(L \times L,\D) \simeq \Fun(L^{\op} \times L, \D).
            \]
        \item\label{thm:hyperplanefolver:2} The equivalence $\tilde{\nu}$
            identifies the $\infty$-category $\Sh_P(X;\D)$ of $P$-constructible
            sheaves with the full subcategory of 
            \[
                \Fun(L^{\op} \times L, \D).
            \]
            consisting of those diagrams that send edges of the form
            \begin{enumerate}
                \item\label{edge:1} $(A,B) \le (A,B')$ with $f(A,B) = f(A,B')$ and
                \item\label{edge:2} $(A,B) \ge (A',B)$ with $f(A,B) = f(A',B)$
            \end{enumerate}
            to equivalences in $\D$.
        \item\label{thm:hyperplanefolver:3} 
            The equivalence $\tilde{\nu}$ identifies
            $\Perv_P(X;\A)$ with the full subcategory of
            \[
                \Fun(L^{\op} \times L, \A)
            \]
            sending the edges listed in \ref{thm:hyperplanefolver:2}
            \ref{edge:1} and \ref{edge:2} to isomorphisms in $\A$. 
    \end{enumerate} 
\end{thm}
\begin{proof}
    Using Theorem \ref{thm:conclusion}, we obtain that, for $A \in L$, the
    normalized projection functor is given by the formula
    \[
        \widetilde{\nu}_A(F)(A,B) = [-\delta(A)]\tfib( F | (A/L) \times B).
    \]
    Using that, for every $A' \le A \in L$, the poset $A'/L/A$ is a cube of
    dimension $\delta(A') - \delta(A)$, it follows that the unnormalized gluing
    functor $\gamma_{A,A'}: \X_{A'} \to \X_{A}$ can be identified with the
    shift $[-\delta(A') + \delta(A)]$ and hence the normalized gluing functor
    can be identified with the identity functor. To see that the normalization
    choices can indeed be made coherently, assume for simplicity that all
    hyperplanes intersect in a single point. In this case, we may identify the
    normalized projection functor with 
    \[
        \widetilde{\nu}_A(F)(A,B) = \tfib( j_! (F | (A/L) \times B)).
    \]
    where $j: A/L \subset L$ is the inclusion. From this formula it can be read
    off that all images of a diagram $F \in \Fun(Q_A,\D)$ under the various
    normalized gluing functors $\gamma_{A',A}$, for $A' \le A$, can be
    identified (pointwise at $B \in L$) with total fibers of the same diagram
    $j_! (F | (A/L) \times B)$ hence coherently identifying $\gamma$ with the
    identity. A suitable adaptation of this argument can also be used to treat
    the general case, showing \ref{thm:hyperplanefolver:0}.

    Statement \ref{thm:hyperplanefolver:1} is an immediate consequence of
    \ref{thm:hyperplanefolver:0}, using 
    \[
        \llim_{A \in L} \Fun(L,\D) \simeq \Fun(L^{\op}, \Fun(L,\D)) \simeq \Fun(L^{\op} \times L,\D)
    \]
    For Part \ref{thm:hyperplanefolver:2}, we use that the $\infty$-category $\Sh_P(X;\D)$ of
    constructible sheaves can be identified with the full subcategory of
    $\Fun(Q,\D)$ inverting edges sent by $f$ to degenerate edges in $P$ (cf. \S
    \ref{sub:cellular_sheaves}), and trace this condition through the foliated
    Verdier duality. 

    Finally, \ref{thm:hyperplanefolver:3} follows immediately from Corollary
    \ref{cor:texactfoliated}.
\end{proof}

\begin{rem}
    \label{rem:ks} The statement of Theorem \ref{thm:hyperplanefolver}
    \ref{thm:hyperplanefolver:3} recovers the main result from
    \cite{KS:hyperplane} (where it is proved using Cousin complexes). Note that
    our approach is rather different, in particular, the {\em derived}
    equivalence of \ref{thm:hyperplanefolver:3} given by foliated Verdier
    duality is new. 
\end{rem}

\subsection{Symmetric products}%
\label{sub:symmetric_products}

We consider the symmetric product $X = \Sym_0^n(\CC)$ (with center of mass $0$)
and use divisor notation $\sum n_k [z_k]$, $n_k \in \NN$, $z_k \in \CC$, to
denote its elements. We equip $X$ with the discriminant stratification, given
by 
\[
    X \to P, \sum n_k [z_k] \mapsto \sum n_k
\]
where $P$ denotes the poset of partitions of $n$, ordered by refinement. 
To study the category $\Sh_P(X;\D)$ of $P$-constructible sheaves, we consider the
foliation with leaves given by the fibers of the map
\[
    X \to \Sym_0^n(\RR), \sum n_k [z_k] \mapsto \sum n_k [\Re(z_k)].
\]
In particular, the leaf $Z_x \subset \Sym_0^n(\CC)$ lying over a point $x = n_1 [x_1]
+ n_2[ x_2] + ... + n_m [x_m]$ with pairwise distinct $x_k$ is homeomorphic to 
\[
    Z_x \cong \Sym_0^{n_1}(\RR) \times \cdots \times \Sym_0^{n_m}(\RR).
\]
The {\em real} symmetric product $\Sym_0^n(\RR)$ is naturally stratified by the
poset $L$ of {\em compositions} (i.e. ordered partitions) of $n$, via the map 
\[
    \Sym_0^n(\RR) \to L, \; \sum n_k [x_k] \mapsto (n_1, ..., n_m).
\]
The leaves in $\Sym_0^n(\CC)$ that lie over $(n_1, ..., n_m) \in L$ vary in a
locally constant fashion so that the corresponding $L$-stratum of
$\Sym_0^n(\CC)$ is homeomorphic to $\RR^m \times Z_x$.

The space $\Sym^n_0(\CC)$ admits a cell structure parameterized by the poset $Q$
of {\em contingency matrices}. These are matrices with entries in $\NN$ whose
row-- and column--sums form compositions of $n$. Forming column--sums equips $Q$
with a map $Q \to L$ and there is a forgetful map $Q \to P$. Finally, for a
composition $l \in L$, we define $\delta(l)$ to be the number of components of $l$.

\begin{exa}
    \label{exa:3}
    For $n=3$, the poset $L$ is given by 
    \[
    \begin{tikzcd}
          & 21 \ar{dr}& \\
        3 \ar{ur}\ar{dr} & & 111\\
          & 12 \ar{ur} & 
    \end{tikzcd}
    \]
    and the $L$-stratified poset $Q$ of contingency matrices splits into the
    strata $\{Q_l\}_{l \in L}$ as shown in Figure \ref{fig:sym3}.
\end{exa}

\begin{figure}[htp]
\centering
    \includegraphics[width=\textwidth]{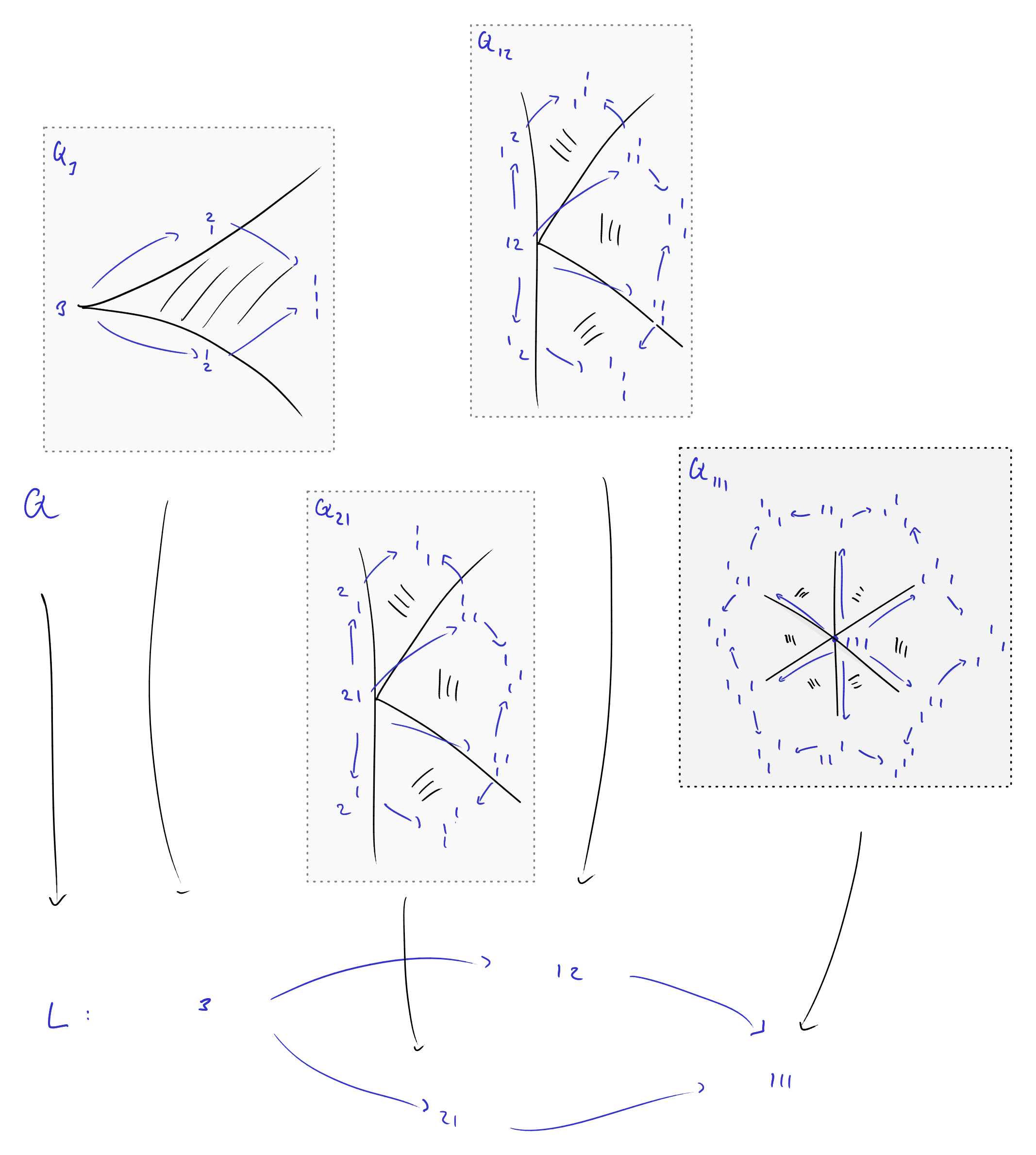}
    \caption{representation of fibers of the map $Q \to L$ for $Q$ the poset of $3 \times 3$ contingency matrices}
\label{fig:sym3}
\end{figure}

It is useful for our analysis to keep track of a bit more structure: we set $T
= L$ and consider the map $\pi: Q \to L \times T$ given by column and row sums
in the respective factors $L$ and $T$. Geometrically, this map corresponds to
taking into account a ``transverse'' foliation with leaves corresponding to the
fibers of the map 
\[
    X \to \Sym_0^n(\RR),\; \sum n_k [z_k] \mapsto \sum n_k [\Im(z_k)].
\]
We take note of the following observations: 
\begin{prop}
    \label{prop:cartesiancube}
    \begin{enumerate}
        \item The map $\pi$ is a Cartesian fibration with discrete fibers.
        \item For every ordered pair $l < l'$ in $L$, the interval $[l,l'] \subset
            L$, i.e. the subposet of intermediate elements $l \le x \le l'$, forms
            a {\em cube} of dimension $k = \delta(l') - \delta(l)$ (i.e. is in
            order-perserving bijection with the poset $[1]^k$).
    \end{enumerate}
\end{prop}
\begin{proof}
    Immediate.
\end{proof}

\begin{thm}
    \label{thm:symfolver} 
    Let $\X = \Fun(Q,\D)$ equipped with the
    stratification $\{\X_l\}_{l \in L}$ with $\X_l = \Fun(Q_l,D)$ from Theorem
    \ref{thm:lax_lim}. Then 
    \begin{enumerate}
        \item The locally Cartesian fibration $\pi: \G(\X) \to L^{\rhd}$ from
            Definition \ref{defi:grothendieck} is (globally) Cartesian so that
            its straightening defines a strictly (coherently) commutative
            gluing diagram $L \to \Cat_{\infty}$.
        \item For $l$ in $L$, the normalized projection functor $\widetilde{\nu}_l: \X \to
            \X_l$ is given, for $F \in \X$ and $q \in Q_l$, by the pointwise formula
            \[
                \widetilde{\nu}_l(F)(q) \simeq \tfib(j_!F|q/Q^{t(q)})
            \]
            where $j_!$ denotes the left Kan extension (i.e. extension by zero)
            along $q/Q^{t(q)} \subset Q^{t(q)}$.
        \item\label{symver:3} For $l < l'$ in $L$, the normalized gluing functor $\widetilde{\gamma}_{l,l'}: \X_{l'} \to
            \X_l$ is given, for $F \in \X_{l'}$ and $q \in Q_l$, by the pointwise formula
            \[
                \gamma_{l,l'}(F)(q) \simeq \bigoplus_{q' \in q/Q^{t(q)}_{l'}} F(q')
            \]
    \end{enumerate} 
\end{thm}
\begin{proof}
    This follows from Theorem \ref{thm:conclusion}
\end{proof}

\begin{exa}
    \label{exa:symfolver}
    Let $n=3$ and $F \in \X$. We let $\D = \D(\Vect)$ be the derived category
    of vector spaces over a field $K$. We denote by $*$ a zero object in $\D$.
    The values of the foliated Verdier dual $\{\nu_{l}(F)\}_{l \in L}$ in the
    transversal leaf $Q^{3}$ over $3$ are
    \begin{equation}
        \label{eq:folverfor}
    \begin{array}{lcr}
                \scriptsize
                \nu_{3}(F)(3) = \left[ 
                \begin{tikzcd}
                    & 21 \ar{dr} & \\
                    3 \ar{ur}\ar{dr} & & 111\\
                    & 12 \ar{ur}& \\
                \end{tikzcd}
                \right]
                & 
                \scriptsize
                \nu_{21}(F)(21) = \left[ 
                \begin{tikzcd}
                    & 21 \ar{dr} & \\
                    * \ar{ur}\ar{dr} & & 111\\
                    & * \ar{ur}& \\
                \end{tikzcd}
                \right]\\
                \scriptsize
                \nu_{12}(F)(12) = \left[ 
                \begin{tikzcd}
                    & * \ar{dr} & \\
                    * \ar{ur}\ar{dr} & & 111\\
                    & 12 \ar{ur}& \\
                \end{tikzcd}
                \right]
                &
                \scriptsize
                \nu_{111}(F)(111) = \left[ 
                \begin{tikzcd}
                    & * \ar{dr} & \\
                    * \ar{ur}\ar{dr} & & 111\\
                    & * \ar{ur}& \\
                \end{tikzcd}
                \right]
    \end{array}
    \end{equation}
    where the square brackets indicate total fibers of the diagrams in $\D$
    obtained by applying $F$ to the depicted diagrams in $Q$. These values are
    related by gluing data exhibiting $\{\nu_l(F)\}_{l \in L}$ as an object of
    the lax limit $\llim_{l \in L} \X_l$ which, when restricted to the slice
    $Q^3$ amounts to the canonical commutative square
    \begin{equation}\label{eq:dualdiag}
        \begin{tikzcd}
                & \nu_{21}(F)(21) \ar{dl}& \\
            \nu_{3}(F)(3) & & \ar{ul} \ar{dl}\nu_{111}(F)(111)\\
                          & \nu_{12}(F)(12) \ar{ul}& 
        \end{tikzcd}
    \end{equation}
    Note first, that when passing from $F$ to $\nu_*(F)$ the arrows withing the
    transversal leaf $Q^3$ have been reversed -- the expected effect of
    foliated Verdier duality.
    Observe further that there is a canonical equivalence between the total
    cofiber of \eqref{eq:dualdiag} and $F(3)$ - we know that foliated Verdier
    duality is an equivalence, here we see how to recover the value of $F$ at
    $3$ from $\nu_*(F)$. For a perverse sheaf $F$, the objects $G$, $F_1$,
    $F_2$, $E$ are pure in degree $0$ so that the total cofiber of
    \eqref{eq:dualdiag} can be identified with the complex of vector spaces
    \[
        F(3) \simeq \nu_{111}(F)(111) \to \nu_{21}(F)(21) \oplus \nu_{12}(F)(12)  \to \nu_{3}(F)(3) \in \D(\Vect)
    \]
    This recovers the perspective of \cite{gm:rebroussement} (and \cite{ks:bruhat})
    where \eqref{cor:symperverse} is established by explicitly reverse
    engineering the perverse sheaf from the data comprised in the gluing data
    for $\{\nu_l(F)\}$. An advantage of our approach is that the equivalence
    is given {\em a priori} and already on the level of derived categories by
    means of lax descent. 
    
    We conclude by computing the values of $\nu_3(F)$ in the leaf $Q_3$,
    besides $\nu_3(F)(3)$ which we have already described above. The values are
    given by 
        \[
                \scriptsize
                \nu_{3}(F)(\mat{2\\1}) = \left[ 
                \begin{tikzcd}
                    &\mat{1 & 1 & \\& & 1} & \\
                    \mat{2 & \\& 1} \ar{ur}& & \ar{d}\ar{ul}\mat{1 & 1 \\& 1} &  \\
                    \mat{2\\1} \ar{u}\ar{d}\ar{urr}\ar{drr} &&  \mat{1 &  &1 \\&1 & }  \\
                    \mat{& 2\\1 & }\ar{dr}&&  \ar{dl}\ar{u}\mat{1 & 1 \\1 & } & \\
                    &\mat{ & 1 &1 \\1& & } & 
                \end{tikzcd}
                \right]
                \quad\quad
                \scriptsize
                \nu_{3}(F)(\mat{1\\2}) = \left[ 
                \begin{tikzcd}
                    &\mat{ & &1 \\1&1 & } & \\
                    \mat{ & 1\\2 & } \ar{ur}& & \ar{d}\ar{ul}\mat{ & 1 \\1 & 1} &  \\
                    \mat{1\\2} \ar{u}\ar{d}\ar{urr}\ar{drr} &&  \mat{ & 1 & \\1& &1 }  \\
                    \mat{& 2\\1 & }\ar{dr}&&  \ar{dl}\ar{u}\mat{1 &  \\1 &1 } & \\
                    &\mat{1 &  & \\&1 &1 } & 
                \end{tikzcd}
                \right]
            \]
        \[
                \scriptsize
                \nu_{3}(F)(\mat{1\\1\\1}) = \left[ 
                \begin{tikzcd}
                          & \mat{1 & & \\& 1 & \\& & 1}  &  \ar{l}\ar{r}  \mat{1 & \\& 1\\& 1}   & \mat{ 1& & \\&  & 1\\& 1& }& \\
                          & \mat{ 1& \\1& \\& 1} \ar{u}\ar{dl}  &       &  \mat{ 1& \\&1 \\1& }\ar{u}\ar{rd} &  \\
                    \mat{ &1 & \\1&  & \\& & 1} & & \mat{1\\1\\1}\ar{uu} \ar{ur}\ar{ul}\ar{dl}\ar{dr}\ar{dd} &&\mat{ & 1& \\&  & 1\\1& & }\\
                          & \mat{ &1 \\1& \\&1 } \ar{ul}\ar{d}  &       & \ar{ur}\ar{d} \mat{ & 1\\&1 \\1& } &  \\
                          & \mat{ & & 1\\1&  & \\& 1& }  &    \mat{ & 1\\1& \\1& }\ar{r}\ar{l}   &  \mat{ & & 1\\&1  & \\1& & } & \\
                \end{tikzcd}
                \right]
            \]
            These values organize into a commutative square $\nu_3(F): Q_3 \to \D$ given by 
            \begin{equation}\label{eq:dualdiag2}
                \begin{tikzcd}
                        & \nu_{3}(F)(\mat{2\\1}) \ar{dr}& \\
                    \nu_{3}(F)(3)\ar{ur}\ar{dr} & & \nu_{3}(F)(\mat{1\\1\\1})\\
                                  & \nu_{3}(F)(\mat{1\\2}) \ar{ur}& 
                \end{tikzcd}
            \end{equation}
            where, as expected from foliated Verdier duality, the arrows {\em
            within} the leaf do not get reversed under the passage from $F$ to
            $\nu_*(F)$. All other values of $\nu_l(F)(q)$ have a similar form,
            by Theorem \ref{thm:symfolver} they are obtained by totalizing the
            corresponding transversal leaf $Q^{t(q)}$ at $q$ with zeroes placed
            at vertices in the complement of $q/Q^{t(q)}$ (cf.
            \eqref{eq:folverfor}).
\end{exa}

From the $t$-exactness of foliated Verdier duality, we recover the following
result of \cite{ks:bruhat}.

\begin{cor}
    \label{cor:symperverse}
    Let $\A$ be an abelian category. The category $\Perv_P(\Sym_0^n(\CC);\A)$
    of $P$-constructible perverse sheaves on $\Sym_0^n(\CC)$ with values in
    $\A$ is equivalent to the category comprised by the following linear
    algebraic data:
    \begin{enumerate}
        \item for every $q \in Q$, an object $F_q \in \A$,
        \item for every $l \in L$, and $\alpha: q \le q'$ in $Q_l$, a
            morphism $\alpha_*: F_q \to F_{q'}$ in $\A$, such that
            \begin{enumerate}
                \item the assignment $\alpha \mapsto \alpha_*$ is covariantly functorial in $Q_l$,
                \item if $p(q) = p(q')$ then $\alpha_*$ is an isomorphism. 
            \end{enumerate}
        \item for every $t \in T$, and $\beta: q \le q'$ in $Q^t$, a
            morphism $\beta^*: F_{q'} \to F_q$ in $\A$, such that
            \begin{enumerate}
                \item the assignment $\beta \mapsto \beta^*$ is contravariantly functorial in $Q^t$,
                \item if $p(q) = p(q')$ then $\beta^*$ is an isomorphism. 
            \end{enumerate}
        \item for $x \in Q^t_l$, $x' \in Q^t_{l'}$, $y \in Q^{t'}_l$ with $\alpha: x \le x'$ and $\beta: x \le y$, we have 
            \[
                \beta_* \circ \alpha^* = \sum \gamma^* \circ \delta_*
            \]
            where the sum ranges over all $y' \in Q^{t'}_{l'}$ with $\gamma: y \le y'$ and $\delta: x' \le y'$.
    \end{enumerate}
\end{cor}
\begin{proof}
    This follows from the normalized foliated Verdier duality equivalence
    \eqref{eq:folver} upon unravelling the commutativity constraints satisfied
    by sections of the Grothendieck construction comprising the lax limit (cf.
    Remark \ref{rem:ordinary_explicit}), using the formula from Theorem
    \ref{thm:symfolver} \ref{symver:3}.
\end{proof}

\subsection{Coxeter systems}
\label{sub:coxeter}

We conclude with one more example that falls into the above
context, generalizing \S \ref{sub:symmetric_products}.

\begin{exa}[Coxeter systems]
\label{exa:example_coxeter_systems}

As explained in \cite{ks:bruhat} the example from \S
\ref{sub:symmetric_products} generalizes by considering, for a finite
reflection group $W \subset \on{O}(\h_{\RR})$, the quotient space $X = W
\backslash \h$ with $\h = \h_{\RR} \otimes_{\RR} \CC$. Adpoting the notation of
\cite{ks:bruhat}, we consider for the corresponding Coxeter system $(W,S)$, the
poset $C = \P(S)$ of subsets of $S$. Let 
\[
    \pi: Q \to C^{\op} \times C^{\op}
\]
denote the contravariant Grothendieck construction of the diagram of sets 
\[
    \Xi: C \times C \to \Set, \; (I,J) \mapsto W \backslash ( W/W_I \times W/W_J )
\]
Setting $L = T = C^{\op}$ realizes the setup of \S \ref{sec:universal}. As
explained in loc. cit. the poset $Q$ parameterizes a regular cell decomposition
of the space $W \backslash \h$ so that we may identify the $\infty$-category $\X =
\Fun(Q,\D)$ with cellular sheaves according to Proposition \ref{prop:strat}.
Generalizing \S \ref{sub:symmetric_products}, the $L$-stratification of $X$ is
foliated with leaves given by the fibers of the map $X \to W \backslash
\h_{\RR}$ induced by taking the real parts of all components, while the $T$
stratification is foliated analogously by the fibers of the map taking
imaginary parts.
The ``mixed Bruhat'' description of perverse sheaves now follows directly from
Theorem \ref{thm:conclusion} combined with the purity result
\cite[Proposition 4.1]{ks:bruhat}. 

\end{exa}

\appendix

\section{Sheaves and $t$-structures}
\label{sec:sheaves}
In this section we introduce the sheaf-theoretic framework used throughout the paper. We first consider locally compact Hausdorff spaces and recall the $\infty$-category $\Sh(X)$ of $\D$-valued sheaves for $\D$ stable and presentable, together with the functorial operations that we will use. We then specialize to the case where $X$ is additionally paracompact and of finite covering dimension equipped with a conical $P$-stratification, with each stratum $X_p$, $p\in P$ being locally contractible. In this setting we recall:
\begin{enumerate}
	\item the full subcategory $\Sh_P(X)\subset \Sh(X)$ of $P$-constructible sheaves and the exodromy equivalence
	\[
	\Sh_P(X)\simeq \Fun(\Exit_P(X),\D);
	\]
	\item for $\A=\Mod_R$ the abelian category of $R$-modules and $\D=\D(\A)$ the unbounded derived $\infty$-category, the standard $t$-structure on $\D$ and its induced $t$-structure on $\Sh(X)$;
	\item the perverse $t$-structures on $\Sh_P(X)$ associated to a perversity function $P\to\mathbb Z$.
\end{enumerate}
Thus, while the general sheaf-theoretic formalism is developed for locally compact Hausdorff spaces, the constructibility and exodromy results are considered under the additional paracompactness and finite covering dimensionality and stratification hypotheses above. The additional assumption of paracompactness and finite dimensionality allows us, in the applications considered below, to identify sheaves with hypersheaves and hence to pass between the ordinary and hyperconstructible formulations without further difficulty.

\subsection{Sheaves and functor formalism}

Recall from \cite{lurie:sag} the definition of sheaves valued in $\D$.
\begin{defi}[{\cite[1.1.2.1]{lurie:sag}}]
	Let $X$ be a topological space and $\D$ any $\infty$-category, we call a presheaf
	\[
	F: \Open(X)^\op \to \D 
	\]
	a \textit{sheaf}, if it satisfies 
	\begin{enumerate}
		\item[$( \ast )$] Let $\U= \{U_\alpha\}$ be an open cover of $U \in \Open(X)$ and denote $\Open(X)/\U \subset \Open(X)$, the subposet consisting of $V \in \Open(X)$ for which $ V \subset U_\alpha$ for some $\alpha$. Then the canonical map
		\[
		F(U)  \to \lim_{V \in (\Open(X)/\U)^\op}F(V)
		\]
		is an equivalence in $\D$. Denote the full subcategory $\Sh(X;\D) \subset \Fun(\Open(X)^\op,\D)$ the $\infty$-category of sheaves valued in $\D$. 
	\end{enumerate}
\end{defi}
\begin{lem}[{\cite[1.3.1.6, 1.3.17]{lurie:sag}}]\label{lem:tensorsheaves}
	Let $\D$ be a presentable $\infty$-category. We obtain an equivalence of $\infty$-categories 
	\[
	\Sh(X;\mathcal{S}) \otimes \D \simeq \Sh(X;\D),
	\]
	where $\otimes$ denotes the tensor product as introduced in \cite[4.8.1]{lurie:ha}.
\end{lem}

Let $\D$ be a stable and presentable $\infty$-category. We abbreviate $\Sh(A):= \Sh(A;\D)$. The assignment $A \to \Sh(A)$ for $A$ locally compact Hausdorff space satisfies due to \cite{volpe2025six} the following:
\begin{enumerate}
	\item[$(F1)$] for every continuous map $f: A \to B$, we have adjunctions
	\[
	f_! : \Sh(A) \llra \Sh(B) : f^!
	\]
	and
	\[
	f^* : \Sh(B) \llra \Sh(A) : f_*
	\]
	such that if $f:A \hra B$ is an embedding $f_!$ and $f_*$ are fully faithful (\cite[6.1]{volpe2025six}),
	\item[$(F2)$] For $j: U \subset C$ open, we have $j^! = j^*$ (\cite[6.5]{volpe2025six}),
	\item[$(F3)$] For $i: A \subset C$ closed, we have $i_! = i_*$ (\cite[6.3]{volpe2025six}),
	\item[$(F4)$] For each pullback diagram 
	\[
	\begin{tikzcd}
		A \arrow[r, "\tilde{a}"] \arrow[d, "\tilde{b}"'] & B \arrow[d, "a"] \\
		C \arrow[r, "b"']                                       & D               
	\end{tikzcd}
	\]
	the diagram 
	\[
	\begin{tikzcd}
		\Sh(A) \arrow[d, "\tilde{b}_!"'] & \Sh(B) \arrow[l, "\tilde{a}^\ast"'] \arrow[d, "a_!"] \\
		\Sh(C)                                  & \Sh(D) \arrow[l, "b^\ast"]                          
	\end{tikzcd}
	\]
	commutes (\cite[6.9]{volpe2025six}).
\end{enumerate}

\subsection{Stratification, Constructible sheaves and the Exit path category}

\begin{defi}[Stratification]
	Let $X$ be a topological space and $P$ a partially orderd set, which is a topological space by means of the Alexandrov-topology, i.e.\ $U \subset P$ is open if and only if it is upwards-closed ($x \leq y$, $x \in U \Rightarrow y \in U$). We call a continuous map 
	\[
	s: X \to P
	\]
	a $P$-stratification on $X$.
\end{defi}
\begin{defi}[Constructibility]
	Let $X \to P$ be a stratified space and $\D$ a stable and presentable $\infty$-category. We say that a (hyper)sheaf $F \in \Sh(X;\D)$ (resp.\ $F\in \Sh^{\on{hyp}}(X;\D)$) is $P$-(hyper)constructible if $F|_{X_p}$ is locally constant. Denote $\Sh_P(X;\D) \subset \Sh(X;\D)$ (resp.\  $\Sh^{\on{hyp}}_P(X;\D) \subset \Sh^{\on{hyp}}(X;\D)$) the full subcategory consisting of $P$-(hyper)constructible (hyper)sheaves.
\end{defi}
\begin{rem}
	Let $X \to P$ be a conically stratified space with locally contractible strata. We have the following tensor decomposition (\cite[Theorem 3 (3)]{porta2022topological}):
	\[
	\Sh_P^{\on{hyp}}(X;\D) \simeq \Sh_P^{\on{hyp}}(X;\mathcal{S})\otimes \D.
	\]
	So that it is compatible with the tensor decomposition of \ref{lem:tensorsheaves}.
\end{rem}
\begin{prop}[Exodromy, {\cite{porta2022topological}}]
	\label{prop:exodromy}
	Let $X \to P$ be a conical $P$-stratification with strata locally contractible strata $X_p$, $p \in P$ and let $\D$ be a stable and presentable $\infty$-category. Then we have an equivalence of $\infty$-categories
	\[
	\Sh^{\on{hyp}}_P(X;\D) \simeq \Fun(\Exit_P(X),\D),
	\]
	where $\Exit_P(X)$ denotes the exit path $\infty$-category of \cite[A.9]{lurie:ha}.
\end{prop}
\begin{prop}[Stability under coarsening, {\cite{haine2026exodromy}}]
	\label{prop:stabilityundercoarsening}
	Given a stratified space $X \to P$ as in \ref{prop:exodromy}, and a map of posets $P \to Q$, such that the stratification $X \to Q$ satisfied the requirement of \ref{prop:exodromy}. Denote $\W$ the class of morphisms in $\Exit_P(X)$ which are sent to equivalences under $\Exit_P(X) \to P \to Q$. Then we have an equivalence of $\infty$-categories 
	\[
	\Exit_Q(X) \simeq \Exit_P(X)[\W^{-1}].
	\]
\end{prop}

\subsection{$t$-structures}
By \ref{lem:tensorsheaves} we have an equivalence
\[
\Sh(X;\D) \simeq \Sh(X;\mathcal{S}) \otimes \D.
\]
\begin{prop}[{\cite[Proposition 3.8]{haine2025standard}}]
	\label{prop: haineinducedtstructure}
	Given an accessible left complete $t$-structure $(\D^{\geq 0},\D^{\leq 0})$ on $\D$ and $\D^{\geq 0}$ closed under filtered limits, then $\Sh(X;\D)$ (resp.\ $\Sh^{\on{hyp}}(X;\D)$) carries a canonical $t$-structure, given by the pair
	\[
	\left(\Sh(X,\mathcal{S}) \otimes \D^{\geq 0},\Sh(X;\mathcal{S}) \otimes \D^{\leq 0}\right),
	\]
	resp.
	\[
	\left(\Sh^{\on{hyp}}(X,\mathcal{S}) \otimes \D^{\geq 0},\Sh^{\on{hyp}}(X;\mathcal{S}) \otimes \D^{\leq 0}\right).
	\]
\end{prop}

\begin{notation}
	In order to avoid clumsy notation we denote the induced $t$-structure on $\Sh(X;\D)$ (resp.\ $\Sh^{\on{hyp}}(X;\D)$) given by \ref{prop: haineinducedtstructure} 
	\[
	\left(D^{\geq 0},D^{\leq 0}\right).
	\]
\end{notation}

\begin{rem}
	In this work we use the cohomological grading conventions. We reformulate the cited statements (which use homological convention, especially \cite{lurie:ha} and \cite{haine2025standard}) accordingly without further mention. In order to translate the two grading conventions, we observe that given a $t$-structure $(D^{\geq 0},D^{\leq 0})$ on $D$ in cohomological grading convention, then on $D^\op$ we have the homologically graded $t$-structure $((D^{\leq 0})^\op,(D^{\geq 0})^\op)$.
\end{rem}

\begin{exa}[standard $t$-structure]
	\label{exa:tstructureonderivedcategory}
	Let $R$ be a connective $\mathbb{E}^1$-ring and denote $\D= \on{LMod}_R$ the category of left $R$-module as introduced in \cite[\S 7]{lurie:ha}. Recall $\D$ carries a $t$-structure (\cite[7.1.1.13]{lurie:ha}) satisfying the assumptions of \ref{prop: haineinducedtstructure}, hence we obtain a $t$-structure on $\Sh(X;\D)$ with (co)connective objects given as above. Suppose $R$ is a connective discrete $\mathbb{E}^1$-ring and let $\A= \Mod_{\pi_0(R)}$ be the abelian category of modules the ordinary ring $\pi_0(R)$, then $\on{LMod_R} \simeq \D(\A)$ accoring to \cite[7.1.1.16]{lurie:ha}.
\end{exa}
In the following lemma we recall basic properties of the standard $t$-structure induced on $\Sh(X;\on{LMod}_R)$
\begin{lem}[{\cite[$\S 3$]{haine2026derived}}]
	Denote $f:X \to Y$ a map of topological space, denote $\Sh^{\on{hyp}}(A):= \Sh^{\on{hyp}}(A;\on{LMod}_R)$ for $R$ as in \ref{exa:tstructureonderivedcategory} and $A$ some topological space, then
	\begin{enumerate}
		\item $f^\ast: \Sh^{\on{hyp}}(Y) \to \Sh^{\on{hyp}}(X)$ is $t$-exact,
		\item $f_\ast: \Sh^{\on{hyp}}(X) \to \Sh^{\on{hyp}}(Y)$ is left $t$-exact (i.e.\ it preserves $D^{\geq 0}$),
		\item If $f$ is an open immersion, then $f_!:\Sh^{\on{hyp}}(X) \to \Sh^{\on{hyp}}(Y)$ is right $t$-exact,
		\item If $f$ is closed immersion, then $f^!:\Sh^{\on{hyp}}(Y) \to \Sh^{\on{hyp}}(X)$ is left $t$-exact and $f_\ast:\Sh^{\on{hyp}}(X) \to \Sh^{\on{hyp}}(Y)$ is $t$-exact,
	\end{enumerate}
	where exactness is always with respect to the standard $t$-structure.
\end{lem}
\begin{cor}[{\cite[3.1.4]{haine2026derived}}]
	Let $X$ be a topological space, then $F \in \Sh^{\on{hyp}}(X;{\on{LMod}_R})$ is (co)connective with respect to the standard $t$-structure if and only if $F_x \in \on{LMod}_R$ is (co)connective for all $x \in X$.
\end{cor}
\begin{prop}[Perverse $t$-structure]\label{prop: perversetstructure}
	Given a accessible left complete $t$-structure $(\D^{\geq 0},\D^{\leq 0})$ on a stable and presentable $\infty$-category $\D$. Given a stratified space $X \to P$ together with a perversity function 
	\[
	\p: P \to \ZZ
	\] 
	then there exists a $t$-strucutre on $\Sh_P(X;\D)$, called the perverse $t$-structure denoted $({}^\p D^{\geq 0},{}^\p D^{\leq 0})$, which is characterized by
	\begin{enumerate}
		\item[$(P^+)$] $F \in {}^\p D^{\geq 0}$ if and only if $i_p^!F \in D^{\geq \p(p)}$ for all $p \in P$,
		\item[$(P^-)$] $F \in {}^\p D^{\leq 0}$ if and only if $i_p^\ast F \in D^{\leq \p(p)}$ for all $p \in P$.
	\end{enumerate}
\end{prop}
\begin{proof}
	Consider the small collection of elements $F \in \Sh_P(X;\D)$ such that $i_p^\ast F \in D^{\leq \p(p)}$ for all $p \in P$. By \cite[Proposition 1.4.4.11.]{lurie:ha} there exists a $t$-structure whose coconnective objects are given by those $F$, such that $i_p^\ast F \in D^{\leq \p(p)}$ for all $p \in P$.
\end{proof}
\subsection{Sheaves and hypersheaves}\label{sub:comparison}
In this section we briefly discuss the two different notions of sheaves and hypersheaves and compare them. We formulate the setting which subsumes all the examples of \S \ref{sec:examples_of_interest}.
\begin{lem}[Comparison sheaves and hypersheaves]\label{lem:compsheaveshypersheaves}
	Let $X$ be a paracompact topological space with finite covering dimension. Then the inclusion
	\[
	\Sh^{\on{hyp}}(X;\mathcal{S}) \hookrightarrow \Sh(X;\mathcal{S})
	\]
	is an equivalence of $\infty$-categories.
\end{lem}
\begin{proof}
	This follows from \cite[7.2.3.6]{lurie:htt} together with \cite[7.2.1.12]{lurie:htt} yielding that $\Sh(X;\mathcal{S})$ is already hypercomplete making the inclusion
	\[
	\Sh^{\on{hyp}}(X;\mathcal{S}) \hookrightarrow \Sh(X;\mathcal{S})
	\]
	an equivalence of $\infty$-categories.
\end{proof}
\begin{rem}\label{rem: comparisonshandhypsh}
	Many statements in the literature are proven in the setting of hypersheaves. Lemma \ref{lem:compsheaveshypersheaves} implies the validity of these results in our setup of $\D$-valued sheaves on paracompact topological space with finite covering dimension. More explicitly we have a chain of equivalences for $X$ satisfying the conclusion of \ref{lem:compsheaveshypersheaves} (e.g.\ $X$ paracompact, finite covering dimension):
	\[
	\Sh^{\on{hyp}}(X;\D) \simeq \Sh^{\on{hyp}}(X;\mathcal{S}) \otimes \D \xrightarrow{\simeq} \Sh(X;\mathcal{S})\otimes \D \simeq \Sh(X;\D). 
	\]
\end{rem}

\begin{rem}\label{rem: exodromyforsheaves}
	Let $X \to P$ be a conically stratified paracompact space with finite covering dimension, such that all strata are locally contractible. In this case Remark \ref{rem: comparisonshandhypsh} yields that the statements of Proposition \ref{prop:exodromy} and \ref{prop:stabilityundercoarsening} hold for $\D$-valued sheaves, i.e.\ we obtain an equivalence of $\infty$-categories
	\[
	\Sh_P(X;\D) \simeq \Fun(\Exit_P(X),\D).
	\]
\end{rem}
The following exhausts all examples that are covered in \S \ref{sec:examples_of_interest}:
\section{Lax limit cones}
\label{sub:laxlimcones}

We develop our general approach to lax diagrams and lax limit cones following
along the lines of the representative example from \S \ref{sub:quiver}. Let $K$
be a simplicial set and let $\pi: \G \to K^{\rhd}$ be a locally Cartesian
fibration. 

The straightening of such a fibration yields, in general, a {\em laxly}
commutative diagram $K^{\rhd} \to \Cat_\infty$, and we wish to formulate the
condition that this diagram forms a lax limit cone directly in terms of the
fibration $\pi$.
In the context of {\em commutative} diagrams of $\infty$-categories, (op)lax
limits can be simply modelled by sections of the respective (co)Cartesian
fibrations. For laxly commutative diagrams we need a slight elaboration (as is
already apparent from \S \ref{sub:quiver}), based on barycentric subdivision:
Given a finite nonempty linearly ordered set $I$, we denote by $\P^*(I)$ the
poset of nonempty subsets of $I$. We denote by $\min: \P^*(I)^{\op} \to I$ the
order--preserving map given by assigning to a subset $J \subset I$ its minimum
$\min(J)$. 

We then define the functor
\[
    \sd: \Delta \lra \sSet^+, \; [n] \mapsto \N(\P^*([n])^{\op})
\]
where we mark all edges of the form 
\[
\begin{tikzcd}
	{[k]} \arrow[rd, "\phi"'] \arrow[rr, "i \mapsto i+1"] &       & {[k+1]} \arrow[ld, "\psi"] \\
	& {[n]} &                           
\end{tikzcd}
\]
and think of elements in $\P^*([n])$ as elements in $\Delta/[n]$.
We further note that the function $\min$ defines a natural transformation 
\[
\min : \sd \Rightarrow Y
\] 
where $Y: \Delta \to \sSet^+, [n] \mapsto (\Delta^n)^{\sharp}$ denotes the
(fully marked) Yoneda embedding. Notice that the marked edges are precisely the
locally Cartesian edges with respect to the maps $\min : \sd([n]) \to [n]$ for
all $n$. Via left Kan extension, we obtain a functor $\sd: \sSet \to \sSet^+$
together with a natural transformation $\sd \Rightarrow (-)^{\sharp}$ which we
still denote by $\min$. 
Finally we can extend this map by adding a cone tip and receive
\[
    \operatorname{min}^\rhd: \sd(-)^\rhd \to (({-})^\rhd)^\sharp.
\] 
The marked edges of $\sd(K)^\rhd$, for $K \in \sSet$ are given by conserving the ones for $\sd(K)$ and adding: for $\phi :[0] \to K$ we mark the maps $\phi \to *$, where $*$ denotes the cone tip.

We now define the $\infty$-categories
\[
    \Gamma(\G) := \Map_{K^{\rhd}}^{\flat}(\sd(K)^\rhd, \G^{\natural})
\]
and
\[
    \Gamma^*(\G) := \Map_{K}^{\flat}(\sd(K), \G^{\natural}|K)
\]
given by functors over $K^{\rhd}$ preserving marked edges. 

\begin{lem}\label{lem:tipprojection}
    The projection 
    \[
        \Gamma(\G) \to \G_*
    \]
    is a trivial Kan fibration. 
\end{lem}
 
To prepare the proof of the lemma, we need to establish some routine maneuvers
within the saturated class of maps which are anodyne with respect to locally
Cartesian fibrations. 
Let $A$,$B$ be marked simplicial sets and $S$ a simplicial set. Given a map $p:
B \to S^{\sharp}$, we say that an inclusion $i: A \to B$ is {\em marked local anodyone
over $S$} if, for every locally Cartesian fibration $q: X \to S$, every lifting
problem
\[
\begin{tikzcd}
    A \ar{r}\ar{d}{i} & X^{\natural}\ar{d}{p} \\
    B \ar{r}{p}\ar[dashed]{ur} & S^{\sharp}
\end{tikzcd}
\]
admits a solution.

\begin{prop} Let $S$ be a simplicial set. 
    \label{prop:lma}
    \begin{enumerate}
        \item\label{lma:1} Suppose $A \to B$ with $p: B \to S^{\sharp}$ is marked local
            anodyne over $S$ and let $t: S \to T$ be a map of simplicial
            sets. Then $A \to B$ is marked local anodyne over $T$ via the map
            $t \circ p: B \to T^{\sharp}$. 
        \item\label{lma:2} Let $n \ge 0$ and $0 < i < n$ and let $\Delta^n \to S$ be any
            map. Then the inner horn inclusion $(\Lambda^n_i)^{\flat} \to
            (\Delta^n)^{\flat}$ is marked local anodyne over $S$.
        \item\label{lma:3} For $n \ge 1$, consider the map $p: \Delta^n \to \Delta^1$ defined as the nerve
            of the poset map $[n] \to [1]$ mapping $0,1,...,n-1$ to $0$ and $n$
            to $1$. Then the right horn inclusion $\Lambda^n_n \to \Delta^n$,
            where the only nondegenerate edge marked in both simplicial sets is
            $\{n-1,n\}$, is marked local anodyne over $\Delta^1$. 
        \item\label{lma:4} Let $n \ge 0$. Let $p: \Delta^1 \times \Delta^n \to
            (\Delta^n)^{\lhd}$ be the nerve of the map of posets
            which collapses $\{0\} \times [n]$ to the initial vertex and
            restricts to the projection onto $[n]$ on $\{1\} \times [n]$. Then
            the smash product
            \begin{equation}\label{eq:smash4}
                \{1\} \times (\Delta^n)^{\flat} \coprod_{\{1\} \times \partial
                \Delta^n} (\Delta^1)^{\sharp} \times (\partial
                \Delta^n)^{\flat} \to  (\Delta^1)^{\sharp} \times (\Delta^n)^{\flat}
            \end{equation}
            of the maps $\{1\} \to (\Delta^1)^{\sharp}$ and $(\partial
            \Delta^n)^{\flat} \to (\Delta^n)^{\flat}$ is marked local anodyne
            over $(\Delta^n)^{\lhd}$. 
        \item\label{lma:5} Let $n \ge 0$. Let $A \to B$ be any inclusion of simplicial sets
            and let $\Delta^1 \times B \to B^{\lhd}$ be the map collapsing
            $\{0\} \times B$ to the initial vertex and restricting to the
            projection onto $B$ on $\{1\} \times B$. Then the smash product
            \[
                \{1\} \times B^{\flat} \coprod_{\{1\} \times A^{\flat}}
                (\Delta^1)^{\sharp} \times A^{\flat} \to  (\Delta^1)^{\sharp}
                \times B^{\flat}
            \]
            of the maps $\{1\} \to (\Delta^1)^{\sharp}$ and $A^{\flat} \to B^{\flat}$ is marked local anodyne
            over $B^{\lhd}$. In particular, by \ref{lma:1}, composing
            with any map $B^{\lhd} \to S$, we obtain a marked local
            anodyne morphism over $S$.
    \end{enumerate}
\end{prop}
\begin{proof}
    This is completely analogous to the proof of \cite[3.1.1.5]{lurie:htt}.
    The only novum is the extra constraint in \ref{lma:3} that right horn
    fillings can only be achieved in situations when the $n$-simplex $\Delta^n$
    is collapsed to an edge $\Delta^1$. Due to our setup, it is guaranteed that
    the right horn fillings needed in \cite[3.1.1.5]{lurie:htt} to deduce \ref{lma:4}
    from \ref{lma:2} and \ref{lma:3} do satisfy this constraint. Statement
    \ref{lma:5} is then derived from \ref{lma:4} ``simplex by simplex'' as in
    loc. cit. 
\end{proof}

\begin{cor}
    \label{cor:bary_anodyne}
    \begin{enumerate}
        \item\label{bary_anodyne:1} Let $n \ge 0$. Consider the map $\min:
            \sd(\Delta^n)^{\rhd} \to (\Delta^n)^{\rhd}$. Then the inclusion
        \[
            \sd(\partial \Delta^n)^{\rhd} \to \sd(\Delta^n)^{\rhd}
        \]
        lies in the saturated hull of the morphisms of type \ref{lma:5}. In particular, it is marked local anodyne.
    \item\label{bary_anodyne:2} Let $A \subset B$ be simplicial sets. 
        Then the map 
        \[
            \sd(A)^{\rhd} \to \sd(B)^{\rhd}
        \]
        lies in the saturated hull of the morphisms of type \ref{lma:5}. In particular, it is marked local anodyne.
    \end{enumerate}
\end{cor}
\begin{proof}
    For the proof of \ref{bary_anodyne:1}, note that the case $n=0$ corresponds
    to the generating anodyne map from \ref{lma:3} (with $n=1$ there). For $n
    \ge 1$, let $B$ be the nerve of the poset $\P(\{1,2,...,n\})^{\op}$ of subsets of
    $\{1,2,...,n\}$ and $A \subset B$ the nerve of $\P(\{2,...,n\})^{\op}$ (for $n =
    1$, $\{2,...,n\}$ is the empty set). Then on underlying simplicial sets,
    the smash product 
    \begin{equation}
        \label{eq:smash}
        A' = \{1\} \times B^{\flat} \coprod_{\{1\} \times A^{\flat}}
        (\Delta^1)^{\sharp} \times A^{\flat} \to  (\Delta^1)^{\sharp}
        \times B^{\flat} = B'
    \end{equation}
    of $A^{\flat} \subset B^{\flat}$ with $\{1\} \subset (\Delta^1)^{\sharp}$
    can be identified with 
    \begin{equation}
        \label{eq:desired}
        \sd(\partial \Delta^n)^{\rhd} \to \sd(\Delta^n)^{\rhd}
    \end{equation}
    so that $A' \to B'$ is marked local anodyne over $(\Delta^n)^{\rhd} =
    B^{\lhd}$ by Proposition \ref{prop:lma} \ref{lma:5}.
    To get the desired markings ($A'$ has fewer marked edges than $\sd(\partial
    \Delta^n)^{\rhd}$), we simply form a pushout of $A' \to B'$ along the
    inclusion $A' \to \sd(\partial \Delta^n)^{\rhd}$ which reproduces
    \eqref{eq:desired} with the {\em correct} marking, therefore showing that
    this latter map is also marked local anodyne.

    \ref{bary_anodyne:2} is then obtained by a standard ``simplex-by-simplex''
    argument, succesively adjoining barycenters via pushouts along the
    morphisms from \ref{bary_anodyne:1}.
\end{proof}

\begin{proof}[Proof of Lemma \ref{lem:tipprojection}]
    We set $A = \emptyset$ and $B = K$. Then the map $\Gamma(\G) \to \G_*$ can be
    identified with the restriction 
    \begin{equation}\label{eq:resanodyne}
        \Map^{\flat}_{K^{\rhd}}(\sd(K)^{\rhd}, \G) \to \Map^{\flat}_{K^{\rhd}}(\sd(\emptyset)^{\rhd}, \G)
    \end{equation}
    along the map $i: \sd(\emptyset)^{\rhd} \subset \sd(K)^{\rhd}$. To show that
    \eqref{eq:resanodyne} is a trivial Kan fibration it suffices to show that
    it has the right lifting property with respect to all boundary inclusions
    $j: \partial \Delta^n \subset \Delta^n$, $n \ge 0$. To this end, it
    suffices to show that the smash product $i \wedge j$ is marked local
    anodyne. But this follows, since $i$ lies within the saturated hull of
    morphisms of type \ref{lma:5}, which are stable under smash products, so
    that $i \wedge j$ is in the saturated hull of morphisms of type
    \ref{lma:5} as well, in particular marked local anodyne.
\end{proof}

\begin{defi}
    \label{defi:lax}
    We say that the locally Cartesian fibration $\pi: \G \to K^{\rhd}$ is a lax
    limit cone, exhibitting $\G_*$ as a lax limit of $\G|K$, if the restriction
    map 
    \[
        \Gamma(\G) \to \Gamma^*(\G)
    \]
    is an equivalence.
\end{defi}

\begin{rem}
    \label{rem:}
    Note that the condition of Definition \ref{defi:lax} is equivalent to
    condition that, for any section of the trivial Kan fibration $\Gamma(\G) \to
    \G_*$ (cf. Lemma \ref{lem:tipprojection}), the functor
    \[
        \G_* \to \Gamma^*(\G)
    \]
    obtained by composing the section with the projection $\Gamma(\G) \to
    \Gamma^*(\G)$ is an equivalence. 
\end{rem}

\begin{rem}
    \label{rem:ordinary_explicit} 
    In general, it is not feasible to explicitly extract all the
    higher coherence data comprised in the $\infty$-category $\Gamma^*(\G)$ in
    the same fashion as done in Remark \ref{rem:informal}. However, when $K$ is
    (the nerve of) an ordinary category and the fibers of $\pi$ are (nerves
    of) {\em ordinary} categories, then this can be easily done. In this case,
    the category $\Gamma^*(\G)$ is the category comprised by the following
    data: 
    \begin{itemize}
        \item for every object $k \in K$, an object $x_k \in \G_k$,
        \item for every morphism $k \to k'$ in $K$, a morphism $f_{k \to k'}: F_{k\to
            k'}(x_{k'}) \to x_k$,
        \item such that, for every composite $k \to k' \to k''$ in $K$, the square
            \[
                \begin{tikzcd}[column sep = 7em]
                    F_{k\to k'} F_{k'\to k''}(x_k'') \ar{r}{F_{k\to k'}(f_{k' \to k''})} \ar[swap]{d}{\alpha_{x_k''}} & F_{k\to k'}(x_k')\ar{d}{f_{k \to k'}}\\
                F_{k\to k''}(x_k'') \ar{r}{f_{k \to k''}} & x_k
            \end{tikzcd}
            \]
            commutes.
    \end{itemize}
    Here $F_{k\to k'}:\G_{k'} \to \G_k$ and $\alpha:  F_{k\to k'} F_{k'\to k''}
    \Rightarrow F_{k\to k''}$, respectively, denote the functors and natural
    transformations, respectively, modelled by the locally Cartesian fibration
    $\pi$. In the situation where $\pi$ is globally Cartesian, then the
    transformations $\alpha$ become equivalences so that the description
    simplifies even further. We will apply this explicit description in a situation where we start with
    a fibration $\pi$ models a diagram of stable $\infty$-categories and then
    show that the transition functors are $t$-exact with respect to chosen
    $t$-structures. Then, the lax limit of the resulting diagram of hearts
    allows for precisely the above explicit description. 
\end{rem}

\section{Recollements}
\label{sec:recollements}

The theory of descent, based on the idea that one can glue a sheaf from its
restrictions to the open subsets of an open cover, is a standard tool to
describe local-to-global phenomena in geometry and topology. Another kind of
``gluing'' is captured by the concept of a recollement which relates sheaves on
an given space to sheaves on an open subset and its closed complement. In the
original form, as introduced in \cite{BBD}, it is not
possible to truly glue the respective sheaf categories on the open and closed
part. The theory of $(\infty,2)$-categories offers a context in which this can
be done in a very elegant way: For any topological space $X$ and any open
subset $U \subset X$, we have an equivalence
\[
    \Sh(X,\D) \simeq \llim (\Sh(U,\D) \overset{i^!j_!}{\lra} \Sh(X\setminus U,\D))
\]
exhibiting the $\infty$-category of sheaves on $X$ as a lax limit over the {\em
gluing functor} $i^!j_!$. In this section, we review some basics of
recollements needed to establish this gluing statement.

\subsection{Classical recollements}%
\label{sub:recollements}

Recall that a recollement of a triangulated category $\X$ consists of full
triangulated subcategories $\A$ and $\B$ along with functors
\begin{equation}\label{eq:classical_recollement}
    \begin{tikzcd}
        \A \ar{r}{i_*} &\ar[shift left=1em,swap]{l}{i^!} \ar[shift right=1em, swap]{l}{i^*} \X \ar{r}{j^*} & \ar[shift left=1em,swap]{l}{j_*} \ar[shift right=1em, swap]{l}{j_!} \B
    \end{tikzcd}
\end{equation}
such that 
\begin{enumerate}
    \item there are adjunctions $i_* \dashv i^* \dashv i^!$ and $j_! \dashv j^* \dashv j_*$, 
    \item the composite $j^* \circ i_*$ is zero,
    \item every object $x \in \X$ admits distinguished triangles 
        \begin{align}
            \label{ext:1}
            j_! j^* x \lra x \lra i_*i^* x \overset{+1}{\lra}\\
            \label{ext:2}
            i_* i^! x \lra x \lra j_*j^* x \overset{+1}{\lra}
        \end{align}
\end{enumerate}
Using the connecting morphism of \ref{ext:1} along with the unit of the
adjunction $i_* \dashv i^!$, we can form the oplax commutative triangle 
\begin{equation}
    \label{eq:oplax}
\begin{tikzcd}[column sep=small]
    & \ar[swap]{dl}{j^*} \ar[dr,"i^*"{name=U}] \X   & \\
    \B \ar[swap]{rr}{F} \ar[to=U, shorten >= 10pt, shorten <= 10pt, Leftarrow] & &  \A  
\end{tikzcd}
\end{equation}
where $F = i^! \circ j_![1]$. Similarly, using the connecting morphism of \ref{ext:2}
along with the counit of the adjunction $i^* \dashv i_*$, we obtain the lax
commutative triangle
\begin{equation}
    \label{eq:lax}
\begin{tikzcd}[column sep=small]
    & \ar[swap]{dl}{j^*} \ar[dr,"i^!"{name=U}] \X   & \\
    \B \ar[swap]{rr}{G}\ar[to=U, shorten >= 10pt, shorten <= 10pt, Rightarrow] & &  \A 
\end{tikzcd}
\end{equation}
with $G = i^* \circ j_*[-1]$.

It will be of central importance to us to work in a context where the triangles
\eqref{eq:oplax} and \eqref{eq:lax} are universal cones, allowing us to recover
$\X$ as (op)lax limits of the ``gluing functors'' $F$ and $G$. To this
end, we need to enhance the triangulated categories suitably -- there are
several ways to achieve this, the most convenient for our purposes will be the
framework of stable $\infty$-categories due to Lurie. The relevant aspects of
the resulting theory of semiorthogonal decompositions and recollements will be
developed in the remainder of this section.

\subsection{Semiorthogonal decompositions}%
\label{sub:sods}

We review some basic aspects of semiorthogonal decompositions and recollements
within the context of stable $\infty$-categories. Most statements are discussed
in more detail in \cite{DKS:spherical}. 

Let $\X$ be a stable $\infty$-category equipped with full stable subcategories
$\A$ and $\B$. We denote by $\on{Ext}(\A,\B)$ the $\infty$-category of exact sequences
\begin{equation}\label{eq:ext}
    \begin{tikzcd}
          b \ar[r]\ar[d] & x\ar[d]\\
          0 \ar[r] & a
    \end{tikzcd}
\end{equation}
in $\X$ with $a \in \A$, $b \in \B$. The ordered pair $(\A,\B)$ is called a
{\em semiorthogonal decomposition of $\X$} if the forgetful functor
$\Ext(\A,\B) \to \X$, associating a sequence as in \eqref{eq:ext} to $x$,
is an equivalence. A bit more informally, this condition amounts to saying that
every object in $\X$ can be expressed in an essentially unique way as an
extension of an object in $\A$ by an object in $\B$.

\begin{rem}
    \label{rem:dual} Note that if $(\A,\B)$ forms a semiorthogonal
    decomposition of $\X$ then $(\B^{\op},\A^{\op})$ forms a semiorthogonal
    decomposition of $\X^{\op}$.
\end{rem}

\begin{lem}
    \label{lem:orthogonal}
    Let $(\A,\B)$ be a semiorthogonal decomposition of $\X$. Then the following
    are equivalent for an object $x \in X$:
    \begin{enumerate}[label=(\roman *)]
        \item\label{it:orth1} $x \in \B$.
        \item\label{it:orth2} Every morphism $x \to a$ in $\X$ with $a \in \A$ is homotopic to
            the zero morphism.
        \item\label{it:orth3} For every object $a \in \A$, the mapping space $\Map_{\X}(x,a)$
            is contractible. 
    \end{enumerate}
\end{lem}
\begin{proof}
    Suppose $x \in \B$ and $a \in \A$. Via the equivalence $\Ext(\A,\B) \to
    \X$, we may identify any morphism $x \to a$ up to equivalence with a
    morphism between the exact sequences $x \to x \to 0$ and $0 \to a \to a$.
    But any such a morphism is evidently homotopic to zero, showing
    \ref{it:orth2}. The equivalence of \ref{it:orth2} and \ref{it:orth3}
    follows from $\Hom_{\X}(x[i],a) \cong \pi_i(\Map_{\X}(x,a),0)$ (cf.
    \cite[1.1.2]{lurie:ha}). We conclude by showing \ref{it:orth2} implies
    \ref{it:orth1}: For a given $x \in \X$ there is an exact sequence $b \to x
    \to a$ of the form \eqref{eq:ext}. Then $x \to a$ being homotopic to the
    zero map implies $b \simeq x \oplus a[-1]$. But this further implies $a
    \simeq 0$, since otherwise $b$ would have a nonzero projection map to
    $a[-1]$, contradicting the implication \ref{it:orth1} $\Rightarrow$
    \ref{it:orth2}. 
\end{proof}

It is convenient for what follows to express various functors and their
adjointness relations in terms of universal arrows in $\X$: 

\begin{defi}
    \label{defi:universal}
    Let $\eT$ be an $\infty$-category and $\eR$, $\eS$ full subcategories of $\eT$.
    A morphism $r \to s$ in $\eT$ with $r \in \eR$ and $s \in \eS$ is called 
    \begin{enumerate}
        \item {\em $(\eR,\eS)$-Cartesian} if it is final in the slice category
            $\eR \times_{\eT} \eT_{/s}$,
        \item {\em $(\eR,\eS)$-coCartesian} if it is initial in the slice category
            $\eS \times_{\eT} \eT_{r/}$.
    \end{enumerate}
    We will typically indicate that a given edge is Cartesian (resp.
    coCartesian), in the above sense, by decorating it with ``$*$'' (resp.
    ``$!$'').
\end{defi}

\begin{rem}
    \label{rem:grothendieck} 
    The ``absolute'' concepts of Cartesian (resp. coCartesian) morphisms from
    Definition \ref{defi:universal} formulated within the category $\X$
    immediately translate into the familiar ``relative'' concepts from
    \cite{lurie:htt} as follows. Let $\G(\eR,\eS)$ denote the full subcategory
    of $\eT \times \Delta^1$ spanned by the objects of $\eR \times \{0\}$ and
    $\eS \times \{1\}$ equipped with the projection map to $\pi:\G(\eR,\eS) \to
    \Delta^1$. Then the $(\eR,\eS)$-Cartesian (resp. coCartesian) edges in $\X$
    correspond precisely to the $\pi$-Cartesian (resp. coCartesian) edges over
    the edge $0 \to 1$. We may thus freely apply results from the familiar
    relative to the absolute context. As demonstrated below, the formulation of
    the respective universal properties directly within $\X$ without passing to
    specific Grothendieck constructions $\G(\eR,\eS)$ gives us more flexibility
    when simultaneously considering the various functors and adjunctions
    comprising a recollement. 
\end{rem}

\begin{lem}
    \label{lem:sodcartesian}
    Let $(\A,\B)$ be a semiorthogonal decomposition of $\X$. Then the following hold:
    \begin{enumerate}
        \item A morphism $x \to a$ in $\X$ with $a \in \A$ is
            $(\X,\A)$-coCartesian if and only if its fiber (or equivalently
            cofiber) lies in $\B$.
        \item A morphism $b \to x$ in $\X$ with $b \in \B$ is
            $(\B,\X)$-Cartesian if and only if its cofiber (or equivalently
            fiber) lies in $\A$.
    \end{enumerate}
\end{lem}
\begin{proof}
    We show the first statement, the second is dual (cf. Remark
    \ref{rem:dual}). We extend the morphism $f: x \to a$ to an exact sequence
    $\fib(f) \to x \to a$. Then, for any morphism $g: x \to a'$ with $a' in
    \A$, the space of morphisms between $f$ and $g$ in $\A \times_{\X} \X_{x/}$
    can be identified with the space of zero homotopies of the composite
    $\fib(f) \to x \to a'$. If $\fib(f) \in \B$ then this space is contractible
    showing that $f$ is $(\X,\A)$-coCartesian. Vice versa, if $f$ is
    $(\X,\A)$-coCartesian, then the space of morphisms in $\X$ from the cofiber
    $\cof(f)$ to any given $a' \in \A$ can be identified with the space of
    morphisms in $\A \times_{\X} \X_{x/}$ between $f$ and the zero morphism $x
    \to a'$. The latter being contractible by assumption implies that $\cof(f)$
    (and hence $\fib(f) = \cof(f)[-1])$) lies in $\A$ by Lemma
    \ref{lem:orthogonal}.
\end{proof}

\begin{rem}
    \label{rem:exactcartesian}
    Note that, as an immediate consequence of Lemma \ref{lem:sodcartesian}, we
    obtain that for every exact sequence 
    \begin{equation}\label{eq:exactcartesian}
        \begin{tikzcd}
            b \ar{r}{*}\ar[d] & x\ar{d}{!}\\
            0 \ar{r} & a
        \end{tikzcd}
    \end{equation}
    in $\Ext(\A,\B)$, the morphism $b \to x$ is in fact $(\B,\X)$-Cartesian and the
    morphism $x \to a$ is $(\X,\A)$-coCartesian. In particular, in view of Remark
    \ref{rem:grothendieck}, this implies the existence of a left adjoint to the
    inclusion functor $\A \subset \X$ and, dually, a left adjoint to the inclusion
    $\B \subset \X$. Indeed, the fibration $\G(\X,\A) \to \Delta^1$ is biCartesian
    where the Cartesian fibration models the inclusion functor $\A \subset \X$ and
    the coCartesian fibration models its right adjoint (cf.
    \cite[5.2.2]{lurie:htt}).
\end{rem}

\begin{defi}
    \label{defi:admissible} Let $\eT$ be a stable $\infty$-category. A full
    stable subcategory $\eR$ is called {\em left} (resp. {\em right}) {\em
    admissible} if the inclusion $\eR \subset \eT$ admits a left (resp. right)
    adjoint. 
\end{defi}

\begin{rem}
    \label{rem:admissible}
    By Lemma \ref{lem:orthogonal} (and its dual), the categories $\A$ and $\B$
    of a semiorthogonal decomposition determine one another and, using variants
    of the arguments and results above, it is straightforward to see that the
    following data are equivalent for a given stable $\infty$-category $\X$:
    \begin{enumerate}[label= (\roman *)]
        \item a semiorthogonal decomposition $(\A,\B)$ for $\X$,
        \item a left admissible subcategory $\A \subset \X$,
        \item a right admissible subcategory $\B \subset \X$.
    \end{enumerate}
\end{rem}

Note that at this point, we have arrived at an $\infty$-categorical description
of the part 
\[
    \begin{tikzcd}
        \A \ar{r}{i_*} & \ar[shift right=1em, swap]{l}{i^*} \X \ar{r}{j^*} & \ar[shift right=1em, swap]{l}{j_!} \B
    \end{tikzcd}
\]
of the recollement diagram \eqref{eq:classical_recollement}. Also note that the canonical exact triangle 
\[
    b \overset{*}{\to} x  \overset{!}{\to} a
\]
from Remark \ref{rem:exactcartesian}, associated to every object $x \in \X$,
corresponds to the distinguished triangle \eqref{ext:1}. It is thus apparent
that semiorthogonal decompositions are slightly weaker structures than
recollements, as to be discussed next. 

\subsection{Recollements as (op)lax limits}%
\label{sub:recollements_as_op_lax_limits}

The following Proposition is straightforward to show using the techniques from \S
\ref{sub:sods}.

\begin{prop}
    \label{prop:rec}
    Let $\X$ be a stable $\infty$-category and let $(\A,\B)$ be a
    semiorthogonal decomposition of $\X$. Then the following are equivalent:
    \begin{enumerate}[label=(\roman *)]
        \item The inclusion functor $\A \subset \X$ admits a right adjoint. 
        \item For every $x \in \X$, there exists an $(\A,\X)$-Cartesian edge $a
            \to x$.
        \item For every $b \in \B$, there exists an $(\A,\B)$-Cartesian edge $a
            \to b$.
        \item There exists a semiorthogonal decomposition $(\B',\A)$ of $\X$.
    \end{enumerate}
\end{prop}

\begin{defi}
    \label{defi:rec}
    A semiorthogonal decomposition $(\A,\B)$ of $\X$ is called {\em
    recollement} if the equivalent conditions of Proposition \ref{prop:rec}
    hold.
\end{defi}

\begin{rem}
    \label{rem:orth}
    The subcategory $\B' \subset \X$, uniquely determined as the complement of
    $\A$ by Lemma \ref{lem:orthogonal}, corresponds, in terms of the diagram
    \ref{eq:classical_recollement} to the image of $\B$ under the functor
    $j_*$. 
\end{rem}

The following Proposition provides a very convenient ``portrait'' relating all
pieces of the recollement in terms of universal properties (in contrast to
spelling out the various functors and their adjunction data). 

\begin{prop}
    \label{prop:grid}
    Let $(\A,\B)$ be a recollement of $\X$. Then the forgetful maps projecting
    to $x$ (or $x'$), provide an equivalence of $\X$ with the $\infty$-category
    $\Rec(\A,\B)$ of diagrams of the form
    \begin{equation}
        \label{eq:grid}
        \begin{tikzcd}
            x \ar{r}{!} \ar{d}{!} & b' \ar{d}{!}\ar{r}& \ar{d}0 \\
            a \ar{r}\ar{d} & a' \ar{r}{*}\ar{d}&  b\ar{d}{*}\\
            0 \ar{r} & a'' \ar{r}{*} & x'
        \end{tikzcd}
    \end{equation}
    in $\X$ with all squares biCartesian, $a,a',a'' \in \A$, $b \in \B$, $b' \in
    \B'$, all edges marked ``$*$'' Cartesian, and all edges marked ``$!$''
    coCartesian.
\end{prop}
\begin{proof}
    The statement follows by starting with the $\infty$-category $\Ext(\A,\B)$
    which projects via an equivalence to $\X$ and adding more and more
    universal data by means of (relative) Kan extensions, where we use Lemma
    \ref{lem:sodcartesian} to control how the various universal properties
    relate to containment constraints of fibers and cofibers. We refer to
    \cite{DKS:spherical} for the details.
\end{proof}

\begin{rem}
    \label{rem:grid}
    Note that there are intermediate variants of Propostion \ref{prop:grid},
    identifying the $\infty$-category $\Rec(\A,\B)$ (and hence $\X$) with
    categories of subdiagrams. For instance, the forgetful map 
    \[
        \Rec(\A,\B) \to \left\{ 
    \begin{tikzcd}
        a' \ar{r}{*} \ar{d}&  b \ar{d}{*}\\
        a'' \ar{r}{*} & x'
    \end{tikzcd}
    \right\} 
    \]
    is an equivalence as well.
\end{rem}

As an immediate consequence of Proposition \ref{prop:grid} and Remark
\ref{rem:grid}, we obtain the following central result.

\begin{cor}
    \label{cor:lax}
    Let $(\A,\B)$ be a recollement for $\X$. 
    \begin{enumerate}
        \item\label{it:laxlim} The forgetful maps 
            \[
                \X \overset{\simeq}{\longleftarrow} \left\{ 
    \begin{tikzcd}
        a' \ar{r}{*} \ar{d}&  b \ar{d}{*}\\
        a'' \ar{r}{*} & x'
    \end{tikzcd}
    \right\} \overset{\simeq}{\lra} \left\{ 
    \begin{tikzcd}
        a' \ar{r}{*}\ar{d}&  b\\
        a'' & 
    \end{tikzcd}
     \right\}
     \] 
     are equivalences. In particular, $\X$ is the lax limit of the gluing
     functor $\B \to \A$ determined by the $(\A,\B)$-Cartesian edges.
    \item\label{it:oplaxlim} The forgetful maps 
            \[
                \X \overset{\simeq}{\longleftarrow} \left\{ \begin{tikzcd}
        x \ar{r}{!} \ar{d}{!} & b' \ar{d}{!}\\
        a \ar{r} & a' 
        \end{tikzcd} \right\} \overset{\simeq}{\lra} \left\{ \begin{tikzcd}
            & b' \ar{d}{!}\\
        a \ar{r} & a' 
         \end{tikzcd} \right\}
     \] 
     are equivalences. In particular, $\X$ is the oplax limit of the gluing
     functor $\B \to \A$ determined by the $(\B,\A)$-coCartesian edges.
    \end{enumerate}
\end{cor}

\begin{rem}
    \label{rem:fibration}
    To relate this statement to the diagrams \eqref{eq:lax} and \eqref{eq:oplax},
    we proceed as follows: Denote by $\L$ the full subcategory of $\X \times
    \Delta^2$ spanned by the vertices of $\A \times \{0\}$, $\B \times \{1\}$, and
    $\X \times \{2\}$. Then the projection functor $\pi: \L \to \Delta^2$ is a locally
    Cartesian fibration which straightens to a lax triangle 
    \begin{equation}\label{eq:laxtriangle}
        \Delta^2 \to \Cat_{\infty}
    \end{equation}
    of the form as depicted in \eqref{eq:lax} where the functors are determined by the
    respective Cartesian edges. In particular, the gluing functor $G: \B \to \A$ is
    determined by the $(\A,\B)$-Cartesian edges which correspond to locally
    $\pi$-Cartesian edges in $\L$ over the edge $0 \to 1$ of $\Delta^2$. Corollary
    \ref{cor:lax} \ref{it:laxlim} amounts to the statement that the triangle is a
    lax limit triangle exhibiting an equivalence
    \begin{equation}\label{eq:reclax}
        \X \simeq \{ (a,b,\eta) | a \in \A, b \in \B, \eta: G(b) \to a  \}.
    \end{equation}
    A dual construction, using Corollary \ref{cor:lax} \ref{it:oplaxlim}, leads to an
    oplax limit triangle exhibiting an equivalence
    \begin{equation}\label{eq:opreclax}
        \X \simeq \{ (a,b,\eta) | a \in \A, b \in \B, \eta: a \to F(b)  \}
    \end{equation}
    where $F: \B \to \A$ is the gluing functor determined by the
    $(\B,\A)$-coCartesian edges.
\end{rem}

\bibliographystyle{alpha}
\bibliography{pw}

\end{document}